\documentclass{amsart}
\usepackage{amssymb,stmaryrd}
\usepackage{amsfonts}
\usepackage{amstext}
\usepackage{geometry}                
\usepackage{graphicx}    
\usepackage{amssymb,amsmath,epsfig}
\usepackage{scalefnt}
\usepackage{MnSymbol}
\usepackage[all]{xypic}
\usepackage{slashed}
\usepackage{extarrows}
\usepackage{stmaryrd}
\usepackage{lettrine}
\usepackage{booktabs}
\usepackage{paralist}
\usepackage{subfig}
\usepackage[dvipsnames]{xcolor}
\usepackage[color=ProcessBlue]{todonotes}
\usepackage{hyperref}
\allowdisplaybreaks
\usepackage{tikz-cd}
\usepackage{mwe}
\pagecolor{white}

\newtheorem{lemma}{Lemma}[section]
\newtheorem{proposition}[lemma]{Proposition}
\newtheorem{corollary}[lemma]{Corollary}

\newtheorem{theorem}[lemma]{Theorem}

\newtheorem{definition}[lemma]{Definition}

\renewcommand{\imath}{\mathrm{i}}

\newcommand*{\doublenabla}{%
  \nabla\mkern-12mu\nabla
}   

\newcommand{\CG}{\hbox{{$\mathcal G$}}}

\newcommand{\CJ}{\hbox{{$\mathcal J$}}}

\newcommand{\C}{\mathbb{C}}
\newcommand{\R}{\mathbb{R}}

\newcommand{\Hom}{\mathrm{Hom}}

\newcommand{\del}{\partial}

\newcommand{\extd}{\mathrm{d}}

\newcommand{\isom}{{\cong}}

\newcommand{\tens}{\mathop{{\otimes}}}

\newcommand{\id}{\mathrm{id}}
\newcommand{\im}{\mathrm{im}}
\newcommand{\<}{\langle}
\renewcommand{\>}{\rangle}

\begin{document}

\author{Shahn Majid and Francisco Sim\~{a}o }
\address{School of Mathematical Sciences\\ Queen Mary University of London \\ Mile End Rd, London E1 4NS }
\email{s.majid@qmul.ac.uk, f.castelasimao@qmul.ac.uk}
\thanks{Ver.2}

\title{Quantum Jet Bundles}
	\begin{abstract} We formulate a notion of jet bundles over a possibly noncommutative algebra $A$ equipped with a torsion free connection. Among the conditions needed for 3rd-order jets and above is that the connection also be flat and its `generalised braiding tensor' $\sigma:\Omega^1\tens_A\Omega^1\to \Omega^1\tens_A
	\Omega^1$ obey the Yang-Baxter equation or braid relations. We also cover the case of jet bundles of a given `vector bundle' over $A$ in the form of a bimodule $E$ with a flat bimodule connection with its braiding $\sigma_E$ obeying the coloured braid relations. Examples include the permutation group $S_3$ with its 2-cycles calculus, $M_2(\C)$, the bicrossproduct model quantum spacetime in two dimensions and $\C_q[SL_2]$ for $q$ a 4th root of unity. \end{abstract}

\maketitle
\section{Introduction}
\label{sec:Intro}

Noncommutative geometry is the idea that geometric constructions can be extended to the case where the `coordinate algebra' $A$ is potentially noncommutative. Originally motivated by the quantisation of phase spaces in quantum mechanics, a more recent motivation is the hypothesis that spacetime itself is better modelled by noncommutative coordinates due to quantum gravity effects. The most well-established approach, coming out  of operator algebras, the Gelfand-Naimark theorem and cyclic cohomology is that of A. Connes~\cite{Con}. Since the 1980s there have also emerged more algebraic-geometry like approaches such as \cite{Sta}, as well as a constructive approach motivated by, but not limited to, the differential geometry of quantum groups and models of quantum spacetime as in \cite{BegMa}. In this work we use this last approach, ensuring good contact with examples in the literature. The starting point here is a differential graded algebra (DGA) $(\Omega,\extd)$ with $A$ appearing as degree 0 (in fact we will only need this to the degree 2 or 2-form level) and a bimodule connection $\nabla:\Omega^1\to \Omega^1\tens_A\Omega^1$, $\sigma:\Omega^1\tens_A\Omega^1\to \Omega^1\tens_A\Omega^1$. More generally, one can also consider other vector bundles via their sections as bimodules $E$, and bimodule connections $\nabla_E$ on them. In the bimodule setting, one has both left and right Leibniz rules and referring the latter back to the left requires~\cite{DVM,Mou} the existence of a bimodule `generalised braiding' map $\sigma_E:E\tens_A\Omega^1\to \Omega^1\tens_AE$. Details are recalled in the preliminary Section~\ref{secpre}.

The problem we partially solve using such data is a long-standing one in noncommutative geometry, namely to find a reasonable notion of `jet bundle' over a noncommutative algebra $A$. Jet bundles underly the mathematical formulation of Lagrangian field theory, Noether theorems and the Euler-Lagrange equations in physics, hence quantum jet bundles would be a critical first step down the road of physics on quantum spacetime. The mathematical problem is, moreover, clear enough for first order jet bundles where classical algebraic geometry already suggests that it is natural to take $\CJ^1_A=A\oplus \Omega^1$ as explained in \cite{LakTho}. This previous work, however, used a notion of `balanced derivations' in the noncommutative case to define jet bundles (in fact a `jet algebra') whereas we insist on a standard DGA as our main input data. Our new proposal for jets of all orders is to first define the sub-bimodule 
\[ \Omega_S^k\subseteq \Omega^1\tens_A\cdots\tens_A\Omega^1\]
  ($k$ copies) defined as the joint kernel of the wedge products between adjacent copies. Then $\CJ^k_A\coloneq \oplus_{j=0}^k\Omega_S^j$ with the underlying bimodule structure inherited from $\Omega^1$. This much, i.e. the jet bundle itself, depends only on the DGA. However,  we also need a `jet prolongation' linear map $j^k\colon A\to \CJ^k_A$ which classically expands a function into its Taylor series to degree $k$ and a second product $\bullet_k$ by $A$ (classically taken from the left or the right on sections of the jet bundle, but in our case as a bimodule with $\bullet_k$ on both sides) such that $j^k$ is a bimodule map. We then have a limit $\CJ_A^\infty$ as the limit  of an increasing sequence of $A$-bimodules
\[
\begin{tikzcd}
 & & {\ \atop \displaystyle\cdots}  & & & A  \arrow[dllll,"j^k",swap,bend right=17] \arrow[dlll,"j^{k-1}",swap, bend right=7]     \arrow[dll, draw=none, swap,"\dots"]\arrow[dl,"j^1",swap] \arrow[d,"\id"]\\
 \cdots \arrow[r] & \CJ^k_A \arrow[r] & \CJ^{k-1}_A \arrow[r] & \cdots\arrow[r] & \CJ^1_A\arrow[r] & \CJ^0_A=A
\end{tikzcd}
\]
with the bullet product in each degree and commuting with the  $j^k$ maps. This occupies most of the paper, with Section~\ref{sec:JA123} for $k=1,2,3$ and Section~\ref{sec:JAgeneral} for general $k$. The low degree cases are developed first in explicit detail so that we can see the required quantum geometric data progressively emerging. They also provide the template for the general case with no further conditions on the geometric data needed for $k>3$. By default, $\CJ^k_A$ refers to the $\bullet_k$ structure on both sides, but our construction respects the underlying inherited  bimodule structure by which $\CJ^k_A$ is regarded as a bundle, with the result that one also has a bimodule in the two mixed cases where we use the bullet product from one side and the inherited bundle product from the other. The $j^k$ maps mean that we are constructing {\em split} jet bundles, but these are the relevant ones for physics since, classically, the splitting is needed to define the contact forms  for the variational double complex \cite{And}. Here  $j^k$ splits the surjection $\CJ^k_A\to A$ going along the bottom.
  
 The further geometric data needed beyond a DGA is  as follows. For the bullet products $\bullet_k$ we need a bimodule map $\sigma:\Omega^1\tens_A\Omega^1\to \Omega^1\tens_A\Omega^1$ and for the higher degree theory this should obey the braid relations in the monoidal category of $A$-bimodules, which is an innovative feature of the present work, and also be compatible with the wedge product. Classically, this $\sigma$ would just be the `flip map' and would not appear as additional data. In the split case that we are interested in, i.e. for the jet prolongation maps, we ask that these $\sigma$ arise as part of a torsion-free bimodule connection $\nabla,\sigma$, and for $k\ge 3$ this should moreover be flat. Requiring the existence of a connection with nice properties is not without precedent in the constructive approach to noncommutative geometry, as a proxy for non-existent local coordinates. For example, a bimodule equipped with a flat  bimodule connection can be seen as an algebraic model of  a sheaf \cite[Chap. 4]{BegMa} and similarly a background connection is used in the construction of an algebra of differential operators in \cite[Chap. 6]{BegMa}, which classically would be dual to the jet bundle. Likewise, the Hopf-Galois or `local triviality' property of a quantum principal bundle can be expressed as existence of a connection on the principal bundle. Moreover,  one should {\em not  expect} that the jet bundle {\em and} prolongation maps can all be constructed from the DGA alone. As already clear in Connes approach to noncommutative geometry, the DGA alone is not enough to recover a manifold in the commutative case, one needs additional structure (in Connes case this is provided by a spectral triple or `Dirac operator', which entails a DGA but contains much more information).  Nevertheless, the requirement for $k\ge 3$ of a flat torsion free connection is a significant restriction of the present work, corresponding in the classical case to an affine structure of some kind \cite{Bok}. It also means classically that the vector fields form a Vinberg or pre-Lie algebra.  This assumption constitutes therefore an interesting class of `nice' quantum manifolds which deserve more study and to which the present paper is initially restricted. 
 
 This is not the end of the story and it should be possible to drop the flatness assumption in a future extension of the theory (this is discussed further in the conclusions Section~\ref{seccon}). Moreover, for $\CJ^1_A$, the construction does not depend on $\nabla,\sigma$ up to isomorphism, while for $\CJ^2_A$ it depends up to isomorphism only on $\sigma$ which classically would be the flip map and not additional data. These results are proven in Section~\ref{sec:JA123}. For higher $k$ we are not aware of an isomorphism if we want to remain compatible with the prolongation maps but otherwise, as noted, the construction of $\CJ^k_A$ as a bimodule depends only on $\sigma$ with various compatibility properties rather than on $\nabla$ directly, the latter only being necessary for the jet prolongation maps $j^k$. It is also of interest that the latter  enter through the `braided integer' maps $[n,\sigma]$ and associated braided binomials previously used in the theory of braided-linear spaces (such as the quantum plane) as Hopf algebras in a braided category \cite[Chap.~10]{Ma:book}. In our case, they will be key to a certain higher-order derivation property in Lemma~\ref{lem:nabla3leib} and its generalisation Lemma~\ref{lem:leibrule}. Braided techniques also provide a certain `braided shuffle product' on the space of symmetric cotensors making this into an algebra $\Omega_S$ used in the jet bundle construction. In some cases, we also have a  `reduced jet bundle' similarly built on the braided-symmetric algebra $S_\sigma(\Omega^1)$ as explained in Section~\ref{secSsigma}.

 Although not the main part of the work, we also cover the case of jet bundles $\CJ^k_E$ relevant to a vector bundle (viewed as a bimodule of sections) over $A$, in Section~\ref{secJE}. Here the order 1 case is a noncommutative version of the Atiyah exact sequence \cite{Ati}, with
\[ 0\to \Omega^1\tens_A E\to \CJ^1_E\to E\to 0\]
with bullet actions built from a given bimodule map $\sigma_E$. As before, this  is not additional data in the classical case as it would just be the flip map. We will see that splittings $j^1_E\colon E\to \CJ^1_E$ of this sequence are then in 1-1 correspondence with bimodule connections $\nabla_E$ with generalised braiding the given $\sigma_E$.  Moreover, Lemma~\ref{lemJ1Eisom} shows that different choices of $\nabla_E,\sigma_E$ then give isomorphic split Atiyah sequences. For higher jets, we only consider split jet bundles where the jet prolongation map $j^k_E\colon E\to \CJ^k_E$ is part of the construction  and depends on a choice of $\nabla_E$ similarly to the above.

Section~\ref{secex} addresses the important task of showing that our construction is compatible with several standard examples of quantum differential geometries as in \cite{BegMa}. These include $M_2(\C)$ the algebra of $2\times 2$ matrices regarded as a `noncommutative coordinate algebra' with its standard 2-dimensional $\Omega^1$, the commutative algebra $\C(S_3)$ of functions on the permutation group of 3 elements with its standard 3-dimensional  $\Omega^1$ (here 1-forms do not commute with functions so we are still doing noncommutative geometry), the 2-dimensional bicrossproduct model Minkowski spacetime with relations $[r,t]=\lambda r$ and its standard 2-dimensional $\Omega^1$ and the quantum group $\C_q[SL_2]$ at $q$ a fourth root of unity with its standard 3-dimensional $\Omega^1$. A striking observation in most of these examples is that by the time we have imposed that $\nabla$ is torsion free, flat and obeys the braid relations for its $\sigma$, the other conditions needed for our theory -- extendability as in \cite{BegMa:bia}, $\wedge$-compatibility \cite[Chap.~8.1]{BegMa}, and the new notion introduced now of `Leibniz compatibility' in Lemma~\ref{lem:nabla3leib}, all hold automatically.

Before turning to the noncommutative constructions, we briefly remind the reader of the key ideas, for orientation purposes only, in the classical theory of jet bundles.  There are many works with more details and we refer, for example, to \cite{ForRom, MusHro}, where the former also contains their use in Lagrangian field theory. If $E\to M$ is a vector bundle (for this classical discussion we denote its sections by $\Gamma(E)$, not simply by $E$ as above), the jet bundle of degree up to order $k$ is a certain affine bundle $J_E^k\to E$, which we also regard as a bundle over $M$ via the projection $E\to M$, together with a map $j^k_E\colon \Gamma(E)\to \Gamma(J^k_E)$  that, roughly speaking, provides the Taylor series expansion of $s\in \Gamma(E)$ at each point of $M$. In physics the sections $\Gamma(E)$ would typically be  `matter fields' of which the simplest `scalar field' case is provided by taking $E$ trivial with fibre $\mathbb{R}$ or $\mathbb{C}$, so that $\Gamma(E)=C^\infty(M)$ is the space of (real or complex valued) functions on the manifold.  Less well-known, but clear by the end of our algebraic version (as discussed in Section~\ref{seccon}) is that in the scalar field case $\Gamma(J^k)$ is just the space of symmetric tensors of degree $(0,i)$ for $i\le k$, i.e symmetric tensor powers of the bundle of 1-forms \cite{Saunders}. This means that we can loosely think of the colimit of the $\Gamma(J^k)$ as something like the algebra $C^\infty_{poly}(TM)$ of functions on the tangent bundle polynomial in the fibre direction, but a lot bigger in allowing powerseries. Now, given $s\in C^\infty(M)$, the jet prolongation map to a function on $TM$ is via the Taylor series at each point of $M$,
\[ s\mapsto (s, \del_i s,\del_i\del_j s,\cdots),\]
where $\del_i$ are partial derivatives in a local coordinate chart dual to local 1-forms $\extd x^i$. Let $v^i$ correspondingly parametrize the tangent space $T_xM$ (e.g. a vector field would have the form $v^i(x)\del_i$, but in our case we consider only one point $x\in M$). Then the Taylor expansion provides functions $v^i\del_is, v^iv^j\del_i\del_j s$ etc evaluated at $x$, as functions of increasing degree on $T_xM$. It is this construction which in an algebraic form we extend to noncommutative geometry in Sections \ref{sec:JA123}, \ref{sec:JAgeneral}.

\subsection*{Acknowledgements}

We thank B. Noohi and the authors of \cite{FMW} for helpful comments on our first arXiv preprint version in relation to \cite{Ati} and to connection-independence. The work \cite{FMW} appeared some months later and provides an interesting more general categorical construction as `jet endofunctors'. 

\section{Preliminaries}\label{secpre}

The general results in the paper work over an field $k$. All algebras are assumed unital. Here we recall the elements we need from noncommutative geometry in the constructive bimodule-based approach. More details are in \cite{BegMa} and in the wider literature cited therein.

By a first order differential algebra, we mean $A$ equipped with an $A$-bimodule $\Omega^1$ (so the left and right actions commute) together with a bimodule derivation $\extd\colon A\to \Omega^1$, meaning that it satisfies the Leibniz rule $\extd(ab)=a\extd b+ (\extd a)b$ for all $a,b\in A$. We also ask that $\Omega^1$ is spanned by elements of the form $a\extd b$ for $a,b\in A$. We generally ask this to extend to a differential graded algebra $(\Omega,\extd)$, where $\Omega$ is a graded algebra and $\extd$ increases degree by 1 and obeys a graded Leibniz rule. Such an extension always exists but the universal such (called the `maximal prolongation') is usually much bigger than desired to match the classical limit and one may specify a quotient of it by further relations. We also require surjectivity so that $\Omega$ is generated by $A$, $\Omega^1$ or equivalently by $A$, $\extd A$. Its product among elements of degree bigger than 0 is denoted by $\wedge$.

If $E$ is a left $A$-module, a left connection is defined as $\nabla_E:E\to \Omega^1\tens_A E$ obeying $\nabla_E(ae)=\extd a\tens e+a\nabla_E e$ for all $e\in E$ and $a\in A$. Here the left output can be evaluated against a `vector field' right module map $X:\Omega^1\to A$ to define a `covariant derivative' along $X$, but we do not need to do that here. When $E$ is a bimodule we want a compatible Leibniz rule for right products and the framework we adopt is that of a left {\em bimodule connection} \cite{DVM,Mou}
\[ \nabla_E(ea)=(\nabla_Ee)a+ \sigma_E(e\tens\extd a)\]
for some bimodule map $\sigma_E:E\tens_A\Omega^1\to \Omega^1\tens_A E$ called the `generalised braiding'. The latter is needed to bring $\extd a\in \Omega^1$ to the left to be consistent with the other terms. This map is not, however, additional data since, if it exists, it is uniquely determined by the above; being a bimodule connection is a property that a left connection on a bimodule can have. The collection of bimodules equipped with bimodule connections has a monoidal category structure, where the tensor product of $(E,\nabla_E)$ and $(F,\nabla_F)$ is built on $E\tens_A F$ with
\[ \nabla_{E\tens F}=\nabla_E\tens\id+(\sigma_E\tens\id)(\id\tens\nabla_F)\]
with $\sigma_E$ used to bring the $\Omega^1$ part of $\nabla_F$ to the left to be consistent with other terms. The corresponding generalised braiding is $\sigma_{E\tens F} = (\sigma_E \tens \id)(\id \tens \sigma_F)$ \cite{BegMa}. We do not usually write the $\tens_A$ with the subscript in the context of elements or morphisms (to avoid notational clutter), but it should be understood. A particularly nice additional property that one can require of a bimodule connection is the notion of extendability of $\sigma_E$ to a map $E\tens_A\Omega^2\to\Omega^2\tens_AE$ in the obvious way \cite[Chap 4]{BegMa}.

The curvature of any left connection is defined as
\[ R_E:E\to \Omega^2\tens_A E,\quad R_E=(\extd\tens \id- \id\wedge\nabla_E)\nabla_E\]
and is a left-module map. When $E=\Omega^1$ itself, we simply denote the bimodule connection $\nabla$ and its `generalised braiding' $\sigma$. In this case we say that $\nabla$ is torsion-free if the torsion tensor
\[ T_\nabla:\Omega^1\to \Omega^2,\quad T_\nabla=\wedge\nabla-\extd\]
vanishes, in which case it can be shown that $\wedge\circ(\id+\sigma)=0$. We will not need a quantum metric in the present paper, but for context this is defined as $g\in \Omega^1\tens_A\Omega^1$ which is nondegenerate in a suitable sense. The strongest sense is the existence of a bimodule inverse $(\ ,\ ):\Omega^1\tens_A\Omega^1\to A$ making $\Omega^1$ its own left and right dual in the monoidal category of $A$-bimodules. One usually asks for $g$ to be `quantum symmetric' in the sense $\wedge g=0$ and this is one of the motivations behind our formulation of a space of `quantum symmetric $k$-forms' $\Omega^k_S$ in Sections~\ref{sec:JA123} and \ref{sec:JAgeneral}. A quantum Levi-Civita connection (QLC) for a given quantum metric is $\nabla$ which is torsion free and obeys $\nabla g=0$ using here the tensor product connection \cite{BegMa:gra,BegMa}. However, we do not assume a metric in general and the connection $\nabla$ has a different and more auxiliary role
as discussed in the Introduction.

Another ingredient we will need is the notion of a `braided integer'. These were introduced as part of the construction of certain Hopf algebras in braided categories \cite{Ma:book,Ma:hod}. In a braided category it is useful to denote morphisms as strings flowing down the page, tensor product of the category by omission \cite{Ma:alg}, and the braiding natural transformation $\psi=\includegraphics{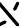}$ between two objects by a braid crossing. Then, if $V$ is an object of a braided category with braiding $\psi:V^{\tens 2}\to V^{\tens 2}$, we define morphisms $[n,\psi]:V^n\to V^n$ by
\[ \includegraphics[scale=0.6]{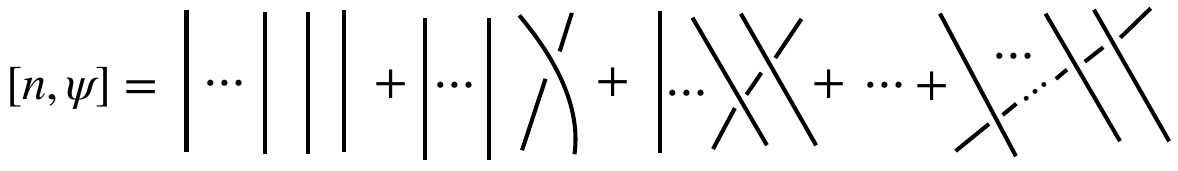}\]
as a generalisation of $q$-integers $[n]_q=1+q+\cdots+q^{n-1}$. We similarly define braided binomial morphisms $\left[{n\atop m},\psi\right]$ iteratively and recovering $[n,\psi]$ when $m=1$. When braided integers etc of different sizes are composed, we fill in with identity strands in a manner specified (usually from one side or the other).

In what follows, we will not employ braided category theory in a formal way, hence we have kept the above discussion minimal, but we will make use of string diagrams notations more generally. They should be regarded as no more than a visualisation device to denote the composition of maps, but generally strands in our case will begin and end on $A$-bimodules and juxtaposition denotes $\tens_A$. For example, the $\wedge$ product $\Omega^1\tens_A\Omega^1\to \Omega^2$ is denoted \includegraphics{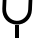} and is a morphism in the category of $A$-bimodules. The $\sigma$ and $\sigma_E$ maps are likewise morphisms (bimodule maps) in this category and these particular ones will be denoted as braid crossings \includegraphics{braid.pdf} as above. They do not necessarily obey braid relations but they can be composed and diagrams such as $[n,\sigma]$ are again bimodule maps. If $\sigma$ obeys the braid relations with tensor products over $A$ then it makes the subcategory generated by sums and tensor products of $\Omega^1$ onto a braided subcategory of the monoidal category of $A$-bimodules, but {\em a priori}, we do not assume this. It will also be convenient to denote a connection $\nabla$ or $\nabla_E$ the other way as a splitting \includegraphics{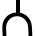}, but note that this is {\em not} a bimodule or even a left module map. Nor is $\extd: \Omega^1\to \Omega^2$. The use of diagrams here is still possible as in \cite{BegMa:bia,BegMa}, but one has to check that while each diagram in a sum of terms may not by itself be well-defined, the condition expressed when all terms are put on one side defines a well defined map on $\tens_A$ between the strands and typically at least a left-module map. The diagrammatic representation of zero torsion and curvature in Fig.~\ref{fig1}(a),(b) respectively, are examples.

\section{Jet bundles $\CJ^k_A$ over a noncommutative algebra to order 3}\label{sec:JA123}

In this section we construct jet bundles or more precisely their sections as `jet bimodule' $\CJ^k_A$ for $k\le 3$  over a potentially noncommutative algebra $A$ equipped with a differential graded algebra $(\Omega,\extd)$. These low order cases allow us to see explicitly how the progressively stronger geometric data emerge and also provide a template for the general case in Section~\ref{sec:JAgeneral}.   For $\CJ^2_A$, we will need to equip $\Omega^1$ with a torsion free connection $\nabla$ while for $\CJ^3_A$ and above, this will need to be flat and obey further conditions. In Section~\ref{secJE}, we will extend this construction further to jet bimodules over a given bimodule $E$ equipped with a flat bimodule connection $\nabla_E$. In physical terms, we restrict in the present section to `scalar fields' where $E=A$ and $\nabla_E=\extd$.

\subsection{First and second order jet bundles over a noncommutative algebra}
\label{secJA12}

We define the 1st-order jet bimodule and jet prolongation map as
\begin{align*}
    &\CJ^1_A = A \oplus \Omega^1,&
    &j^1(s) = s + \extd s,
\end{align*}
where $s\in A$ is prolonged to $j^1(s)\in\CJ^1_A$. Here, the first order derivatives of $s$ are encoded in $j^1(s)$ though $\extd s$. The left and right bullet actions of $a\in A$ on $\CJ^1_A$ are defined on each degree as
\begin{align*}
    &a\bullet_1 s = (a + \extd a) s,&
    &s\bullet_1a = s (a + \extd a),\\
    &a\bullet_1 \omega  = a \omega,&
    &\omega \bullet_1 a = \omega a,
\end{align*}
where $\omega\in\Omega^1$. Here and in general, we use $\bullet_k$ to denote the bimodule actions of $A$ on $\CJ^k_A$. It follows immediately from the Leibniz rule  that $\CJ^1_A$ becomes a bimodule and $j^1$ a bimodule map where $A$ is a bimodule by left and right action on itself. There is also clearly a bimodule map surjection $\pi_1:\CJ^1_A\to A$ given by projecting out the $\Omega^1$ component and such that $\pi_1 \circ j^1=\id$.

For $k=2$, we suppose that $\nabla\colon \Omega^1 \rightarrow \Omega^1 \tens_A \Omega^1$ is a torsion free bimodule connection on $\Omega^1$ with generalised braiding $\sigma\colon \Omega^1 \tens_A \Omega^1 \rightarrow \Omega^1 \tens_A \Omega^1$ and define the sub-bimodule of \emph{quantum symmetric 2-forms} as
\begin{align*}
    \Omega^2_S\coloneq \ker\wedge\subseteq \Omega^1\tens_A\Omega^1
\end{align*}
which classically corresponds to the symmetric rank $(0,2)$-tensors. We set
\begin{align*}
    &\CJ^2_A = A \oplus \Omega^1  \oplus \Omega^2_S,&
    &j^2(s) = j^1(s)+ \nabla \extd s.
\end{align*}
The usual formula for $j^2(s)$ in terms of symmetric 2nd partial derivatives appears in a central basis, see (\ref{jbasis}) later.  We define the left and right bullet actions of $a\in A$ on $\CJ^2_A$ as
\begin{align*}
    &a\bullet_2 s = a \bullet_1 s + (\nabla \extd a)s,&
    &s\bullet_1a = s \bullet_1 a + s \nabla \extd a,\\
    &a\bullet_2 \omega  = a \bullet_1 \omega + [2,\sigma] (\extd a \tens \omega),&
    &\omega \bullet_2 a = \omega \bullet_1 a + [2,\sigma] (\omega \tens \extd a),\\
    &a \bullet_2 \omega^1 \tens \omega^2 = a \omega^1 \tens \omega^2,&
    &\omega^1 \tens \omega^2 \bullet_2 a = \omega^1 \tens \omega^2 a
\end{align*}
for $s\in A$, $\omega \in \Omega^1$ and $\omega^1 \tens \omega^2 \in \Omega^2_S$. Here  $[2,\sigma] :=\id+\sigma$ is a `braided integer' bimodule map on $\Omega^{1\tens_A 2}$. Note also that $\bullet_2$ is the recursively defined via $\bullet_1$.

\begin{proposition}\label{propJ12} Let $\nabla$ be a torsion free bimodule connection on $\Omega^1$. Then $\CJ_A^2$ is an $A$-bimodule and $j^2:A\to \CJ^2_A$ is a bimodule map. Moreover,  $\pi_2:\CJ_A^2\to \CJ_A^1$ given by quotienting out $\Omega^2_S$ is a bimodule surjection with $\pi_2\circ j^2=j^1$. \end{proposition}
\begin{proof}
 Since we assume $\nabla$ to be torsion free, we have ${\rm im}(\id + \sigma) = {\rm im} [2,\sigma] \subseteq \ker \wedge = \Omega^2_S$, ensuring that the bullet action of $A$ on $\CJ^2_A$ indeed lands on $\CJ^2_A$. Before addressing the bimodule action axioms for $k=2$, it is helpful to note the 2nd-order Leibniz rule for $\nabla \extd$
\begin{align}
    \nabla \extd(ab) &= \nabla ((\extd a) b + a \extd b)
    = (\nabla \extd a) b + \sigma(\extd a \tens \extd b) + \extd a \tens \extd b + a\nabla \extd b\nonumber\\
    &= (\nabla \extd a) b + [2,\sigma] (\extd a \tens \extd b) + a\nabla \extd b.
    \label{eq:2ndOrderLeibniz}
\end{align}
It is also useful to consider these calculations in an inductive manner, using the results for $\bullet_1$ when computing the properties of $\bullet_2$. We again compute these degree by degree
\begin{align*}
    a \bullet_2 &(b \bullet_2 s)
    = a \bullet_2 (b \bullet_1 s)
    + a (\nabla \extd b)  s
    = a \bullet_1 (b \bullet_1 s)
    + (\nabla\extd a) bs + [2,\sigma] (\extd a \tens (\extd b) s) + a (\nabla \extd b)  s\\
    &=(ab) \bullet_1 s +( \nabla\extd(ab)) s
    = (ab)\bullet_2 s.
\end{align*}
In the second equality we wrote $b \bullet_1 s = bs + (\extd b) s$ and found extra terms that this picks up when acted upon by $a\bullet_2$. We then used (\ref{eq:2ndOrderLeibniz}) and that the generalised braiding $\sigma$ is a bimodule map to write $\sigma(\extd a \tens (\extd b) s) = \sigma(\extd a \tens \extd b) s$ to recognise the answer. The proof of the right action is strictly similar and omitted, while for a bimodule
\begin{align*}
  a \bullet_2 (s \bullet_2 b)
    &=a \bullet_2 (s \bullet_1 b)
    + a s \nabla\extd b
    = a \bullet_1 (s \bullet_1 b)
    + (\nabla \extd a) sb + (\id+\sigma) \extd a \tens s\extd b
    + a s \nabla\extd b\\
    &= (a \bullet_1 s) \bullet_2 b
    + ((\nabla \extd a)s) \bullet_2 b
    = (a \bullet_2 s) \bullet_2 b.
\end{align*}
Next, the additional term in $j^2$  lands in $\CJ^2_A = A\oplus \Omega^1 \oplus \Omega^2_S$ as $\nabla \extd s$ is quantum symmetric due to torsion freeness of the connection  so that $\wedge \nabla \extd s = \extd^2 s = 0$. The computations to show that $j^2$ is a bimodule map can again be done in an inductive manner using (\ref{eq:2ndOrderLeibniz}) as
\begin{align*}
    j^2(as)
    &= j^1(as) + \nabla\extd(as)
    = a \bullet_1 j^1(s) + (\nabla\extd a) s + [2,\sigma](\extd a \tens \extd s) + a\nabla\extd s\\
    &= a \bullet_2 j^1(s) + a\bullet_2\nabla\extd s
    = a\bullet_2 j^2(s),\\
    j^2(sa)
    &= j^1(sa) + \nabla(sa)
    = j^1(s) \bullet_1 a
    + (\nabla\extd s) a + [2,\sigma](\extd s \tens \extd a) + s\nabla\extd a\\
    &= j^1(s) \bullet_2 a + (\nabla\extd s) \bullet_2 a
    = j^2(s) \bullet_2 a.
\end{align*}
The projection and its stated properties are clear. It also follows that $j^2$ is split by $\pi_1\circ \pi_2\colon \CJ^2_A\to A$ given that $j^1$ was split by $\pi_1$.
\end{proof}

Finally, it is natural to ask to what extent the 2nd-order jet bimodule $\CJ^2_A$ depend on choice of the connection $(\nabla,\sigma)$. Following the classical notion of jet equivalence, consider two jet bimodules $\CJ^2_A,\tilde{\CJ}^2_A$ with respective jet prologantion maps $j^2\colon A \to \CJ^2_A$ and $\tilde j^2\colon A \to \tilde{\CJ}^2_A$ and bullet actions $\bullet_2$, $\tilde \bullet_2$. We then say that $\CJ^2_A, \tilde{\CJ}^2_A$ are equivalent as jet bimodules if there is a isomorphism $\varphi \in {}_A\Hom_A(\CJ^2_A,\tilde{\CJ}^2_A)$ w.r.t. the bullet actions such that the following diagram commutes
\[
\begin{tikzcd}
\CJ^2_A \arrow{rr}{\varphi} && \tilde{\CJ}^2_A\\
& \arrow{ul}{j^2}  A \arrow{ur}[swap]{\tilde j^2} 
\end{tikzcd}.
\]

\begin{lemma}
\label{lemma:equivalenceJ2A}
Consider $\CJ^2_A, \tilde{\CJ}^2_A$ as constructed from the connections $(\nabla,\sigma)$, $(\tilde \nabla, \sigma)$ on $\Omega^1$ respectively. Then $\CJ^2_A \simeq \tilde{\CJ}^2_A$ as  $A$-bimodules w.r.t. the bullet actions.
\end{lemma}

\begin{proof}
Consider the map $\varphi\colon \CJ^2_A \to \tilde \CJ^2_A$ defined on $a + \omega_1 + \omega_2 \in \CJ^2_A$ as 
\[
\varphi(a + \omega_1 + \omega_2) = a + \omega_1 + (\tilde \nabla - \nabla) \omega_1+ \omega_2.
\]
To see that $\varphi$ fulfils $\varphi \circ j^2 = \tilde j^2$ we compute
\[
\varphi( j^2(a) ) = \varphi(a+\extd a+ \nabla \extd a) 
= a+\extd a + (\tilde \nabla - \nabla) \extd a + \nabla \extd a = \tilde j^2(a).
\]
To check that $\varphi$ is indeed a bimodule map note that
\begin{align*}
&\varphi(b \bullet_2 (a + \omega_1 + \omega_2) ) 
= \varphi(ba + (\extd b) a + (\nabla \extd b )a + b\omega_1 + [2,\sigma] \extd b \tens \omega_1 + b \omega_2)\\
&= ba + (\extd b) a + (\tilde \nabla - \nabla)( (\extd b) a) + (\nabla \extd b )a + b\omega_1 + (\tilde \nabla - \nabla ) (b\omega_1)+ [2,\sigma] \extd b \tens \omega_1 + b \omega_2\\
&= ba + (\extd b) a + (\tilde \nabla \extd b)a + b\omega_1 + b\tilde \nabla \omega_1 - b \nabla \omega_1+ [2,\sigma] \extd b \tens \omega_1 + b \omega_2\\
&= b \tilde \bullet_2 a + b \tilde \bullet_2 \omega_1 + b \tilde \bullet_2 (\tilde \nabla \omega_1 - \nabla \omega_1 + \omega_2) = b \tilde \bullet_2 \varphi(a + \omega_1 + \omega_2) 
\end{align*}
and
\begin{align*}
&\varphi((a+\omega_1 + \omega_2) \bullet_2 b)
= \varphi(ab + a \extd b + a\nabla \extd b + \omega_1b +[2,\sigma] (\omega_1 \tens \extd b) + \omega_2 b)\\
&= ab + a\extd b +(\tilde \nabla - \nabla)(a\extd b) +a\nabla\extd b + \omega_1 b + (\tilde \nabla - \nabla)(\omega_1 b) + [2,\sigma] (\omega_1 \tens \extd b) + \omega_2 b\\
&= ab + a\extd b +a (\tilde \nabla - \nabla)\extd b + a\nabla\extd b + \omega_1 b + (\tilde \nabla - \nabla)(\omega_1) b + [2,\sigma] (\omega_1 \tens \extd b) + \omega_2 b\\
&= a \tilde \bullet_2 b + \omega_1 \tilde \bullet_2 b + ( (\tilde \nabla - \nabla)(\omega_1) + \omega_2) \tilde \bullet_2 b
= \varphi(a + \omega_1 + \omega_2) \tilde \bullet_2 b
\end{align*}
\end{proof}

Therefore our 2nd-order jet construction only depends on the choice of braiding $\sigma$. It is interesting to note that, if we consider $\CJ^2_A, \tilde{\CJ}^2_A$ with the right bullet action and the inherited action from $\Omega^1$ on the left, then they are isomorphic via the same $\varphi$ as above, even if one takes two connections $(\nabla,\sigma)$, $(\tilde\nabla,\tilde\sigma)$ with different braidings. This is due to the identity
\[
(\tilde \nabla - \nabla)(\omega_1 b) + [2,\sigma] (\omega_1 \tens \extd b) = (\tilde \nabla - \nabla)(\omega_1) b + [2,\tilde \sigma] (\omega_1 \tens \extd b)
\]
which can be used in the second line of the last computation above to show that $\varphi$ is still a right-module map in this case. One can also check that $\varphi$ is a left module map in this case.

\subsection{Third order jets bundle over a noncommutative algebra}
\label{sec:JA3}
We now construct $\CJ^3_A$ which is considerably more involved than $\CJ^{1}_A,\CJ^2_A$, and will lay the groundwork for the general case. Here we have to make further assumptions, which we will illustrate and motivate during this subsection. These are presented in Fig.~\ref{fig1} using the diagrammatic notion where maps are read downwards, as discussed at the end of Section~\ref{secpre}. Recall that we denoted the wedge product by a join,  $\nabla$ by a splitting and  its associated $\sigma$ by a braid crossing.

 From now on, we will use the notation $\wedge_i$,$\sigma_i$ for the wedge product and braiding respectively applied to the $i,i+1$ tensor factors of $\Omega^1 \tens_A \Omega^1 \tens_A \Omega^1$. We define sub-module of \emph{quantum symmetric 3-forms} as $\Omega^3_S=\ker \wedge_1 \cap \ker \wedge_2$, the joint kernel of $\Omega^1 \tens_A \Omega^1 \tens_A \Omega^1$ by applying $\wedge$ in the adjacent places. We continue with a torsion free $\nabla\colon \Omega^1\to \Omega^1\tens_A\Omega^1$ as needed for $\CJ^2_A$ and now denote by
\begin{align*}
    &\nabla_2\colon \Omega^1 \tens_A \Omega^1 \rightarrow \Omega^1 \tens_A \Omega^1 \tens_A \Omega^1,&
    &\nabla_2 = (\nabla \tens \id) + \sigma_1 (\id \tens \nabla)
\end{align*}
the tensor product connection on $\Omega^1 \tens_A \Omega^1$, with braiding
\begin{align*}
    &\sigma_{\nabla_2}\colon \Omega^1 \tens_A \Omega^1 \tens_A \Omega^1 \rightarrow \Omega^1 \tens_A \Omega^1 \tens_A \Omega^1,&
    &\sigma_{\nabla_2} = \sigma_1 \sigma_2.
\end{align*}
The 3rd-order jet bimodule and jet prolongation bimodule map will be defined as
\begin{align*}
    &\CJ^3_A = A \oplus \Omega^1 \oplus \Omega^2_S \oplus \Omega^3_S,&
    &j^3(s) = s + \extd s + \nabla \extd s + \nabla_2 \nabla \extd s
\end{align*}
for certain bimodule bullet actions $\bullet_3$ to be determined.

We first analyse the conditions for the proposed $j^3$ to land in the right place, namely that $\nabla_2\nabla \extd s \in \Omega^3_S$. For this we need flatness and the $\wedge$-compatibility condition as shown in Fig.~\ref{fig1}(b) and~\ref{fig1}(c) respectively.
\begin{figure}
 \[ \includegraphics[scale=0.75]{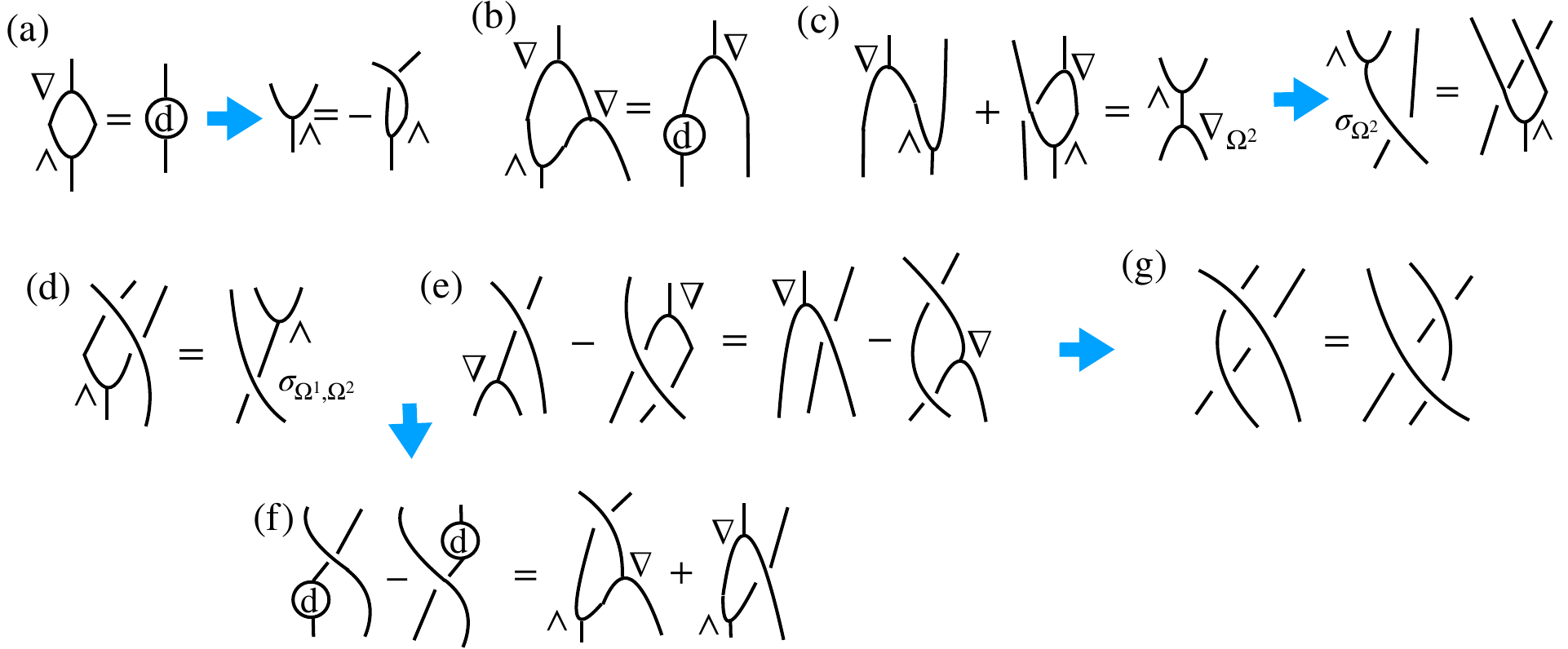}\]
 \caption{\label{fig1} Conditions for the connection $\nabla$: (a) torsion-free and its consequence, (b) flat, (c) $\wedge$-compatibility and its consequence, (d) extendability (e) Leibniz compatibility $\nabla_2\sigma_1=\sigma_2\nabla_2$ in Lemma~\ref{lem:nabla3leib}, which with (a), (d) implies condition (f) in Lemma~\ref{lemR} for the curvature to be a bimodule map. It also implies the Yang-Baxter or braid relations  (g).}
 \end{figure}

\begin{lemma}\label{nablaOmegaS} Suppose that $\nabla$ is torsion free, flat and {\em $\wedge$-compatible} in the sense that
\[ \nabla_{\Omega^2}(\eta\wedge\omega):=(\nabla\eta)\wedge\omega+ \wedge_2 \sigma_1(\eta\tens \nabla\omega)\]
for all $\eta,\omega\in \Omega^1$ is a well-defined bimodule connection on $\Omega^2$. Then $\nabla_2$ restricts to a bimodule connection on $\Omega^2_S$ and $\nabla_2\nabla\extd s\in \Omega^3_S$ for all $s\in A$. \end{lemma}
\begin{proof} For wedge acting on the first factors, we have
\begin{align*}
    \wedge_1 \nabla_2 \nabla \omega
    &= \wedge_1(\nabla \tens \id + \sigma_1(\id \tens \nabla))\nabla \omega= (\wedge \nabla \tens \id + \wedge_1 \sigma_1(\id \tens \nabla))\nabla \omega\\
    &= (\extd\tens \id - \wedge_1 (\id \tens \nabla))\nabla \omega= (\extd\tens \id - (\id \wedge \nabla))\nabla \omega= R_\nabla(\omega)=0,
\end{align*}
where we used $\wedge \nabla = \extd$, $\im(\id + \sigma) \subseteq \ker \wedge$ and that the curvature $R_\nabla$ of $\nabla$ vanishes. We actually only needed $R_\nabla\extd s=0$, but since $R_\nabla$ is a left module map, this is equivalent. For $\wedge$ to vanish in the second position, we have
\begin{align*}
    \wedge_2 \nabla_2 \nabla \omega
    &= \wedge_2(\nabla \tens \id + \sigma_1(\id \tens \nabla))\nabla \omega=\nabla_{\Omega^2}(\wedge\nabla\omega)=\nabla_{\Omega^2}\extd \omega,
\end{align*}
since $\nabla$ is assumed to be $\wedge$-compatible and torsion free. This vanishes if $\omega=\extd s$. In fact, we have proven the stronger property
\begin{equation}\label{wedgecompatcon} \wedge_2\nabla_2(\Omega^2_S)=0\end{equation}
so that $\nabla_2$ restricts to $\nabla_{\Omega^2_S}\colon \Omega^2_S\to \Omega^1\tens_A\Omega^2_S$ as also required.
It is shown in \cite[p. 575]{BegMa} that (\ref{wedgecompatcon}) holds iff $\nabla$ is $\wedge$-compatible, i.e. the converse is also true.
\end{proof}

The $\wedge$-compatibility condition used in Lemma~\ref{nablaOmegaS} also entails the existence of a braiding for $\nabla_{\Omega^2}$, which is necessarily given by
 \[ \sigma_{\Omega^2}: \Omega^2\tens_A
 \Omega^1 \to \Omega^1\tens_A\Omega^2,\quad \sigma_{\Omega^2}(\omega\wedge\eta\tens\zeta)=\wedge_2\sigma_1\sigma_2(\omega\tens\eta\tens\zeta)\]
as  shown  in Fig.~1(c).

Next, we will need the notion that $\nabla$ is {\em extendable} \cite[Def. 4.10]{BegMa} in the sense that $\sigma$ extends to a well defined bimodule map
 \[ \sigma_{\Omega^1,\Omega^2}:\Omega^1\tens_A
 \Omega^2 \to \Omega^2\tens_A\Omega^1,\quad  \sigma_{\Omega^1,\Omega^2}(\omega\tens\eta\wedge\zeta)=\wedge_1\sigma_2\sigma_1(\omega\tens\eta\tens\zeta)\]
 for all $\omega,\eta,\zeta\in\Omega^1$. This is shown in Fig.~\ref{fig1}(d) and restores a certain symmetry to our assumptions. It is also motivated geometrically as part of the conditions for an object in the monoidal category ${}_A\CG_A$ in \cite{BegMa}, which consists of bimodules with extendable connections of which the curvatures are bimodule maps.

 \begin{lemma}\label{lemR} An extendable connection $\nabla$ on $\Omega^1$ has curvature a bimodule map iff the condition in Fig.~\ref{fig1}(f) holds.  In this case the curvature of $\nabla_2$ is
 \[ R_{\nabla_2}=R_\nabla\tens\id+( \sigma_{\Omega^1,\Omega^2}\tens\id)(\id\tens R_\nabla).\]
 \end{lemma}
 \begin{proof} For the first part, the operator $(\extd\tens \id-\id\wedge\nabla)$ is a left module map as $\nabla$ is a left connection, which in turn makes $R_\nabla$ a left module map as is well-known, using $\extd^2=0$. On the other side
 \begin{align*}R_\nabla(\omega a)&=(\extd\tens\id-\wedge_1\nabla)((\nabla\omega)a+\sigma(\omega\tens\extd a))\\
 &=R_\nabla(\omega)a-\wedge_1(\id\tens\sigma)(\nabla\omega\tens \extd a)+(\extd\tens_A\id-\wedge_1\nabla)\sigma(\omega\tens\extd a)
 \end{align*}
 for all $\omega\in\Omega^1$ and $a\in A$. So we have a bimodule map iff Fig.~\ref{fig1}(f) holds when applied to $\omega\tens_A\extd a$, since $\extd^2a=0$. Writing $C(\omega,\eta)$ for the left minus the right side of (f), it is easy to see from $\nabla$ a bimodule connection, the graded Leibniz rule for $\extd$, and extendability in Fig.~\ref{fig1}(d) that $C(\omega,a\eta)=C(\omega a,\eta)$ for all $a\in A$. Hence if $C(\omega,\extd b)=0$ for all $b,\omega$ then $C(\omega,a\extd b)=0$ and hence $C(\omega,\eta)=0$ for all $\omega,\eta\in\Omega^1$. One can similarly prove that $C(\omega,\eta a)=C(\omega, \eta)a$ and more easily that $C(a\omega,\eta)=aC(\omega,\eta)$. Hence the condition actually makes sense as the vanishing of a certain bimodule map from $\Omega^1\tens_A\Omega^1$ to itself.  The last part is a special case of the proof in \cite{BegMa:bia,BegMa} that ${}_A\CG_A$ has tensor products, so we omit details. \end{proof}

 We are mainly interested in the case of zero curvature $R_\nabla=0$ in view of Lemma~\ref{nablaOmegaS}, and zero is a bimodule map so the above will also ensure that $\nabla_2$ is flat. It also means that the stated diagram (f) must hold in the flat extendable case. It will be useful, however, whether or not the connection is flat, to impose a novel condition, shown in Fig.~\ref{fig1}(e), which in the torsion free extendable case implies (f). To further understand its significance, we note that the `braided integer' and `co-braided integer' maps on $(\Omega^{1})^{\tens_A 3}$ restrict to
\begin{align}
\label{3sig}
 [3,\sigma]:= \id+\sigma_2+\sigma_1\sigma_2:\Omega^2_S\tens_A\Omega^1\to \Omega^3_S,\quad
[3,\sigma]':=\id+\sigma_1+\sigma_2\sigma_1: \Omega^1\tens_A\Omega^2_S\to\Omega^3_S
\end{align}
where the indices on $\sigma$ refer to the position. That $[3,\sigma],[3,\sigma]'$ land in the right place as shown assumes that the $\wedge$-compatibility and extendability in Fig.~\ref{fig1}(c),(d) hold. We also note that
\begin{equation}\label{32} [3,\sigma]([2,\sigma]\tens\id)=[3,\sigma]'(\id\tens [2,\sigma])\end{equation}
holds if and only if the Yang-Baxter or braid relations shown in Fig.~\ref{fig1}(g) holds. The proof is just a matter of writing out both sides as 6 terms and comparing. This is part of the theory of Hopf algebras in braided categories but applied now in the category of $A$-bimodules with tensor product over $A$.

\begin{lemma}\label{lem:nabla3leib} Let $\nabla$ be a bimodule connection on $\Omega^1$. The 3rd-order Leibniz rule
\[ \nabla_2\nabla\extd(ab)=(\nabla_2\nabla\extd a)b +[3,\sigma](\nabla\extd a\tens \extd b)+[3,\sigma]'(\extd a\tens\nabla\extd b)+ a(\nabla_2\nabla\extd b)\]
holds for all $a,b\in A$ if and only if the {\em Leibniz-compatibility} condition $\nabla_2\sigma_1=\sigma_2\nabla_2$ in Fig.~\ref{fig1}(e) holds. In this case, $\sigma$ obeys the braid relations and if $\nabla$ is torsion free and extendable then $R_\nabla$ is a bimodule map.
\end{lemma}
\begin{proof} Computing the left hand side using $\extd(ab)=(\extd a)b+a\extd b$ and the iterated Leibniz properties of $\nabla$ and $\nabla_2$, and  comparing, equality of the stated 3rd-order Leibniz rule needs
\[ [3,\sigma](\nabla\extd a\tens \extd b)=\sigma_{\nabla_2}(\nabla\extd a\tens\extd b)-\sigma_1(\extd a\tens\nabla\extd b)-\sigma_2\sigma_1(\extd a\tens\nabla\extd b)+\nabla_2[2,\sigma](\extd a\tens\extd b).\]
Dissecting $[3,\sigma]$ and $\nabla_2$ and cancelling terms requires the condition in Fig.~\ref{fig1}(e). More precisely, if we write $C(\omega,\eta) = \nabla_2 \sigma_1(\omega \tens \eta) - \sigma_2 \nabla_2(\omega \tens \eta)$ for the left hand side minus the right hand side, we need $C(\extd a,\extd b)$ to vanish. But by the Leibniz properties of a connection, one can see that $C:\Omega^{1\tens_A 2}\to \Omega^{1\tens_A 3}$ is a well-defined left module map. It follows that its vanishing on $\extd a,\extd b$ is equivalent to its vanishing on all $\omega,\eta$. In fact, the map $C=\doublenabla(\sigma)$ is the covariant derivative of $\sigma$ in the notation of \cite{BegMa} and as such is a right module map iff $\sigma$ intertwines the braiding before and after $\sigma$ by \cite[p302]{BegMa}, which is iff the braid-relations (g) given that the braiding on the bimodule $\Omega^1\tens_A\Omega^1$ is itself $\sigma_1\sigma_2$.  Moreover, zero is a right module map, hence condition (e) implies (g). Also, applying $\wedge_1$ to the condition (e) in the torsion free case extendable case gives the condition (f) and we then use Lemma~\ref{lemR}. Incidentally, applying instead $\wedge_2$ gives something which is automatic in the $\wedge$-compatible case, i.e. a weaker form of $\wedge$-compatibility.  \end{proof}

With these preliminary observations, we are now ready to state and prove our main result of this section.

\begin{theorem}\label{thmJ3}(Construction of $\CJ^3_A$).  Let $\nabla$ be torsion free, flat, $\wedge$-compatible, extendable and Leibniz-compatible. Then
\begin{align*}
  &a \bullet_3 s = a \bullet_2 s + (\nabla_2\nabla \extd a) s,&
  &s\bullet_3a=s \bullet_2 a +s (\nabla_2\nabla \extd a) ,\\
  &a \bullet_3 \omega =a \bullet_2 \omega+[3,\sigma](\nabla \extd a\tens\omega) ,&
  &\omega\bullet_3 a=\omega \bullet_2  a+[3,\sigma]'(\omega\tens\nabla \extd a) ,\\
  &a \bullet_3 (\omega^1 \tens \omega^2) = a \bullet_2 (\omega^1 \tens \omega^2) +[3,\sigma]'( \extd a \tens \omega^1 \tens \omega^2),&
  &(\omega^1 \tens \omega^2)\bullet_3a=  (\omega^1 \tens \omega^2)\bullet_2 a +[3,\sigma]( \omega^1 \tens \omega^2\tens \extd a), \\
  &a \bullet_3 (\omega^1 \tens \omega^2 \tens \omega^3)
  =  a \omega^1 \tens \omega^2 \tens \omega^3, &
  &(\omega^1 \tens \omega^2 \tens \omega^3)\bullet_3a = \omega^1 \tens \omega^2 \tens \omega^3a
\end{align*}
for  $s\in A$, $\omega\in \Omega^1$, $\omega^1\tens \omega^2\in \Omega^2_S$ and $\omega^1\tens \omega^2 \tens \omega^3 \in \Omega^3_S$,  makes $\CJ^3_A$ a bimodule and $j^3$ a bimodule map. Quotienting out $\Omega^3_S$ gives a bimodule surjection  $\pi_3: \CJ^3_A\to \CJ^2_A$  such that $\pi_3\circ j^3=j^2$.
\end{theorem}
\begin{proof} (1) All the stated action maps land in one of the components $\Omega^i_S$ of $\CJ^3_A$  given Lemma~\ref{nablaOmegaS} and given where the $[3,\sigma]$ and $[3,\sigma]'$ land.

(2) We check that we indeed have an action when acting on each degree. On degree 0,
\begin{align*} a\bullet_3(b\bullet_3 s)&=a\bullet_3(b\bullet_2 s+(\nabla_2\nabla\extd b)s)\\
&=a\bullet_2(b\bullet_2 s)+(\nabla_2\nabla\extd a)b s+[3,\sigma](\nabla\extd a\tens(\extd b)s)+[3,\sigma]'(\extd a\tens(\nabla\extd b)s)+ a(\nabla_2\nabla\extd b)s
\end{align*}
where we decompose $b\bullet_2 s=bs+ (\extd b)s+(\nabla\extd b)s$, apply the stated $a\bullet_3$ and recombine all the $a\bullet_2 $ parts of these as the first term of the result. The next three terms are the other terms from $a\bullet_3$ in this process. For this to equal $(ab)\bullet_3 s$, we need  the 3rd-order Leibniz rule in Lemma~\ref{lem:nabla3leib}.

Next, on degree 1,  we have
\begin{align*} a\bullet_3(b\bullet_3 \omega)&=a\bullet_3(b\bullet_2\omega+[3,\sigma](\nabla\extd b\tens\omega))\\
&=a\bullet_2(b\bullet_2 \omega)+[3,\sigma](\nabla\extd a\tens b\omega)+  [3,\sigma]'(\extd a\tens[2,\sigma](\extd b\tens\omega))+[3,\sigma](a\nabla\extd b\tens\omega),\\
(ab)\bullet_3\omega&=(ab)\bullet_2\omega+[3,\sigma](a\nabla\extd b+\extd a\tens\extd b+(\nabla\extd a)b+\sigma(\extd a\tens\extd b)\tens\omega)
\end{align*}
where for the second equality we decomposed $b\bullet_2\omega=b\omega+[2,\sigma](\extd b\tens\omega)$, applied $a\bullet_3$ to each term and then recombined the $a\bullet_2$ parts of these. For the last equality, we used the Leibniz properties on $ab$. For these expressions to be equal, we need the braid relations (g) according to equation (\ref{32}), but these are implied by Leibniz compatibility.

On degree 2, we have
\begin{align*} a\bullet_3(b\bullet_3 (\omega^1\tens \omega^2))&=a\bullet_3(b\bullet_2(\omega^1\tens \omega^2)+[3,\sigma]'(\extd b\tens \omega^1\tens \omega^2))\\
&=a\bullet_2(b\bullet_2 (\omega^1\tens \omega^2)+[3,\sigma]'(\extd a\tens  b\omega^1\tens  \omega^2)+a[3,\sigma]'(\extd b\tens \omega^1\tens \omega^2)\\
&=(ab)\bullet_2 (\omega^1\tens \omega^2)+[3,\sigma]'(((\extd a)b+a\extd b)\tens \omega^1\tens \omega^2)=(ab)\bullet_3 (\omega^1\tens \omega^2)
\end{align*}
without further conditions needed. The action on degree 3 works automatically also.

(3) The right actions work the same way without further conditions being needed. For example, on degree 0 we again need Lemma~\ref{lem:nabla3leib}.

(4) Now we check that these actions commute, so that we have a bimodule. On degree 0,
\begin{align*}(a\bullet_3s)\bullet_3 b&=(a\bullet_2 s+ (\nabla_2\nabla\extd a)s)\bullet_3 b\\
&=(a\bullet_2 s)\bullet_2 b+ as\nabla_2\nabla\extd b+[3,\sigma]'((\extd a)s\tens \nabla\extd b)+[3,\sigma]((\nabla\extd a)s\tens \extd b)+ (\nabla_2\nabla\extd a)sb,\\
a\bullet_3(s\bullet_3 b)&=a\bullet_3(s\bullet_2 b+s\nabla_2\nabla\extd b)\\
&=a\bullet_2(s\bullet_2 b)+(\nabla_2\nabla\extd a)sb+[3,\sigma](\nabla\extd a\tens  s\extd b)+[3,\sigma]'(\extd a\tens  s\nabla\extd b)+as\nabla_2\nabla\extd b
\end{align*}
by the same method as above, i.e. decomposing $a\bullet_2 s$, $s\bullet_2 b$, applying the other $\bullet_3$ and recombining. Comparing, we see that these are equal without further conditions needed.

On degree 1, we have
\begin{align*}(a\bullet_3\omega)\bullet_3 b&=(a\bullet_2 \omega+[3,\sigma](\nabla\extd a\tens \omega))\bullet_3 b\\
&=(a\bullet_2\omega)\bullet_2 b+[3,\sigma]'(a\omega\tens \nabla\extd b)+[3,\sigma]([2,\sigma](\extd a\tens \omega)\tens \extd b)+[3,\sigma](\nabla\extd a\tens \omega b),\\
a\bullet_3(\omega\bullet_3 b)&=a\bullet_3(\omega\bullet_2 b+[3,\sigma]'(\omega\tens \nabla\extd b))\\
&=a\bullet_2(\omega\bullet_2 b)+[3,\sigma](\nabla\extd a\tens \omega b)+[3,\sigma]'(\extd a\tens [2,\sigma](\omega\tens \extd b))+[3,\sigma]'(a\omega\tens \nabla\extd b)
\end{align*}
by the same method as before. For these expressions to be equal, we need the braid relations (g) according to equation (\ref{32}).

Similarly on degree 2,
\begin{align*}(a\bullet_3(\omega^1\tens \omega^2))\bullet_3 b&=(a\bullet_2 (\omega^1\tens \omega^2))\bullet_3b+[3,\sigma]'(\extd a\tens \omega^1\tens \omega^2)b\\
&=a\omega^1\tens \omega^2b+[3,\sigma](a\omega^1\tens \omega^2\tens \extd b)+[3,\sigma]'(\extd a\tens \omega^1\tens \omega^2b)\\
&=a\bullet_3((\omega^1\tens \omega^2)\bullet_2 b)+[3,\sigma](a\omega^1\tens \omega^2\tens \extd b)=a\bullet_3((\omega^1\tens \omega^2)\bullet_3 b)
\end{align*}
without further conditions needed. The bimodule property on degree 3 works automatically also.

(5) Finally, it remains to check that $j^3$ is a bimodule map (in fact, the stated $\bullet_3$ action was discovered by reverse-engineering this requirement). Similarly to the methods above, we have
\begin{align*}
    j^3(as)
    &=j^2(as)+\nabla_2\nabla\extd(as)\\
    &=a\bullet_2 j^2(s)+(\nabla_2\nabla\extd a)s+[3,\sigma](\nabla\extd a\tens \extd s)+[3,\sigma]'(\extd a\tens \nabla\extd s)+ a\nabla_3\nabla\extd s\\
    &=a\bullet_3(j^2(s)+\nabla_2\nabla\extd s)
    =a\bullet_3j^3(s)
\end{align*}
precisely by the $\nabla_2\nabla\extd$ Leibniz property in Lemma~\ref{lem:nabla3leib}. Similarly for the action from the other side. The properties of $\pi_3$ are clear and imply that $\pi_1\circ\pi_2\circ\pi_3$ splits $j^3$ given that $\pi_1\circ \pi_2$ split $j^2$ in Proposition~\ref{propJ12}.
\end{proof}

We are in fact forced to the above $\bullet_3$ actions by the requirement for $j^3$ to be a bimodule map and an assumption that the actions factor through the iterated derivatives of $a$. This then lays the groundwork for the general case.

\section{General jet bundle $\CJ^k_A$ over a noncommutative algebra}
\label{sec:JAgeneral}

We are now ready to construct bimodules $\CJ^k_A$ over a possibly noncommutative algebra $A$ for any $k$. Following the pattern in Section~\ref{sec:JA123}, we define the \emph{space of quantum symmetric k-forms} $\Omega^k_S$ as the joint kernel of $\Omega^{1 \tens_A k} = \Omega^1\tens_A\cdots\tens_A\Omega^1$ when applying $\wedge$ in any two adjacent places
\begin{align*}
    \Omega^k_S \coloneq \bigcap_{i<k} \ker \left\{\wedge_i\colon \Omega^{1\tens_A k} \rightarrow \Omega^{1\tens_A (i-1)} \tens_A \Omega^{2} \tens_A \Omega^{1\tens_A (k-i-1)}\right\}
\end{align*}
where $\wedge_i$ is taken to act on the $i,i+1$ tensor factors in $\Omega^{1\tens_A k}$. We set $\Omega^0_S=A$, $\Omega^1_S=\Omega^1$ and
 define the $k$-th order jet bimodule as
\[ \CJ^k_A\coloneq \bigoplus_{i=0}^{k}\Omega^i_S=A\oplus\Omega^1\oplus\Omega^2_S\oplus \cdots \oplus\Omega^k_S,\]
together with the  bullet $A$-actions $\bullet_k$ which we will need to define.

As before, we assume a bimodule connection $\nabla\colon \Omega^1 \rightarrow \Omega^{1}\tens_A \Omega^1$ on $\Omega^1$ with generalised braiding $\sigma\colon \Omega^{1}\tens_A \Omega^1 \rightarrow \Omega^{1}\tens_A \Omega^1$. This automatically extends to tensor product connections $\nabla_n$ on higher tensor powers of $\Omega^1$, which can be characterised recursively by
\begin{align*}
    &\nabla_n\colon \Omega^{1 \tens_A n} \rightarrow \Omega^1 \tens_A \Omega^{1 \tens_A n},&
    &\nabla_n = \nabla \tens \id^{n-1} + \sigma_1 (\id \tens \nabla_{n-1}),
\end{align*}
where $\id^k$ denotes the identity on $\Omega^{1\tens_A k}$ and we set $\nabla_1 = \nabla$, $\nabla_0 = \extd$. These connections are bimodule connections with generalised braidings
\begin{align*}
    &\sigma_{\nabla_n}\colon \Omega^{1 \tens_A n} \tens_A \Omega^1 \rightarrow \Omega^1 \tens_A \Omega^{1 \tens_A n},&
    &\sigma_{\nabla_n} = (\sigma \tens \id^n)(\id \tens \sigma_{\nabla_n}) = \sigma_1 \dots \sigma_n,
\end{align*}
where we recall that $\sigma_i$ denotes $\sigma$ acting on the $i,i+1$ tensor factors. It is useful to note that $\nabla_n$ can also be written as
\begin{align}
    \nabla_n &= \sum^{n-1}_{k=0} \sigma_1 \dots \sigma_k (\id^k \tens \nabla \tens \id^{n-1-k})
    = \nabla_l \tens \id^{n-l} + \sigma_1\dots \sigma_l (\id^{l} \tens \nabla_{n-l})
    \label{eq:Splitnablan}
\end{align}
for $l=1,\cdots,n-1$, where in the first expression we think of $\nabla_n$ as acting via $\nabla$ on each factor in $\Omega^{1 \tens_A n} = \Omega^1 \tens_A \dots \tens_A \Omega^1$ followed by swapping the new $\Omega^1$ to the far left  by repeatedly applying $\sigma$. In the second expression, we split the domain of $\nabla_n$ as $\Omega^{1 \tens_A n} = \Omega^{1\tens_A l} \tens_A \Omega^{1\tens_A (n-l)}$ and think of $\nabla_n$ as the tensor product connection on this.

If $\nabla$ is extendable then, similarly to the expression for $R_{\nabla_2}$ in Lemma~\ref{lemR}, we can find a recursive expression for the curvature $R_{\nabla_n} \colon \Omega^{1\tens_A n} \rightarrow \Omega^2 \tens_A \Omega^{1\tens_A n}$ of $\nabla_n$ as
\begin{align}
    R_{\nabla_n}  = R_{\nabla} \tens \id + (\sigma_{\Omega^1,\Omega^2} \tens \id) (\id \tens R_{\nabla_{n-1}})
    = \sum^{n-1}_{k=0} \sigma_{\Omega^1,\Omega^2,1}\dots \sigma_{\Omega^1,\Omega^2,k} (\id^k \tens R_\nabla \tens \id^{n-1-k}).
    \label{eq:HigherCurvatures}
\end{align}
Hence, $R_\nabla = 0$ directly implies $R_{\nabla_n} = 0$ for all $n$. The proof is deferred until later, see Lemma~\ref{lem:TensorProductCurvature}. Armed with these tensor product connections $\nabla_n$, we now define
\[ j^k\colon A\to \CJ^k_A,\quad j^k(s)=\sum_{i=0}^k \nabla^i s;\quad \nabla^i\colon A \rightarrow \Omega^{1 \tens_A i},\quad \nabla^i s = \nabla_{i-1} \nabla^{i-1} s = \nabla_{i-1} \nabla_{i-2} \dots \nabla \extd s\]
where $\nabla^0=\id$ and $\nabla^1=\extd$. We will need to show that the higher order derivatives $\nabla^i$ land in the right places so that $\im \, j^k\subseteq \CJ^k_A$ and that $j^k$  is a bimodule map for the $\bullet_k$ actions on $\CJ^k_A$ and left and right multiplication on $A$.
This will entail extensive use of braided binomials $[{n \atop k},\sigma]$ cf \cite{Ma:book},  in line with the braided integers $[n,\sigma]$ encountered already in Section~\ref{sec:JA123}.

\subsection{Higher order Leibniz rules}
\label{sec:HigherConnectionsA}

In this section, we study the higher order derivatives $\nabla^n$ and establish their key properties.

\begin{lemma}
\label{lem:higherderivetivessymmetric}
    Let $\nabla$ be a torsion free, $\wedge$-compatible bimodule connection. Then $\nabla_n$ restricts to a bimodule connection on $\Omega^n_S$. If in addition $\nabla$ is flat then  $ \im \,\nabla^n \subseteq \Omega^n_S$.
\end{lemma}
\begin{proof} We take $n\ge 3$ as earlier cases are clear or covered in Section~\ref{sec:JA123}. For the 1st assertion, we have $\wedge_i\nabla_n=0$ on $\Omega^n_S$ for $i\ge 3$ by using the original definition of $\nabla_n$; for the first term this acts commutes to act on the later factors of $\Omega^n_S$ and for the 2nd term it acts on the output of $\nabla_{n-1}$ which as induction hypothesis we assume lands in $\Omega^1\tens_A\Omega^{n-1}_S$. For $\wedge_2$, we use (\ref{eq:Splitnablan}) with $l=2$. Then $\wedge_2$ vanishes on the first term as $\nabla_2$ has its output in $\Omega^1\tens\Omega^2_S$ and vanishes on the 2nd term using the 2nd part of $\wedge$-compatibility in Fig.~\ref{fig1}(c).

Now assume that $R_\nabla=0$. Then the higher curvatures also vanish by \eqref{eq:HigherCurvatures} and hence
\begin{align*}
    \wedge_1 \nabla^n a &= \wedge_1 \nabla_{n-1} \nabla^{n-1} a
    = \wedge_1 (\nabla \tens \id + \sigma_1( \id \tens \nabla_{n-2}))\nabla_{n-2}\nabla^{n-2} a \\
    &= (\extd \tens \id - \id \wedge \nabla_{n-2})\nabla_{n-2}\nabla^{n-2} a
    = R_{n-2}(\nabla^{n-2} a) = 0.
\end{align*}
For $i\ge 2$, we have $\wedge_i\nabla^na=\wedge_i\nabla_{n-1}\nabla^{n-1}a=0$ since, as inductive hypothesis, $\nabla^{n-1}a$ has image in $\Omega^{n-1}_S$ and by our first assertion, $\nabla_{n-1}$ acts on this. \end{proof}

The higher order Leibniz rules for $\nabla^n (a b)$, as in the  classical case, will need a notion of `binomials'. Motivated by Lemma~\ref{lem:nabla3leib}, we use braided binomials defined recursively as follows. 
\begin{definition}(cf \cite{Ma:book,Ma:hod})
\label{def:binomials}
The braided binomial morphisms $\left[{n \atop k},\sigma \right]\colon \Omega^{1 \tens_A n} \rightarrow \Omega^{1 \tens_A n}$ for $k=0,\dots,n$ are recursively defined as
\begin{align*}
    &\left[{n \atop 0},\sigma \right] = \left[{n \atop n},\sigma \right] = \id^n,&
    &\left[{n \atop k},\sigma \right] = \left(\id \tens \left[{n-1 \atop k-1},\sigma \right]\right) \sigma_1 \dots \sigma_{n-k} +  \left(\id \tens \left[{n-1 \atop k},\sigma \right]\right).
  \end{align*}
Furthermore we define the braided integers, co-braided integers and braided factorials
\begin{align*}
    &[n,\sigma]\coloneq \left[{n \atop 1},\sigma \right]
    = \id + \sigma_{n-1} + \sigma_{n-2} \sigma_{n-1} + \dots + \sigma_{1} \dots \sigma_{n-1},\\
    &[n,\sigma]'\coloneq \left[{n \atop n-1},\sigma \right]
    = \id + \sigma_{1} + \sigma_{2} \sigma_{1} + \dots + \sigma_{n-1} \dots \sigma_{1},\\
    &[n,\sigma]! = [n,\sigma]([n-1,\sigma]! \tens \id).
\end{align*}\end{definition}
For example,
\begin{align*}
    &\left[2,\sigma \right] = \id^2 + \sigma_1,&
    &\left[3,\sigma \right] = \id^3 + \sigma_2 + \sigma_1 \sigma_2,&
    &\left[3,\sigma \right]' = \id^3 + \sigma_1 + \sigma_2 \sigma_1
\end{align*}
are exactly the braided integers we have encountered in Section~\ref{sec:JA123} when dealing with $\CJ^2_A$ and $\CJ^3_A$ respectively.  The present context and conventions are different from \cite{Ma:book,Ma:hod} and also we do not {\em a priori} assume that $\sigma$ obeys the braid relations. We assume $n\ge 1$, albeit the case $n=0$ can be understood formally as the identity on $A$.

\begin{lemma}
\label{lem:binomialssymmetricforms}
If $\nabla\colon \Omega^1 \rightarrow \Omega^1 \tens_A \Omega^1$ is  torsion free, $\wedge$-compatible and extendable then the braided binomials $\left[{n \atop k},\sigma\right]$ obey \begin{align*}
  \left[{n \atop k},\sigma \right]\left(\Omega^{n-k}_S \tens_A \Omega^{ k}_S\right) \subseteq \Omega^{n}_S,\quad
    [n,\sigma]! \left(\Omega^{1\tens n}\right) \subseteq \Omega^n_S.
\end{align*}
\end{lemma}
\begin{proof} $n\le 2$ are clear or covered Section~\ref{sec:JA3}, hence we assume $n\ge 3$. For the 1st stated result, the $k=0,n$ cases are immediate so we also assume $1\le k\le n-1$. We first prove that
\begin{equation}\label{partiallem}\left[{n \atop k},\sigma \right](\Omega^{n-k}_S \tens_A \Omega^{k}_S) \subseteq \Omega^1 \tens_A \Omega^{n-1}_S\end{equation}
 by induction on $k$. We use the recursive Definition~\ref{def:binomials} of the braided binomials and split $\Omega^{n-k}_S \tens_A \Omega^{k}_S$ in a preferred way for each term. Looking at the first term in the recursive definition, we note that
\begin{align*}
  \sigma_1 \dots \sigma_{n-k} (\Omega^{n-k}_S \tens_A \Omega^{k}_S)\subseteq \sigma_1 \dots \sigma_{n-k}(\Omega^{n-k}_S \tens_A \Omega^1 \tens_A \Omega^{k-1}_S) \subseteq  \Omega^1 \tens_A \Omega^{n-k}_S \tens_A \Omega^{k-1}_S
\end{align*}
as easily seen from the string diagram in Fig.~\ref{fig:binomialssymmetricforms}(a), applying $\wedge_i$ with $i=2,\dots,n-1$  other than $i=n-k+1$ and using $\wedge$-compatibility shown in Fig.~\ref{fig1}(c). We then apply $\id\tens \left[{n-1 \atop k-1},\sigma \right]$ see that the first term lands in $\Omega^1\tens_A\Omega^{n-1}_S$. We assumed the 1st stated assertion for $n=1$ as an induction hypothesis. Now going back to the second term of the recursive definition, we similarly consider $\Omega^{n-k}_S \tens_A \Omega^{k}_S\subseteq
    \Omega^1 \tens_A \Omega^{n-1-k}_S \tens_A  \Omega^{k}_S$ and apply $\id\tens \left[{n-1 \atop k},\sigma \right]$ using our inductive hypothesis to again land in the $\Omega^1\tens_A\Omega^{n-1}_S$. This proves (\ref{partiallem}).

To finish the proof of the first stated result, we have to consider the action of $\wedge_1$. We first assume  $n-2\geq k \geq 2$ and defer $k=1,n$ till later. Using the recursive definition of the braided binomials twice now gives
\begin{align*}
    \wedge_1 &\left[{n \atop k},\sigma \right]
    = \wedge_1 \left(\id \tens  \left[{n-1 \atop k-1},\sigma \right]\right) \sigma_1 \dots \sigma_{n-k}
    +  \wedge_1 \left(\id \tens  \left[{n-1 \atop k},\sigma \right]\right)\\
    &=\wedge_1  \left(\id^2\tens  \left[{n-2 \atop k-2},\sigma \right]\right) \sigma_2 \dots \sigma_{n-k}\sigma_1 \dots \sigma_{n-k}
    + \wedge_1  \left(\id^2 \tens  \left[{n-2 \atop k-1},\sigma \right]\right)\sigma_1 \dots \sigma_{n-k}\\
    &\quad+\wedge_1 \left(\id^2 \tens  \left[{n-2 \atop k-1},\sigma \right]\right) \sigma_2 \dots \sigma_{n-k}
    +\wedge_1  \left(\id^2 \tens  \left[{n-2 \atop k},\sigma \right]\right).
\end{align*}
In all terms, we can move $\wedge_1$ to the right as it commutes with the braided binomial in the other tensor factors. Then the last term vanishes as we assume our expression acts on $\Omega^{n-k}_S \tens_A \Omega^{k}_S$ with at most $k=n-2$. The 2nd and third expressions combine to the right factor $(\id + \sigma_1) \sigma_2 \dots \sigma_{n-k}$ which is again killed by $\wedge_1$.  To see that the first term vanishes, we use the extendability property in Fig.~\ref{fig1}(d) as shown in Fig.~\ref{fig:binomialssymmetricforms}(b) to move $\wedge_1$ to the right where it again vanishes when acting on $\Omega^{n-k}_S \tens_A \Omega^{k}_S$.

It is left to show how $\wedge_1$ acts on $\left[{n \atop k},\sigma \right]$ for $k=1,n-1$, i.e. on the braided integers $[n,\sigma]$ and co-braided integers $[n,\sigma]'$. When applying $\wedge_1$ to $\left[n,\sigma \right]$, we can swap the wedge product with all the terms up to $\sigma_3\dots \sigma_{n-1}$, where it then acts directly on $\Omega^{n-1}_S \tens_A \Omega^1$. The last two terms can be written as $\wedge(\id + \sigma_1) \sigma_2\dots \sigma_{n-1}$ and therefore vanish. In the case of $\left[n,\sigma \right]'$, the first two terms will vanish due to $\wedge(\id + \sigma_1) = 0$. For the rest, we again use extendability (Fig.~\ref{fig1}(d)), to swap $\wedge_1$ with the braidings, which turns it into $\wedge_2$ and makes it act directly on $\Omega^1 \tens_A \Omega^{n-1}_S$. This completes the proof of the first stated result.

Finally, setting $k=1$, we have $[n,\sigma](\Omega_S^{n-1}\tens_A\Omega^1)\subseteq \Omega^n_S$. Iterating this gives the second stated result, using the definition of the braided factorials.
\end{proof}

\begin{figure}
 \[ \includegraphics[scale=.8]{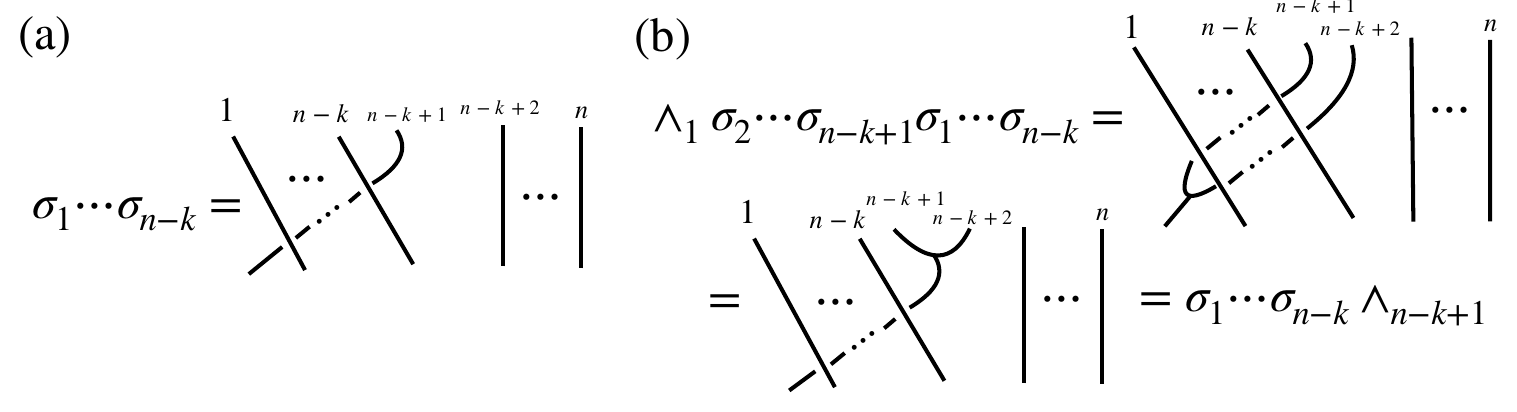}\]
 \caption{\label{fig:binomialssymmetricforms} String diagramss used in the proof of Lemma~\ref{lem:binomialssymmetricforms}.}
\end{figure}

We will also need a further property of the braided binomials, which is a generalisation of  \cite[Thm~10.4.12]{Ma:book},  recovered as the case $m=1$.

\begin{lemma}
\label{lem:propertybimoduleactions}
If $\sigma\colon \Omega^1 \tens_A \Omega^1\rightarrow\Omega^1 \tens_A \Omega^1$ satisfies the braid relations $\sigma_i \sigma_{i+1} \sigma_i= \sigma_{i+1} \sigma_i \sigma_{i+1}$, then the braided binomials satisfy the following relation
for all $n,k,m$ with $n \geq k \geq m$
\begin{align*}
    \left[{n \atop k},\sigma \right] \left( \id^{n-k}\tens \left[{k \atop m},\sigma \right]\right)
    = \left[{n \atop m},\sigma \right]
    \left(\left[{n-m \atop k-m},\sigma \right] \tens \id^{m}\right).
\end{align*}
\end{lemma}
\begin{proof} This proceeds similarly to the proof for the $m-1$ case and will need a result  (cf. \cite[Lem.~10.4.11]{Ma:book}) that if $\sigma$ obeys the braid relations then
\begin{align}
    \label{eq:2ndconditionbimoduleactions}
    \sigma_1 \dots \sigma_{n-m} \left(\left[{n - m \atop k-m},\sigma \right] \tens \id^{m}\right)
    =\left(\id \tens \left[{n - m \atop k-m},\sigma \right] \tens \id^{m-1}\right) \sigma_1 \dots \sigma_{n-m}.
\end{align}
The proof is analogous to that in \cite{Ma:book} so we omit details. This expresses that the braided binomial can be taken through braid crossings and is the only place where we use the braid relations.

Also observe that if $m=0,k$ (appropriately understood) then there is nothing to prove. So we assume $1\le m\le k-1$. Next, using the Definition~\ref{def:binomials} for the first binomial and in one case also for the 2nd binomial,  we have
\begin{align*}
    &\left[{n \atop k},\sigma \right]
    \left( \id^{n-k}\tens  \left[{k \atop m},\sigma \right]\right)
    = \left\{\left(\id \tens  \left[{n-1 \atop k-1},\sigma \right] \right)
    \sigma_1 \dots \sigma_{n-k}
    +
    \left(\id \tens  \left[{n-1 \atop k},\sigma \right] \right)\right\}
    \left( \id^{n-k}\tens  \left[{k \atop m},\sigma \right]\right)\\
    &=  \left(\id \tens  \left[{n-1 \atop k-1},\sigma \right] \right)
    \sigma_1 \dots \sigma_{n-k}
    \left( \id^{n-k}\tens  \id \tens  \left[{k-1 \atop m-1},\sigma \right]\right)
    \sigma_{n-k+1} \dots \sigma_{n-m}\\
    &\quad+
    \left(\id \tens  \left[{n-1 \atop k-1},\sigma \right] \right)
    \sigma_1 \dots \sigma_{n-k}
    \left( \id^{n-k}\tens  \id \tens  \left[{k-1 \atop m},\sigma \right]\right)
    +
    \left(\id \tens  \left[{n-1 \atop k},\sigma \right] \right)
    \left( \id^{n-k}\tens  \left[{k \atop m},\sigma \right]\right)\\
     &=\left\{ \id\tens   \left[{n-1 \atop k-1},\sigma \right]
    \left( \id^{n-k}\tens   \left[{k-1 \atop m-1},\sigma \right]\right)\right\}\sigma_1 \dots  \sigma_{n-m}\\
    &\quad+
    \left\{ \id\tens\left[{n-1 \atop k-1},\sigma \right]
    \left( \id^{n-k}\tens  \left[{k-1 \atop m},\sigma \right]\right) \right\}\sigma_1 \dots \sigma_{n-k}
    +
    \id\tens  \left[{n-1 \atop k},\sigma \right]
    \left( \id^{n-k}\tens  \left[{k \atop m},\sigma \right]\right)
\end{align*}
where we commute the $\sigma_1\cdots\sigma_{n-k}$ arising from the inductive definition of the first braided binomial to the right since it acts on different tensor factors from the relevant parts of the other braided binomial. In each of the three terms we now use the stated result for $n-1$ as an inductive hypothesis and then recombine in the reverse process to the above. Thus, our expression is
\begin{align*}
        &=\left\{\id \tens \left(\left[{n-1 \atop m-1},\sigma \right] \right)
    \left(\left[{n - m \atop k-m},\sigma \right] \tens  \id^{m-1}\right)\right\}\sigma_1 \dots \sigma_{n-m}\\
    &\quad+
    \left\{\id\tens  \left[{n-1 \atop m},\sigma \right]\left( \left[{n-1-m \atop k-1-m},\sigma \right] \tens  \id^{m}\right)\right\}\sigma_1 \dots \sigma_{n-k}
    + \id\tens  \left[{n-1 \atop m},\sigma \right]\left(\left[{n-1-m \atop k-m},\sigma \right]\tens \id^m\right)\\
    &=\left(\id \tens  \left[{n-1 \atop m-1},\sigma \right] \right)
    \left(\id \tens  \left[{n - m \atop k-m},\sigma \right] \tens  \id^{m-1}\right) \sigma_1 \dots \sigma_{n-m}\\
    &\quad+
    \left(\id\tens  \left[{n-1 \atop m},\sigma \right]\right)\left\{\left( \id \tens  \left[{n-1-m \atop k-1-m},\sigma \right] \tens  \id^{m}\right) \sigma_1 \dots \sigma_{n-k}
    + \left(\id \tens  \left[{n-1-m \atop k-m},\sigma \right]\tens \id^m\right)\right\}\\
    &=\left(\id \tens  \left[{n-1 \atop m-1},\sigma \right] \right)
    \left(\id \tens  \left[{n - m \atop k-m},\sigma \right] \tens  \id^{m-1}\right) \sigma_1 \dots \sigma_{n-m}
    +
    \left(\id\tens  \left[{n-1 \atop m},\sigma \right]\right)\left(\left[{n-m \atop k-m},\sigma \right] \tens  \id^{m}\right)\\
    &=
    \left(\id \tens  \left[{n-1 \atop m-1},\sigma \right] \right)
    \sigma_1 \dots \sigma_{n-m}
    \left(\left[{n - m \atop k-m},\sigma \right] \tens  \id^{m}\right)
    +
    \left(\id\tens  \left[{n-1 \atop m},\sigma \right]\right)\left(\left[{n-m \atop k-m},\sigma \right] \tens  \id^{m}\right)\\
    &=
    \left[{n \atop m},\sigma \right]
    \left(\left[{n - m \atop k-m},\sigma \right] \tens  \id^{m}\right).
\end{align*}
For the 3rd equality we used the inductive Definition~\ref{def:binomials} in reverse and for the 4th we used our initial observation.  The fifth is our inductive Definition~\ref{def:binomials} in reverse again.
\end{proof}

An immediate consequence of Lemmas~\ref{lem:binomialssymmetricforms},~\ref{lem:propertybimoduleactions} is the following algebra structure on the same vector space as the jet bundle and which can be used to express its $\bullet_k$ bimodule structure.

\begin{corollary}
\label{cor:AlgebraQuantumSymmetricForms}
If $\nabla\colon \Omega^1 \rightarrow \Omega^1 \tens_A \Omega^1$ is torsion free, $\wedge$-compatible and extendable, and $\sigma$ obeys the braid relations then the braided binomials restrict to define a unital associative product
\begin{align*}
    &\odot=\left[{i+j \atop j},\sigma \right]\colon \Omega^{i}_S \tens_A \Omega^{j}_S \rightarrow \Omega^{i+j}_S
\end{align*}
making $\Omega_S \coloneq \bigoplus^\infty_{k=0} \Omega^k_S$ into a graded algebra of `quantum symmetric forms'.
\end{corollary}
\begin{proof}
The product here is to take the tensor product and then use to the braided-binomials to symmetrize. That the product then lands in the right space is clear from Lemma~\ref{lem:binomialssymmetricforms}. Let $\omega \in \Omega^i_S, \eta \in \Omega^j_S, \rho \in \Omega^k_S$ and denoting the algebra product by $\odot$, we have
\begin{align*}
    &(\omega \odot \eta) \odot \rho
    = \left( \left[{i+j \atop j},\sigma \right] \omega \tens \eta\right) \odot\rho
    = \left[{i+j+k \atop k},\sigma \right]\left(\left[{i+j \atop j},\sigma \right]\tens \id^k\right) \omega \tens \eta \tens \rho\\
    &=
    \left[{i+k+j \atop k+j},\sigma \right]\left(\id^i\tens \left[{j+k \atop k},\sigma \right]\right) \omega \tens \eta \tens \rho
    =
    \omega\odot \left(\left[{j+k \atop k},\sigma \right] \eta \tens \rho\right)
    = \omega\odot \left(\eta \odot \rho\right).
\end{align*}
where the 3rd equality is Lemma~\ref{lem:propertybimoduleactions}. The unit is that of $A$ in degree 0. \end{proof}

The algebra $\Omega_S$ is a subalgebra (namely, the restriction to quantum-symmetric forms) of a `braided shuffle algebra' ${\rm Sh}_\sigma(\Omega^1)$  defined as the tensor algebra $T_A(\Omega^1)$ with the braided-binomials as product. The case over a field is the shuffle braided-Hopf algebra ${\rm Sh}(V)$ in \cite{MaTao:dua}.  Next, we want to see how the Leibniz compatibility condition $\nabla_2 \sigma_1 = \sigma_2 \nabla_2$  translates to the connections $\nabla_n$.

\begin{lemma}
\label{lem:Nabla&Binomials}
 Assuming Leibniz compatibility Fig.~\ref{fig1}(e), we have for $i = 1,\dots,n-1$
 \begin{align*}
    \nabla_n \sigma_i = \sigma_{i+1} \nabla_n,\quad
    \nabla_n\left[{n \atop k},\sigma \right] = \left(\id \tens  \left[{n \atop k},\sigma \right]\right)\nabla_n.
 \end{align*}
\end{lemma}
\begin{proof} The second part is an immediate corollary of the first part since the braided binomial is built from sums and products of braids $\sigma_i$ with $i$ in the relevant range. To prove the first part, we split the domain of $\nabla_n$ as $\Omega^{1 \tens_A n} = \Omega^{1\tens_A (i-1)} \tens_A \Omega^{1\tens_A 2} \tens_A \Omega^{1\tens_A (n-i-1)}$ in a similar manner to~\eqref{eq:Splitnablan}, so that
\begin{align*}
    \nabla_n \sigma_i = [(\nabla_{i-1}\tens  \id^2 \tens  \id^{n-i-1})
    + \sigma_1 \dots \sigma_{i-1} (\id^{i-1}\tens  \nabla_2 \tens  \id^{n-i-1})
    + \sigma_1 \dots \sigma_{i+1} (\id^{i-1}\tens  \id^2 \tens  \nabla_{n-i-1})]\sigma_i.
\end{align*}
In the first term  $\sigma_i$ and $\nabla_{i-1}$ do not act on the same factors of $\Omega^{1\tens_A n}$ so we can pass $\sigma_i$ to the left, keeping in mind that it turns into $\sigma_{i+1}$ since $\nabla_{i-1}$ adds an additional $\Omega^1$ to the tensor product. For the second term we use Leibniz compatibility $\nabla_2 \sigma_1 = \sigma_2 \nabla_2$ but remembering the placement of $\nabla_2$ this appears now as $\nabla_2\sigma_i=\sigma_{i+1}\nabla_2$ and we then move $\sigma_{i+1}$ again to the far left as it acts on different factors. Finally, for the 3rd term, we can pass $\sigma_i$ left through $\nabla_{n-i-1}$ as it acts on different spaces and then
use $\sigma_1 \dots \sigma_n \sigma_l = \sigma_{l+1}\sigma_1 \dots \sigma_n$ for $n > l \geq 1$, which follows from the braid relations, which by Lemma~\ref{lem:nabla3leib} are implied by Leibniz compatibility. Hence for all three terms, we can move $\sigma_i$ to the far left as $\sigma_{i+1}$, as required. \end{proof}

With this in hand, we are ready to tackle the higher order Leibniz rules.

\begin{lemma}
\label{lem:leibrule}
Let  $\nabla$ be a Leibniz compatible bimodule connection on $\Omega^1$ as in Fig.~\ref{fig1}(e). Then $\nabla^n\colon A\rightarrow \Omega^{1\tens_A n}$ obeys the $n$-th order Leibniz rule
\begin{align*}
 \nabla^n(ab) = \sum^n_{k=0} \left[{n\atop k},\sigma\right] \nabla^{n-k} a \tens  \nabla^k b,
\end{align*}
for all $a,b \in A$.
\end{lemma}
\begin{proof} This is clear for $n=0$ and $n=1$ (where it is the bimodule derivation property of $\extd$). We proceed by induction, assuming the result for $n$. Then by Lemma~\ref{lem:Nabla&Binomials},
\begin{align*}
    \nabla^{n+1}(ab) = \nabla_n \nabla^n (ab)
    = \nabla_n \left(\sum^n_{k=0} \left[{n\atop k},\sigma\right] \nabla^{n-k} a \tens  \nabla^k b\right)
   = \sum^n_{k=0} \left(\id \tens  \left[{n\atop k},\sigma\right]\right)  \nabla_n(\nabla^{n-k} a \tens  \nabla^k b).
\end{align*}
The $k=0,n$ terms have $\nabla_n$ acting on $ (\nabla^n a) b$, $a\nabla^n b$ respectively and are computed by the bimodule connection property (where  $\sigma_{\nabla_n} = \sigma_1 \dots \sigma_n$), while for the intermediate terms we use (\ref{eq:Splitnablan}) with $l=n-k$. This gives
\begin{align*}
    \nabla^{n+1}(ab)
    &= \nabla_{n}((\nabla^{n} a) b)+  \nabla_{n}(a (\nabla^{n} b) )
    + \sum^{n-1}_{k=1} \left(\id \tens  \left[{n\atop k},\sigma\right]\right) \nabla_{n-k}(\nabla^{n-k} a) \tens  \nabla^k b \\
    &\quad + \sum^{n-1}_{k=1} \left(\id \tens  \left[{n\atop k},\sigma\right]\right) \sigma_1 \dots \sigma_{n-k} \nabla^{n-k} a \tens  \nabla_k(\nabla^k b) \\
    & = (\nabla^{n+1} a) b
    +\sigma_1 \dots \sigma_n \nabla^{n} a \tens  \extd b+ \extd a \tens   \nabla^{n} b + a  \nabla^{n+1} b.
\\
    &\quad + \sum^{n-1}_{k=1} \left(\id \tens  \left[{n\atop k},\sigma\right]\right) \nabla^{n+1-k} a \tens  \nabla^k b + \sum^{n-1}_{k=1} \left(\id \tens  \left[{n\atop k},\sigma\right]\right) \sigma_1 \dots \sigma_{n-k} \nabla^{n-k} a \tens  \nabla^{k+1} b\\
   &    = (\nabla^{n+1} a) b
    + \sum^{n}_{k=1} \left(\id \tens  \left[{n\atop k},\sigma\right]\right) \nabla^{n+1-k} a \tens  \nabla^k b \\
    &\quad+ \sum^{n-1}_{k=0} \left(\id \tens  \left[{n\atop k},\sigma\right]\right) \sigma_1 \dots \sigma_{n-k} \nabla^{n-k} a \tens  \nabla^{k+1} b)
    + a  \nabla^{n+1} b\\
    &=(\nabla^{n+1} a) b
    + \sum^{n}_{k=1} \left[{n+1\atop k},\sigma\right]
    \nabla^{n+1-k} a \tens  \nabla^k b
    + a  \nabla^{n+1} b
 \end{align*}
where the 3rd equality includes the  $\nabla^{n} a \tens  \extd b$ term as the $k=0$ term of the second sum and the $\extd a \tens   \nabla^{n} b$ terms as the $k=n$ term of the first sum. For the 4th equality, shift $k$ in the second sum so that it sums over the same indices as the first and use the recursive Definition~\ref{def:binomials} in reverse.
\end{proof}

Note that if $\nabla$ is in addition torsion free, flat, $\wedge$-compatible, extendable and $\sigma$ obeys the braid relations then we can write the last result as
\begin{equation}\label{nablancirc} \nabla^n(ab)=\sum_{k=0}^n \nabla^{n-k}a\odot\nabla^k b.\end{equation}
Also note that the collection of all $\nabla_n$ can be viewed as a connection $\nabla$ on $\Omega_S$ acting on the appropriate degree. Likewise, by Lemma~\ref{lem:higherderivetivessymmetric} and (\ref{eq:Splitnablan}) with $l=n-k$ one can see that $\nabla_n$ on $\Omega^{n-k}_S\tens_A\Omega^k_S$ is just the tensor product of the connections on $\Omega^{n-k}_S$ and $\Omega^k_S$ and this data can be viewed as a tensor product connection $\nabla$ on $\Omega_S\tens_A\Omega_S$ acting in the appropriate degrees. Then the 2nd part of Lemma~\ref{lem:Nabla&Binomials} just says geometrically that $\odot$ intertwines these bimodule connections, i.e., \[  \begin{array}{ccc} \Omega_S\tens_A\Omega_S&{\buildrel \nabla\over \to} &\Omega^1\tens_A\Omega_S\tens_A\Omega_S\\
\odot\downarrow & & \downarrow \id\tens\odot\\
\Omega_S&{\buildrel \nabla\over \to}&\Omega^1\tens_A\Omega_S\end{array}\]
commutes. In the notation of \cite[Chap~4]{BegMa}, the map $\odot$ is covariantly constant as a map, i.e. $\doublenabla(\odot)=0$.

\subsection{Jet bimodules $\CJ^k_A$}

With the above theory of higher order derivatives on $\Omega^{1\tens_A n}$ and braided binomials in hand, we have now done all the work to write down the jet bimodules of any order over a possibly noncommutative algebra $A$.

\begin{theorem}
\label{thm:JkA}
Let $\nabla$ be a torsion free, flat, $\wedge$-compatible, extendable and Leibniz-compatible bimodule connection on $\Omega^1$ as shown in Fig. \ref{fig1}(a)-(e). Then
\begin{align*}
    \CJ^k_A = \bigoplus^k_{i=0}  \Omega^i_S = A\oplus \Omega^1_S \oplus \cdots \oplus \Omega^k_S,\quad j^k\colon A \rightarrow \CJ^k_A,\quad j^k(s)=s+\nabla^1 s+\cdots+\nabla^k s
\end{align*}
form an $A$-bimodule and bimodule map with actions $\bullet_k$ given by
\begin{align*}
    &a \bullet_k \omega_j
    = j^{k-j}(a) \odot  \omega_j
    = \sum^k_{i=j} \nabla^{i-j} a \odot  \omega_j ,&
    &\omega_j \bullet_k a
    = \omega_j \odot j^{k-j}(a)
    = \sum^k_{i=j} \omega_j \odot \nabla^{i-j} a .
\end{align*}
where $\omega_j \in \Omega^j_S$ and $\odot$ is the product in Corollary \ref{cor:AlgebraQuantumSymmetricForms} applied in the appropriate degree. Quotienting out $\Omega^k_S$ gives a bimodule surjection $\pi_k\colon \CJ^k_A\to \CJ^{k-1}_A$ such that $\pi_k\circ j^k=j^{k-1}$.
\end{theorem}

\begin{proof}
(1) First note that both the action $\bullet_k$ and the jet prolongation map $j^k$ land in $\CJ^k_A$. Recall that $\nabla^l a \in \Omega^{l}_S$ due to Lemma \ref{lem:higherderivetivessymmetric}, from which immediately follows that $\im \,j^k \subset \CJ^k_A$. For the action $\bullet_k$, as we assume $\omega_j \in \Omega^{j}_S$, then due to Corollary \ref{cor:AlgebraQuantumSymmetricForms} we have $\nabla^{i-j}a \odot \omega_j \in \Omega^{i}_S$. Since the sum in $a\bullet_k\omega_j$ runs over $i=j,\dots,k$, the highest degree terms will be in $\Omega^k_S$ and  hence $a\bullet_k\omega_j \in \CJ^k_A$.

(2) Next, for $a,s \in A$, we compute
\begin{align*}
    j^k(as) &= \sum^k_{i=0} \nabla^{i}(as)
    = \sum^k_{i=0} \sum^i_{j=0} \nabla^{i-j}a \odot \nabla^js
    = \sum^k_{j=0} \sum^k_{i=j} \nabla^{i-j}a \odot \nabla^js
    = \sum^k_{j=0} j^{k-j}(a) \odot \nabla^js= \sum^k_{j=0} a \bullet_{k} \nabla^js
   \\
   & =a \bullet_{k} j^k(s),\\
  j^k(sa) &= \sum^k_{i=0} \nabla^{i}(sa)
  = \sum^k_{i=0} \sum^i_{j=0} \nabla^{j}s \odot \nabla^{i-j}a
  = \sum^k_{j=0} \sum^k_{i=j} \nabla^{j}s \odot \nabla^{i-j}a
  = \sum^k_{j=0} \nabla^{j}s \odot j^{i-j}(a)= \sum^k_{j=0} \nabla^{j}s \bullet_k a\\
 & = j^k(s)\bullet_{k} a
\end{align*}
so that if $\bullet_k$ are actions then $j^k$ is a bimodule map.

(3) Next, for $\omega_j \in \Omega^j_S$, $j=0,\dots,k$ and $a,b\in A$, we compute
\begin{align*}
  &a\bullet_k (b \bullet_k \omega_j)
  = a\bullet_k (j^{k-j}(b) \odot \omega_j)
  = \sum^k_{i=j} a\bullet_k (\nabla^{i-j}b \odot \omega_j)
  = \sum^k_{i=j} j^{k-i}(a) \odot (\nabla^{i-j}b \odot \omega_j)\\
  &= \sum^k_{i=j} (j^{k-i}(a) \odot \nabla^{i-j}b) \odot \omega_j
  = \sum^k_{i=j} (a \bullet_{k-j} \nabla^{i-j}b) \odot \omega_j
  = (a \bullet_{k-j} j^{k-j}(b)) \odot \omega_j
  = j^{k-j}(ab) \odot \omega_j = (ab) \bullet_k \omega_j
\end{align*}
where we used that $\odot$ is associative by Corollary \ref{cor:AlgebraQuantumSymmetricForms} and that $j^{k-j}$ is a bimodule map w.r.t. $\bullet_{k-j}$ for all $j=0,\dots,k$. In the same manner we can compute
\begin{align*}
    &(\omega_j \bullet_k b) \bullet_k a
    = (\omega_j \odot j^{k-j}(b))\bullet_k a
    = \sum^k_{i=j} (\omega_j \odot \nabla^{i-j}b)\bullet_k a
    = \sum^k_{i=j} (\omega_j \odot \nabla^{i-j}b) \odot j^{k-i}(a)\\
    &= \sum^k_{i=j} \omega_j \odot (\nabla^{i-j}b \odot j^{k-i}(a))
    = \sum^k_{i=j} \omega_j \odot (\nabla^{i-j}b \bullet_{k-j} a)
    = \omega_j \odot (j^{k-j}(b)) \bullet_{k-j} a)
    = \omega_j \odot j^{k-j}(ba)
    = \omega_j \bullet_k (ab)
\end{align*}
It is just left to show that the left and right actions commute. This is
\begin{align*}
    &(a\bullet_k \omega_j) \bullet_k b
    = \sum^{k-j}_{m=0} (\nabla^m a \odot \omega_j) \bullet_k b
    = \sum^{k-j}_{m=0} \sum^{k-j}_{n=m}\nabla^m a \odot \omega_j \odot \nabla^{n-m} b
    = \sum^{k-j}_{n=0} \sum^{n}_{m=0} \nabla^m a \odot \omega_j \odot \nabla^{n-m} b\\
    &= \sum^{k-j}_{n=0} \sum^{n}_{m=0} \nabla^{n-m} a \odot \omega_j \odot \nabla^{m} b
    = \sum^{k-j}_{m=0} \sum^{k-j}_{n=m} \nabla^{n-m} a \odot \omega_j \odot \nabla^{m} b
    = \sum^{k-j}_{m=0} a \bullet_k(\omega_j \odot \nabla^{m} b)
    = a\bullet_k (\omega_j \bullet_k b).
\end{align*}

(4) To see that $\pi_k\colon \CJ^k_A \rightarrow \CJ^{k-1}_A$ is a bimodule map, it is helpful to write the bimodule action $\bullet_k$ in a recursive matter, i.e. in terms of $\bullet_{k-1}$. First we define the $A$-bimodule action $\bullet_0$ on $\CJ^0_A = A$ as simple multiplication in $A$. Then for a generic $k$, the left action can be recursively written as
\begin{align}\label{bulletiterative}
  &a \bullet_k \omega_j = j^{k-j}(a) \odot \omega_j
  = j^{k-j-1}(a) \odot \omega_j + \nabla^{k-j}(a) \odot \omega_j
  = a \bullet_{k-1} \omega_j + \nabla^{k-j} a \odot  \omega_j.
\end{align}
where $\omega_j \in \Omega^j_S$ for $j=0,\dots,k-1$. Note that the added term $\nabla^{k-j} a \odot  \omega_j$ is in $\Omega^k_S$ and therefore in the kernel of $\pi_k$. We then have $\pi_k(a \bullet_k \omega_j) = a \bullet_{k-1} \omega_j = a \bullet_{k-1} \pi_k(\omega_j)$ since $j<k$. For $j=k$ we simply have $a\bullet_k \omega_k = a\omega_k$ and $\pi_k(a\omega_k) = 0$. The property that $\pi_k$ is a right-module map can be shown in the same manner.
It follows by iteration that $\pi_1\circ \pi_2\circ \cdots\circ \pi_k$ splits $j^k$. These considerations are analogous to the recursive analysis presented for the cases $k=1,2,3$ in Section \ref{sec:JA123}. It is also worth noting that by the very definition of $\CJ_A^k$, we have an exact sequence
\[0\to \Omega^k_S\to \CJ^k_A{\buildrel \pi_k\over \longrightarrow} \CJ^{k-1}_A\to 0\]
of underlying bimodules, which seems to be of interest in the subsequent endofunctor generalisation in \cite{FMW}.\end{proof}

The infinite jet bimodule $\CJ^\infty_A$ can then be defined as the limit of the sequence
\begin{align}\label{Jinfty}
\cdots  \to \CJ^k_A{\buildrel \pi_k\over\to} \CJ^{k-1}_A\to\cdots\to \CJ^2_A{\buildrel \pi_2\over\to}\CJ^1_A{\buildrel \pi_1\over \to} \CJ^0_A=A
\end{align}
of bimodules with bimodule maps compatible with $j^k\colon A\to \CJ^k_A$, with projections $\pi_{\infty,k} \colon \CJ^\infty_A \to \CJ^k_A$ for all $k$. This sequence can be thought of formally as a pro-object in the category of $A$-bimodules, in the classical terminology used on \cite{delgado2018lagrangian}. Then the collection $(j_0, j_1, j_2, j_3, . . .)$ of jet evaluations represent a bimodule map from $A$ to this pro-object. Note also that $\CJ^\infty_A$ as a vector space is much bigger than the `quantum symmetric forms' $\Omega_S= \bigoplus^\infty_{k=0} \Omega^k_S$ in Corollary~\ref{cor:AlgebraQuantumSymmetricForms} (since direct sums have elements with support in only a finite number of different degrees). Moreover, the latter has the trivial $A$-bimodule structure, inherited from that of $\Omega^1$ rather than the bullet bimodule structure of  $\CJ^\infty_A$ obtained from the $\bullet_k$ presented above.

\subsection{Reduced jet bundle}\label{secSsigma}

In some cases, we can give an alternative description of $\CJ^k_A$ as built on the degree $\le k$ part of a different graded algebra, namely the `braided-symmetric algebra' $S_\sigma(\Omega^1)$ associated to $\sigma$ obeying the braid relations as is the case in Theorem~\ref{thm:JkA}. This algebra is defined as
\[ S_\sigma(\Omega^1)={T_A\Omega^1\over \oplus_k\ker[k,\sigma]!}={T_A\Omega^1\over \oplus_k\ker \sum_{\pi\in S_k} \sigma_{i_1}\cdots\sigma_{i_{l(\pi)}}}\]
where, for the second expression, we write $\pi=s_{i_1}\cdots s_{i_{l(\pi)}}$ of length $l(\pi)$ as a product of simple reflections $s_i=(i,i+1)$ and $\sigma_i$ denotes $\sigma$ in the $i,i+1$ tensor power. These ideas are familiar in the vector space case where the first point of view is that of a braided-linear space \cite{Ma:book,Ma:hod} and the second has been called a `Nichols-Woronowicz algebra'. The difference is that now we work over $A$ rather than over a field, in fact more in the spirit of a symmetric version of the Woronowicz exterior algebra  on a Hopf algebra \cite{Wor}.

In our case of interest, when $\nabla$ is torsion free, $\wedge$-compatible, extendable and its $\sigma$ obeys the braid relations, we have
\[ {\Omega^1{}^{\tens k}\over \ker [k,\sigma]!}\cong \im  [k,\sigma]! \subseteq \Omega^k_S\]
by the inclusion  in Lemma~\ref{lem:binomialssymmetricforms} and combining in different degrees to an algebra inclusion
\[ S_\sigma(\Omega^1)\subseteq \Omega_S\]
for the product in Corollary~\ref{cor:AlgebraQuantumSymmetricForms}. This follows from the braided-binomial theorem cf.\cite[Thm~10.4.12]{Ma:book} which, in our conventions, and now over $A$, reads $[i,\sigma]!=\left[{i\atop j},\sigma\right]([i-j,\sigma]!\tens[j,\sigma]!)$. We let $S_\sigma^i(\Omega^1)$ be the degree $i$ component of $S_\sigma(\Omega^1)$.

\begin{corollary} Suppose that $\nabla$ obeys the conditions in Theorem~\ref{thm:JkA} and that moreover
\[ \im (\nabla^i)\subseteq \im[i,\sigma]!\]
for all $i\le k$.  Let $d^i\colon A\to S^i_\sigma(\Omega^1)$ be such that $\nabla^i(a)=[i,\sigma]!(d^i(a))$. Then there is a {\em reduced jet bimodule}
\[ \CJ^k_\sigma=\bigoplus_{0\le i\le k}S^i_\sigma(\Omega^1), \quad a\bullet_k \omega= j^{k-j}_\sigma(a) \omega,\quad \omega\bullet_k a=\omega j^{k-j}_\sigma(a),\quad j^k_\sigma(a)=\sum_{0\le i\le k}d^i(a)\]
for a bimodule map $j_\sigma^k\colon A\to \CJ^k_\sigma$ and all $a\in A$, $\omega\in S_\sigma^j(\Omega^1)$. There is also a bimodule map $\CJ^k_\sigma\to \CJ^{k-1}_\sigma$ where we quotient out the top degree and which connects $j^k_\sigma$ to $j^{k-1}_\sigma$.
\end{corollary}
\begin{proof} This is immediate once we note that the product in $S_\sigma(\Omega^1)$ corresponds to a restriction of the product $\odot$ of $\Omega_S$. \end{proof}

In nice cases, we will have an equality and the reduced bundle is the same but just expressed differently, usually more algebraically by generators and relations. We will also see an example of a strict reduction in Section~\ref{sec:M2}.


\section{Jets $\CJ^k_E$ for a vector bundle}\label{secJE}

We now consider a general `vector bundle' over the base manifold. In our setting, we work with a bimodule $E$ of `sections'  equipped with a flat connection $\nabla_{E}\colon E\to \Omega^1 \tens_A E$. Here a `sheaf' $E$ in \cite[Chap.~4]{BegMa} is loosely modelled by such data. The flat connection comes with a bimodule map $\sigma_E\colon E\tens_A\Omega^1 \to\Omega^1 \tens_A E$ for the right side Leibniz rule. Similarly to the case of a noncommutative algebra, we want to construct the jet bimodule $\CJ^k_E$, which we set to be
\begin{align*}
    \CJ^k_E = E \oplus \Omega^1 \tens_A E \oplus \dots \oplus \Omega^k_S \tens_A E=\bigoplus_{i\le k}\Omega_S^k\tens_A E.
\end{align*}
One can think of this as $\CJ^k_A\tens_A E$ but in the tensor product we use the right action of $A$ inherited from $\Omega^1$ and not the $\bullet_k$ right action. We will make $\CJ^k_E$ into a $A$-bimodule with jet actions $\bullet_k$ contructed from those of $\CJ^k_A$ and also construct bimodule maps
\begin{align*}
    j^k_E\colon E \rightarrow \CJ^k_E.
\end{align*}
The case $E=A$ and $\nabla_E=\extd$ viewed as $A\to \Omega^1\tens_A A\simeq \Omega^1$ recovers our previous $\CJ^k_A$. We start by explicitly considering the cases $k=1,2$ to see the conditions on $\nabla_E$ that need to be imposed.

The $k=1$ case just needs a generic connection $\nabla_E$, while the $k=2$ case needs $\nabla$ torsion free and $\nabla_E$ flat, extendable and Leibniz compatible as in Fig.~\ref{fig:J2E}, analogous to the ones we have already seen for $\nabla$ in Fig.~\ref{fig1}. As before, Leibniz-compatibility can be viewed as $\doublenabla(\sigma_E)=0$ which now implies the coloured braid relations Fig.~\ref{fig1}(d) since $\doublenabla(\sigma_E)$ is a bimodule map if and only of these hold. This follows from \cite[p.302]{BegMa} given the  braidings $(\sigma_E\tens\id)(\id\tens\sigma)$ on $E\tens_A\Omega^1$ and $(\sigma\tens\id)(\id\tens\sigma_E)$ on $\Omega^1\tens_AE$. Next, we automatically have a tensor product connection which we will denote $\nabla_{(n,E)}$  on $\Omega^{1\tens_A n} \tens_A E$,  as $\nabla$ is a bimodule connection and we can repeatedly tensor by it. The general case in Section~\ref{sec:JEgeneral} uses the induced $n$th-order derivatives and their higher order Leibniz rules to construct $\CJ^k_E$, building on our results for  $\CJ^k_A$ and using the braiding $\sigma_E$ to keep $E$ to the far right.

\subsection{First and second order jet bundles for the vector bundle case}
Let $E$ be an $A$-bimodule. For the split jet bundle we will also need a  bimodule connection $\nabla_{E},\sigma_E$ as discussed, but for the bimodule structure in the case of the first-order jet bimodule, we need only $\sigma_E$.

\begin{proposition}
\label{prop:J1E}
  Let $\sigma_E:E\tens_A\Omega^1\to \Omega^1\tens_AE$ be a bimodule map, and $a\in A, s\in E$, $\omega \in \Omega^1$.

  (1) The actions
  \begin{align*}
    &a\bullet_1 s = as + \extd a\tens s,&
    &s \bullet_1 a = sa + \sigma_E(s \tens \extd a),\\
    &a\bullet_1 (\omega \tens s) = a \omega \tens s,&
    &(\omega \tens s )\bullet_1 a = \omega \tens s a
  \end{align*}
  make $\CJ^1_E=E\oplus \Omega^1\tens_A E$ a bimodule and
  \[ 0\to \Omega^1\tens_A E\to \CJ^1_E\to E\to 0 \]
  (where we quotient out the $\Omega^1\tens_A E$ part) an exact sequence of bimodules.

  (2) Bimodule map splittings $j^1_E$ of this exact sequence are in 1-1 correspondence with bimodule connections $\nabla_E$ with generalised braiding $\sigma_E$ and take the form
  \begin{align*}
      &j^1_E\colon E \rightarrow \Omega^1 \tens_A E,&
      &j^1_E(s) = s + \nabla_E s.
  \end{align*}
\end{proposition}

\begin{proof}
  (1) We first check the bimodule structure. For $a,b\in A$, we have
  \begin{align*}
    a\bullet_1(b \bullet_1 s)
    &=a\bullet_1(bs + \extd b\tens s)
    = abs + \extd a \tens bs + a \extd b\tens s
    = (ab)s + \extd(ab) \tens s
    = (ab) \bullet_1 s,\\
    (s\bullet_1 a) \bullet_1 b
    &= (sa + \sigma_E(s \tens \extd a)) \bullet_1 b
    = sab + \sigma_E(sa \tens \extd b) + \sigma_E(s \tens \extd a) b
    = s(ab) + \sigma_E(s \tens \extd (ab)) = s \bullet_1 (ab),\\
    (a \bullet_1s) \bullet_1 b
    &= (as + \extd a\tens s) \bullet_1 b
    = asb + \sigma_E(as \tens \extd b) + \extd a\tens sb
    = a\bullet_1(sb + \sigma_E(s\tens \extd b))
    = a \bullet_1(s \bullet_1 b).
  \end{align*}
  These properties for the action on $\omega\tens s$ are immediate as the product $\bullet_1$ is simply given by left or right multiplication by $a\in A$. It is immediate from the form of the bimodule structure that the map $\pi_1\colon\CJ^1_E\to E$ setting $\Omega^1\tens_AE$ to zero part a bimodule map and gives the  exact sequence stated.

 (2)  If we are give $\nabla_E$ then we define $j^1_E$ as stated with $j^1_E(s) \in E \oplus \Omega^1 \tens_A E = \CJ^1_E$ as $\nabla_E s \in \Omega^1 \tens_A E$. We check that $j^1_E$ is a bimodule map
  \begin{align*}
    j^1_E(as) &= as + \nabla_E(as)
    = as + \extd a \tens s + a\nabla_E s
    = a \bullet_1 (s + \nabla_E s)
    = a \bullet_1 j^1_E(s),\\
    j^1_E(sa) &= sa + \nabla_E(sa)
    = sa + \nabla_E (s) a + \sigma_E(s\tens \extd a)
    = (s +\nabla_E s)\bullet_1a = j^1_E(a) \bullet_1 a.
  \end{align*}
 Conversely, given a splitting $j^1_E\colon E\to \CJ^1_E$ means $\pi_1\circ j^1_E=\id$, so $j^1_E(s)=s+\nabla_E s$ for some map $\nabla_E\colon E\to \Omega^1_E$. Reversing the proof that $j^1_E$ is a bimodule map immediately gives that $\nabla_E$ is a bimodule connection.
\end{proof}

It should be noted that at $k=1$, this construction is independent of the choice of connection $(\nabla_E,\sigma_E)$. Similarly to before, we say that $\CJ^1_E,\tilde \CJ^1_E$ with respective jet prologantion maps $j^1_E\colon E \to \CJ^1_E$, $\tilde j^1_E\colon E \to \tilde \CJ^1_E$ and bullet actions  $\bullet_1$, $\tilde \bullet_1$, are equivalent as jet bimodules if there is a isomorphism $\varphi \in {}_A\Hom_A(\CJ^1_E,\tilde{\CJ}^1_E)$ such that $\varphi \circ j^1_E = \tilde j^1_E$. We have:

\begin{lemma}\label{lemJ1Eisom}
Consider $\CJ^1_E, \tilde \CJ^1_E$ as constructed from the connections $(\nabla_E, \sigma_E)$, $(\tilde \nabla_E, \tilde \sigma_E)$. Then $\CJ^1_E \simeq \tilde \CJ^1_E$ as jet bimodules.
\end{lemma}

\begin{proof}
Consider the map $\varphi\colon \CJ^1_E \to \tilde \CJ^1_E$ defined on $s+\omega_1 \tens s_1 \in \CJ^1_E$ as
\[
\varphi(s+\omega_1 \tens s_1) = s + (\tilde \nabla_E - \nabla_E)s + \omega_1 \tens s_1.
\]
It automatically fulfils the above commuting diagram as $\varphi \circ j^1_E = \tilde j^1_E$. To check that it is indeed a bimodule map we compute
\begin{align*}
\varphi(a \bullet_1 (s+\omega_1 \tens s_1) ) 
&= \varphi(a s + \extd a \tens s + a \omega_1 \tens s_1)
= as + (\tilde \nabla_E - \nabla_E) (as) + \extd a \tens s + a \omega_1 \tens s_1\\
&= as + a(\tilde \nabla_E - \nabla_E) s + \extd a \tens s + a \omega_1 \tens s_1 = a\tilde \bullet_1\varphi(s+\omega_1 \tens s_1). 
\end{align*}
and 
\begin{align*}
\varphi(a \bullet_1 (s+\omega_1 \tens s_1) ) 
&= \varphi(s a+ \sigma_E (s \tens \extd a) + \omega_1 \tens s_1a)
= sa + (\tilde \nabla_E - \nabla_E) (sa) +  \sigma_E (s \tens \extd a)  + a \omega_1 \tens s_1\\
&= as + (\tilde \nabla_E - \nabla_E) s )a+ \tilde \sigma_E (s \tens \extd a)  + a \omega_1 \tens s_1 
= a\tilde \bullet_1\varphi(s+\omega_1 \tens s_1). 
\end{align*}
due to the right Leibniz rule of the connections $\nabla_E, \tilde \nabla_E$.
\end{proof}

For the 2nd-order jet bimodule we also assume a torsion free connection $\nabla\colon \Omega^1 \to \Omega^1 \tens_A \Omega^1$ on $\Omega^1$ with a generalised braiding $\sigma\colon \Omega^1 \tens_A \Omega^1 \to\Omega^1 \tens_A \Omega^1$. The induced bimodule connection on $\Omega^1\tens_AE$ will be denoted
\begin{align*}
    &\nabla_{(1,E)} \colon \Omega^1\tens_AE \rightarrow \Omega^{1\tens_A 2}\tens_AE,&
    &\nabla_{(1,E)} = \nabla \tens \id_E + \sigma_1 (\id \tens \nabla_E),
\end{align*}
where $\id$ denotes the identity on $\Omega^1$ and $\id_E$ the identity on $E$.  The associated generalised braiding is  $\sigma_{\nabla_{(1,E)}} = \sigma_1 \sigma_{E,2}$. We then define
\begin{align*}
    &\CJ^2_E = E \oplus \Omega^1 \tens_A E \oplus \Omega^2_S \tens_A E,&
    &j^2_E(s) = s + \nabla_E s + \nabla_{(1,E)}\nabla_E s.
\end{align*}

As before, we need to impose some conditions on this datum, shown in Fig.~\ref{fig:J2E} and analogous to the conditions imposed on $\nabla$ when dealing with $\CJ^3_A$, see Fig.~\ref{fig1}. Moreover, note that the identities in Fig.~\ref{fig:J2E} are trivial in the case $\nabla_E = \extd$, $\sigma_E  = \id\colon A\tens_A \Omega^1 \to \Omega^1 \tens_A A$, so we did not already need them as additional data for the case of $\CJ^2_A$.

\begin{figure}
\[ \includegraphics[width=1\textwidth]{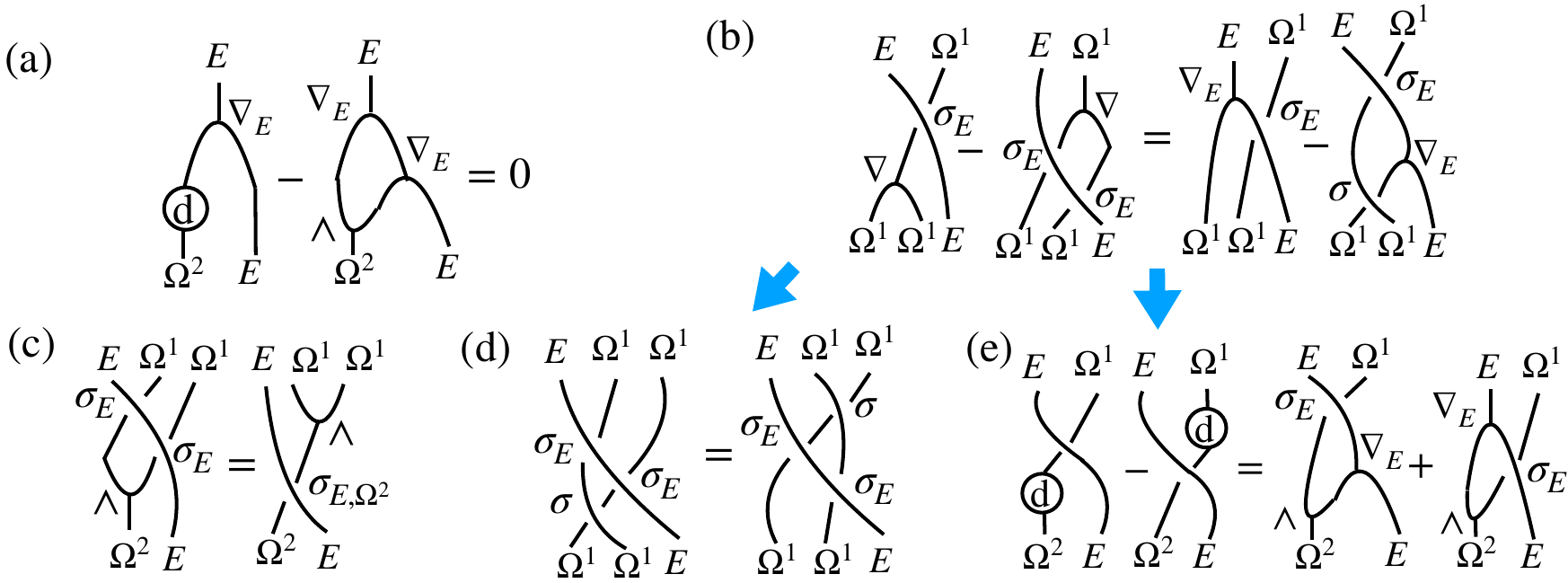}\]
\caption{\label{fig:J2E} Conditions for $\CJ^2_E$. Here, (a) says that $\nabla_E$ is flat, (b) is the Leibniz compatibility condition, (c) the extendability property of a left connection as in \cite[Def. 4.10]{BegMa} and (d) the coloured braid relations implied by (b). Given (c), condition (e) characterises that $R_E$ is a bimodule map and is also implied by (b) if $\nabla$ is torsion free.}
\end{figure}

\begin{proposition}
\label{prop:J2}
 Let $ \nabla$ be torsion free with $\sigma$ obeying the braid relations and $\nabla_E$ flat, extendable and Leibniz compatible (c.f. Fig.~\ref{fig:J2E}(a)-(c)). Then the $k=2$  bullet actions defined by
\begin{align*}
  &a \bullet_2 s = a \bullet_1 s +  (\nabla \extd a) \tens s,&
  &s \bullet_2 a = s \bullet_1 a + \sigma_{E,2}\sigma_{E,1} (s\tens \nabla\extd a),\\
  &a \bullet_2 (\omega \tens s) = a \bullet_1 (\omega \tens s) + [2,\sigma]_1 (\extd a \tens \omega \tens s),&
  &(\omega \tens s) \bullet_2 a = (\omega \tens s) \bullet_1 a + [2,\sigma]_1\sigma_{E,2} (\omega \tens s \tens \extd a),\\
  &a \bullet_2 (\omega^1 \tens \omega^2 \tens s) = a\omega^1 \tens \omega^2 \tens s,&
  &(\omega^1 \tens \omega^2 \tens s) \bullet_2 a = \omega^1 \tens \omega^2 \tens sa
\end{align*}
for $a\in A, s\in E, \omega  \in \Omega^1, \omega^1 \tens \omega^2 \in \Omega^2_S$, make $\CJ_E^2$ a bimodule and $j^2_E$ a bimodule map.
\end{proposition}
\begin{proof}
(1) We need to check that the product lands in $\CJ^2_E$, namely that the terms added to $\bullet_1$ are in $\Omega^2_S \tens_A E = \ker \wedge \tens_A E$. Since we assume $ \nabla$ to be torsion free, we have $\wedge  \nabla = \extd$ and $\wedge (\id +  \sigma) = 0$. In the case of $a\bullet_2 s$ and $s\bullet_2 a$, this gives $\wedge_1 ( \nabla \extd a \tens s) = \extd^2 a \tens s = 0$ and
\begin{align*}
  \wedge_1 \sigma_{E,2}\sigma_{E,1} (s \tens \nabla\extd a)
  = \sigma_{E,\Omega^2}(s \tens \wedge \nabla\extd a) = 0,
\end{align*}
where we used that $\sigma_E$ is extendable (condition in Fig \ref{fig:J2E}(c)). The added terms in $a\bullet_2 (\omega \tens s)$ and $(\omega \tens s)\bullet_2 a$ are quantum symmetric due to $\wedge (\id +  \sigma) = 0$ and $a \bullet_2 (\omega^1 \tens \omega^2 \tens s)$, $(\omega^1 \tens \omega^2 \tens s)\bullet_2 a$ are automatically quantum symmetric as $\omega^1\tens \omega^2 \in \Omega^2_S$.

(2) We check that $\bullet_2$ makes $\CJ^2_E$ into a bimodule, starting with the action $\bullet_2$ on $s\in E$. Let $a,b\in A$ and recall that we have the 2nd-order Leibniz rule for $ \nabla \extd$ in equation \eqref{eq:2ndOrderLeibniz}. Then
\begin{align*}
  a &\bullet_2(b \bullet_2 s)
  = a \bullet_2 (b \bullet_1 s +  \nabla \extd b \tens s)
  = a \bullet_1 (b \bullet_1 s) + \nabla \extd a \tens bs + [2,\sigma]_1 (\extd a \tens \extd b \tens s) + a \nabla \extd b \tens s\\
  &= (a b) \bullet_1 s + \nabla \extd(ab) \tens s = (a b) \bullet_2 s,\\
  (s&\bullet_2 b)\bullet_2 a
  = (s\bullet_1 b +  \sigma_{E,2}\sigma_{E,1} (s\tens \nabla\extd b))\bullet_2 a \\
  &= (s \bullet_1 b) \bullet_1 a
  + \sigma_{E,2}\sigma_{E,1} (sb \tens \nabla\extd a)
  + [2,\sigma]_1 \sigma_{E,2}\sigma_{E,1} (s \tens \extd b \tens \extd a)
  + \sigma_{E,2}\sigma_{E,1} (s\tens (\nabla\extd b) a)
  \\
  &= s \bullet_1 (ba)
  + \sigma_{E,2}\sigma_{E,1} (s \tens b\nabla\extd a)
  + \sigma_{E,2}\sigma_{E,1} [2,\sigma]_2 (s \tens \extd b \tens\extd a)
  + \sigma_{E,2}\sigma_{E,1} (s\tens (\nabla\extd b) a)\\
  &= s \bullet_1 (ba) + \sigma_{E,2}\sigma_{E,1} (s \tens \nabla\extd (ba)) =  s \bullet_2 (ba).\end{align*}
  These computations are analogous to the $E=A$ case, but now with the coloured braid relations Fig.~\ref{fig:J2E}(d) for the 3rd equality of the second group, which is, however implied by Leibniz-compatibility. Similarly,
\begin{align*}  (a&\bullet_2 s) \bullet_2 b\\
&
  = (a\bullet_1 s +  \nabla \extd a \tens s) \bullet_2 b
  = (a\bullet_1 s) \bullet_1 b +\sigma_{E,2}\sigma_{E,1} (as \tens  \nabla \extd b)
  + [2, \sigma]_1\sigma_{E,2} (\extd a \tens s \tens \extd b) +  \nabla \extd a \tens sb\\
  &= a\bullet_1 (s \bullet_1 b)
  + a \sigma_{E,2}\sigma_{E,1} (s \tens  \nabla \extd b)
  + [2, \sigma]_1 (\extd a \tens \sigma_E(s \tens \extd b))
  + \nabla \extd a \tens sb
  \\
  &= a\bullet_2 (s \bullet_1 b)
  + a \bullet_2 \sigma_{E,2}\sigma_{E,1} (s \tens  \nabla \extd b)
  = a\bullet_2 (s \bullet_2 b).
\end{align*}
 In the case of $\omega \tens s \in \Omega^1 \tens_A E$ we have
\begin{align*}
  a\bullet_2(b \bullet_2 (\omega \tens s))
 & = a\bullet_2(b \bullet_1 (\omega \tens s) + [2,\sigma]_1 (\extd b \tens \omega \tens s))\\
  &= a\bullet_1 (b \bullet_1(\omega \tens s)) + [2,\sigma]_1 (\extd a \tens b\omega \tens s) + [2,\sigma]_1 (a\extd b \tens \omega \tens s)\\
  &= (ab) \bullet_1 (\omega \tens s) + [2,\sigma]_1 (\extd(a b) \tens \omega \tens s)
  = (ab)\bullet_2 (\omega \tens s),\\
  ((\omega \tens s) \bullet_2 b) \bullet_2 a
 & =  ((\omega \tens s) \bullet_1 b + [2,\sigma]_1\sigma_{E,2} (\omega \tens s \tens \extd b)) \bullet_2 a\\
  &=((\omega \tens s) \bullet_1 b) \bullet_1 a + [2,\sigma]_1 \sigma_{E,2} (\omega \tens sb \tens \extd a) + [2,\sigma]_1 \sigma_{E,2} (\omega \tens s \tens \extd b a)\\
  &=(\omega \tens s)\bullet_1(ba) + [2,\sigma]_1 \sigma_{E,2} (\omega \tens s \tens \extd (ba))
  = (\omega \tens s) \bullet_2(ba),\\
  (a\bullet_2 (\omega \tens s)) \bullet_2 b
  &= (a \bullet_1 (\omega \tens s) + [2, \sigma]_1 (\extd a \tens \omega \tens s)) \bullet_2 b\\
  &= (a \bullet_1 (\omega \tens s)) \bullet_1 b + [2, \sigma]_1 \sigma_{E,2} (a\omega \tens s\tens \extd b)
  + [2, \sigma]_1 (\extd a\tens \omega \tens  sb)\\
  &= a \bullet_1 ((\omega \tens  s) \bullet_1 b) + a \bullet_2 [2, \sigma]_1 \sigma_{E,2} (\omega \tens  s\tens  \extd b)
  + [2, \sigma]_1 (\extd a\tens  \omega \tens  sb)\\
  &=a\bullet_2 ((\omega \tens  s) \bullet_1 b + [2, \sigma]_1 \sigma_{E,2} (\omega \tens  s \tens  \extd b))
  = a\bullet_2 ((\omega \tens  s) \bullet_2 b).
\end{align*}
Again for the actions on $\omega^1 \tens  \omega^2 \tens  s \in \Omega^2_S \tens_A E$ these properties are straight forward as $\bullet_2$ is simply multiplication by $a\in A$.

(3) We need to check that the map $j^2_E$ is well-defined, namely that $\nabla_{(1,E)} \nabla_E s \in \Omega^2_S \tens_A E$. Recall that we assume $\nabla_E$ to be flat, i.e. $R_E = (\extd \tens  \id - \id \wedge \nabla_E)\nabla_E = 0$, shown in Fig. \ref{fig:J2E}(i). Applying $\wedge_1$ to $\nabla_{(1,E)} \nabla_E s$ gives
\begin{align*}
  \wedge_1 \nabla_{(1,E)} \nabla_E s = \wedge_1(\nabla \tens  \id + \sigma_1 \id \tens  \nabla_E) \nabla_E s
  = (\extd \tens  \id - \id \wedge \nabla_E) \nabla_E s
  = R_E(s) = 0,
\end{align*}
where we have used that $\nabla$ is torsion free.

(4) To complete the proof, we show that $\bullet_2$ makes $j^2_E$ a bimodule map. We first compute the 2nd-order Leibniz rule for $\nabla_{(1,E)}\nabla_E$,
\begin{align*}
  \nabla_{(1,E)} \nabla_E (as)
  &= \nabla_{(1,E)}( \extd a \tens  s + a\nabla_E s)
  = \nabla \extd a \tens  s + \sigma_1 (\extd a \tens  \nabla_E s) + \extd a \tens  \nabla_E s + a\nabla_{(1,E)}\nabla_E s\\
  &=  \nabla \extd a\tens  s + [2,\sigma]_1 (\extd a \tens  \nabla_E s) + a \nabla_{(1,E)}\nabla_E s.
\end{align*}
Using this property we can show
\begin{align*}
  j^2_E(as) &= j^1_E(as) + \nabla_{(1,E)}\nabla_E(as)
  = a\bullet_1 j^1_E(s) + \nabla \extd a\tens  s + [2,\sigma]_1 (\extd a \tens  \nabla_E s) + a \nabla_{(1,E)}\nabla_E s \\
  &= a\bullet_2 j^1_E(s) +a \bullet_2 \nabla_{(1,E)}\nabla_E s = a\bullet_2 j^2_E(s),
\end{align*}
where we used $j^1_E(as) = a\bullet_1 j^1_E(s)$. For the other side, we again compute a 2nd-order Leibniz rule for $\nabla_{(1,E)}\nabla_E$, where we will need the Leibniz compatibility condition for $\nabla_E$ shown in Fig.~\ref{fig:J2E}(b). Using the coloured braid relations implied by this, one can show that this condition is compatible with the tensor product $\tens_A$ in a similar manner to the proof of Lemma~\ref{lemR}. The second order Leibniz rule then reads
\begin{align*}
  \nabla_{(1,E)}& \nabla_E(sa) = \nabla_{(1,E)} (\nabla_E s a + \sigma_E (s\tens  \extd a))\\
  &= (\nabla_{(1,E)} \nabla_E s) a + \sigma_1 \sigma_{E,2} (\nabla_E s \tens  \extd a) + (\nabla \tens  \id) \sigma_E (s\tens  \extd a) + \sigma_1 (\id \tens  \nabla_E) \sigma_E (s\tens  \extd a) \\
  &= (\nabla_{(1,E)} \nabla_E s) a + \sigma_1 \sigma_{E,2} (\nabla_E s \tens  \extd a) + \sigma_{E,2} \sigma_{E,1} (s \tens  \nabla \extd a) + \sigma_{E,2} (\nabla_E s \tens  \extd a)\\
  &= (\nabla_{(1,E)} \nabla_E s) a + [2,\sigma]_1 \sigma_{E,2} (\nabla_E s \tens  \extd a) + \sigma_{E,2} \sigma_{E,1} (s \tens  \nabla \extd a)
\end{align*}
where used Leibniz compatibility for the 3rd equality. We therefore have
\begin{align*}
  j_E^2(sa) &= j_E^1(sa) + \nabla_{(1,E)}\nabla_E(sa)\\
 & = j_E^1(s) \bullet_1 a + (\nabla_{(1,E)} \nabla_E s) a + [2,\sigma]_1 \sigma_{E,2} (\nabla_E s \tens  \extd a) + \sigma_{E,2} \sigma_{E,1} (s \tens  \nabla \extd a)\\
  &= j_E^1(s) \bullet_2 a + (\nabla_{(1,E)} \nabla_E s) \bullet_2 a
  = j_E^2(s) \bullet_2 a.
\end{align*}
\end{proof}

We again have a bimodule map surjection $\pi_2$ defined by quotienting out the top degree. As for $k=1$, one can again consider in what sense these constructions depend on $(\nabla,\sigma)$ and $(\nabla_E,\sigma_E)$. This is the same for all $k\geq 2$ and hence deferred to the end of the next subsection.

\subsection{Higher order Leibniz rules for the $\CJ^k_E$ case}\label{sec:JEgeneral}

Here we collect the geometric data and some facts needed for the general case. First, we observe as alluded to  near the start of Section~\ref{sec:JAgeneral} that extendable connections have a simple rule for their tensor product curvature.

\begin{lemma}
\label{lem:TensorProductCurvature}
Let $\nabla_E$ be an extendable bimodule connection. Then $R_E$ is a bimodule map iff $\nabla_E$ is Leibniz compatible as shown in Fig.~\ref{fig:J2E}(b). In this case, if $\nabla_F$ is another connection, then $\nabla_{E\tens_A F}$ has curvature
\[ R_{E\tens_A F}=R_E\tens \id + (\sigma_{E,\Omega^2}\tens \id)\id\tens  R_F.\]
\end{lemma}
\begin{proof} The proof is similar to that of Lemma~\ref{lemR} but with the diagrams as in Fig.~\ref{fig:J2E}, so we omit the  details. One can also view this as piece of the proof in \cite[Thm~4.15]{BegMa} that the DG category ${}_A\CG_A$ of extendable connections with bimodule map curvatures has a tensor product. \end{proof}

In our case, $\nabla$ is assumed to be extendable and $E$ plays the role of $F$. Hence if $\nabla$ is extendable, flat, Leibniz compatible and $\nabla_E$ is flat then the tensor product connection $\nabla_{(1,E)}$ on $\Omega^1\tens_AE$ is flat. Such considerations were not needed for for $\CJ^2_E$, but in line to the formula $\wedge_1 \nabla_{(1,E)}\nabla_E s = R_E(s)$ in the proof there that needed $\nabla_E$ to be flat,  we will similarly need $\nabla_{(1,E)}$ flat to show that the map $j^3_E$ is indeed in the kernel of $\wedge_1$.

Similarly, iterating the lemma, we see that all the higher $\nabla_{(n,E)}\colon \Omega^{1\tens_A n} \tens_A E \rightarrow \Omega^{1\tens_A (n+1)} \tens_A E$ defined recursively through
\begin{align*}
  &\nabla_{(0,E)} = \nabla_{E},&
  &\nabla_{(n+1,E)} = \nabla \tens  \id^{n-1} \tens  \id_E + \sigma_1 (\id \tens  \nabla_{(n,E)})
\end{align*}
are flat. Here  $\nabla_{(n,E)}$ applies the appropriate connection to every factor in $\Omega^{1\tens_A n} \tens_A E$ and then uses $\sigma$ to switch the new term to the utmost left factor, i.e. one can also write
\begin{align}
    \label{eq:SplitnablanE}
    \nabla_{(n,E)} = \sum^{n-1}_{k=0} \sigma_1 \dots \sigma_k (\id^k \tens  \nabla \tens  \id^{n-1-k} \tens  \id_E) + \sigma_1 \dots \sigma_n (\id^n \tens  \nabla_E).
\end{align}
where the first term is $\nabla_n\tens\id_E$. The associated generalised braidings $\sigma_{(n,E)}\colon \Omega^{1\tens_A n} \tens_A E \tens_A \Omega^1 \rightarrow \Omega^{1\tens_A (n+1)} \tens_A E$ are given by
\begin{align*}
    &\sigma_{(0,E)} = \sigma_E,&
    &\sigma_{(n,E)} = \sigma_1 (\id \tens  \sigma_{(n-1,E)}) = \sigma_1 \dots \sigma_n \sigma_{E,n+1}
\end{align*}
where $\sigma_{E,n+1}$ means that we apply $\sigma_{E}$ to the $(n+1)$-th tensor factor. We also define
\begin{align*}
    &\sigma^n_{E}\colon  E \tens_A \Omega^{1\tens_A n}
    \rightarrow \Omega^{1\tens_A n} \tens_A E,&
    &\sigma^n_{E} = \sigma_{E,n} \dots  \sigma_{E,1}
\end{align*}
and note that for an extendable $\nabla_E$, this restricts to $E\tens_A\Omega^n_S\to \Omega^n_S\tens E$. We can view all of these together as $\sigma_{E,\Omega_S}\colon E\tens_A\Omega_S\to\Omega_S\tens_A E$ acting by $\sigma^n_E$ on each degree $n$.

\begin{lemma}
\label{lem:BraidE&Binomials}
If $\sigma$, $\sigma_E$ satisfy the coloured braid relations $\sigma_{i} \sigma_{E,i+1}\sigma_{E,i}=\sigma_{E,i+1}\sigma_{E,i}\sigma_{i+1}$ then
\begin{align*}
    \left(\left[{n \atop k},\sigma\right] \tens \id_E \right)\sigma^{n}_{E}
    = \sigma^{n}_{E}\left(\id_E \tens \left[{n \atop k},\sigma\right] \right).
\end{align*}
If $\nabla_E$  is moreover extendable then
\[ (\odot\tens\id)(\id\tens\sigma_{E,\Omega_S})(\sigma_{E,\Omega_S}\tens\id)=\sigma_{E,\Omega_S}(\id\tens\odot).\]
\end{lemma}
\begin{proof}
First note that the coloured braid relations imply $\sigma_l \sigma^{n}_E= \sigma^{n}_E\sigma_{l+1}$ for $l=1,\dots, n-1$. It is clear that the equation we want to show holds for $n=1$. Assuming that it holds for $n-1$, we then compute
\begin{align*}
    &\left(\left[{n \atop k},\sigma\right] \tens \id_E \right)\sigma^{n}_E
    =
    \left(\id \tens \left[{n-1 \atop k-1},\sigma\right] \tens \id_E \right)\sigma_1\dots \sigma_{n-k}\sigma^{n}_E
    +
    \left(\id \tens \left[{n-1 \atop k},\sigma\right] \tens \id_E \right)\sigma^{n}_E\\
    &=\left(\id \tens \left[{n-1 \atop k-1},\sigma\right] \tens \id_E \right)\sigma^{n}_E\sigma_2\dots \sigma_{n-k+1}
    +
    \id \tens \left\{\left(\left[{n-1 \atop k},\sigma\right] \tens \id_E \right)\sigma^{n-1}_E\right\}\sigma_{E,1}\\
    &=
    \id \tens \left\{\sigma^{n-1}_E\left(\id_E \tens \left[{n-1 \atop k-1},\sigma\right]\right)\right\}\sigma_{E,1}\sigma_2\dots \sigma_{n-k+1}
    +
    \id \tens \left\{\sigma^{n-1}_E\left(\id_E \tens \left[{n-1 \atop k-1},\sigma\right]\right)\right\}\sigma_{E,1}\\
    &=\sigma^{n}_E \id_E \tens \left\{\left(\id \tens \left[{n-1 \atop k-1},\sigma\right]\right)\sigma_1\dots \sigma_{n-k}
    +
    \left(\id \tens \left[{n-1 \atop k-1},\sigma\right]\right)\right\}
    = \sigma^{n}_E\left(\id_E \tens \left[{n \atop k},\sigma\right] \right),
\end{align*}
where we assumed the lower case as the induction hypothesis for the 3rd equality.
\end{proof}

The second part implies that we can extend the product $\odot$ on $\Omega_S$ to a bimodule structure on $\Omega_S\tens_AE$ which will then be used as a base on which to build the jet bimodule for $E$.

\begin{corollary}
\label{cor:OmegaSEbimodule}
Let $\nabla$ be a torsion free, $\wedge$-compatible, extendable connection on $\Omega^1$. Let $\sigma$ satisfy the braid relations and $\sigma_E$ the coloured braid relations, and let $\nabla_E$ be extendable. Then $\Omega_S \tens E$ is an $\Omega_S$-bimodule with the left and right actions $\odot_E$ given by
    \begin{align*}
        &\eta \odot_E (\omega \tens s) =  \eta \odot \omega \tens s,&
        &(\omega \tens s) \odot_E \eta = \omega\odot\sigma_{E,\Omega_S}(s\tens\eta)    \end{align*}
    for all $s\in E$, $\omega,\eta \in \Omega_S$.
\end{corollary}
\begin{proof} That the left $\odot_E$ is an action and commutes with the right $\odot_E$ is immediate from the form of the actions and $\odot$ associative. That the right $\odot_E$ is an action follows immediately from Lemma~\ref{lem:BraidE&Binomials} in the more abstract form. One can also dissect the content of this by degree. \end{proof}

As in the $E=A$ case, we are interested in analysing $n$-th order derivatives for any bimodule $E$, which will have similar properties to the ones we encountered before. They are defined as
\begin{align*}
    &\nabla^n_E\colon E \rightarrow \Omega^{1\tens_A n} \tens_A E,&
    &\nabla^n_E s = \nabla_{(n-1,E)}  \nabla^{n-1}_E s =
    \nabla^n_E s = \nabla_{(n-1,E)}\dots\nabla_{(0,E)} s,
\end{align*}
with $\nabla^0_E = \id_E$.

\begin{lemma}
\label{lem:higherderivetivessymmetricE}
Let $\nabla$ be a torsion free, flat, $\wedge$-compatible bimodule connection on $\Omega^1$ and $\nabla_E$ a flat bimodule connection on $E$. Then the image of $\nabla^n_E$ is a subset of the space of quantum symmetric $n$-forms with values in $E$,
\begin{align*}
  \im \nabla^n_E \subseteq \Omega^n_S \tens_A E.
\end{align*}
\end{lemma}

\begin{proof}
The cases $n=0,1$ are trivial and the case $n=2$ has already been shown in Proposition \ref{prop:J2}. For $n>2$ the proof works in the same manner as for Lemma \ref{lem:higherderivetivessymmetric}. When applying $\wedge_1$ on $\nabla^n_E s$, we again find that it vanishes as we assumed flatness,
\begin{align*}
    \wedge_1 \nabla^n_E s
    &= \wedge_1\nabla_{(n-1,E)}\nabla_{(n-2,E)}\nabla^{n-2}_{E} s
    =(\extd \tens \id_E - \wedge_1 (\id \tens  \nabla_{(n-2,E)}))\nabla_{(n-2,E)} s\nabla^{n-2}_{E} s\\
    &= R_{(n-2,E)}(\nabla^{n-2}_Es) = 0.
\end{align*}

Next, we note that $\nabla_{(n-1,E)}$ restricts to a connection $\Omega^{n-1}_S\tens_AE$ as is clear from  its the tensor product form  (\ref{eq:SplitnablanE}). The first term here is $\nabla_{n-1}$ which by  Lemma~\ref{lem:higherderivetivessymmetric} restricts to $\Omega^{n-1}S$ and the second term also restricts by $\wedge$-compatibility of $\nabla$. Then for $i=2,\dots,n-1$, we have $\wedge_i\nabla_E^ns=\nabla_{(n-1,E)}\nabla_E^{n-1}s=0$ where we take as induction hypothesis that $\nabla^{n-1}_E$ lands in $\Omega^{n-1}_S\tens_A E$. \end{proof}

We now turn to the generalisation of the $n$-th order Leibniz rules. In contrast to the $E=A$ case, we have both left and right ones. During the proof, we will also see how $\nabla_{(n,E)}$ can be exchanged with the braidings $\sigma, \sigma_E$ and the braided binomials $\left[{n \atop k}, \sigma \right]$, in an analogous fashion to Lemma \ref{lem:Nabla&Binomials}.

\begin{lemma}
Let $\nabla$, $\nabla_E$ be Leibniz compatible as in Fig.~\ref{fig1}(e) and Fig.~\ref{fig:J2E} (b). Then
\begin{align*}
    &\nabla^n_E(as) = \sum^{n}_{k=0} \left(\left[{n \atop k},\sigma \right]\tens  \id_E \right) \nabla^{n-k} a \tens  \nabla^k_E s,\\
    &\nabla^n_E(sa) = \sum^{n}_{k=0} \left(\left[{n \atop k},\sigma \right]\tens  \id_E\right)(\id^{n-k}\tens\sigma_E^k) (\nabla^{n-k}_E s \tens  \nabla^k a).
\end{align*}
for all $a\in A, s\in E$.
\end{lemma}
\begin{proof} Taking the tensor product from of $\nabla_{(n,E)}$  in (\ref{eq:SplitnablanE}), we first observe that
\begin{align}\label{nablanEswap}
    \nabla_{(n,E)} \sigma_i &= \sigma_{i+1} \nabla_{(n,E)},\quad \nabla_{(n,E)} \left(\left[{n \atop k}, \sigma \right] \tens  \id_E\right) = \left(\id\tens  \left[{n \atop k}, \sigma \right] \tens  \id_E\right) \nabla_{(n,E)}.
\end{align}
for $i=1,\dots,n-1$, $k=0,\dots,n$. Here the left hand side has two terms of which one is $(\nabla_n\tens\id_E)\sigma_i$ and we apply Lemma~\ref{lem:Nabla&Binomials}, and the other is $\sigma_1 \dots \sigma_{n} \sigma_i (\id^{n}\tens  \nabla_E)$ and we apply the braid relations for $\sigma$ implied by Leibniz compatibility. This then implies the second half of (\ref{nablanEswap}) as before.

Next, we introduce the notation
\begin{align*}
    &\nabla_{(i,E,j)} \colon \Omega^{1\tens_A i} \tens_A E \tens_A \Omega^{1\tens_A j} \rightarrow \Omega^{1\tens_A (i+1)} \tens_A E \tens_A \Omega^{1\tens_A j},\\
    &\nabla_{(i,E,j)}
    = \nabla_{(i,E)} \tens  \id^j + \sigma_1 \dots \sigma_i\sigma_{E,i+1} (\id^i \tens  \id_E \tens  \nabla_j)
\end{align*}
for the tensor product of $\nabla_{(i,E)}$ as a bimodule connection and $\nabla_j$ from Section~\ref{sec:JAgeneral}. We  now prove the further relation
\begin{align}\label{nablajEnswap}
    \nabla_{(j,E,n-j)} \sigma_{E,j}
    = \sigma_{E,j+1}
    \nabla_{(j-1,E,n-j+1)}
\end{align}
for $j=1,\dots,n$. Here $n=j=1$ means $ \nabla_{(1,E,0)} \sigma_{E,1} = \sigma_{E,2} \nabla_{(0,E,1)}$, which is our assumed Leibniz compatibility condition for $\nabla_E$  in Fig.~\ref{fig:J2E}(b). Now proceeding inductively,
\begin{align*}
    \nabla_{(j,E,n-j)} \sigma_{E,j}
    &= \left\{\nabla_{j-1} \tens  \id \tens  \id_E \tens  \id^{n-j}
    + \sigma_1 \dots \sigma_{j-1}\left(\id^{j-1} \tens  \nabla_{(1,E,0)} \tens  \id^{n-j} \right)\right.\\
    &\quad\left.+\sigma_1 \dots \sigma_j \sigma_{E,j+1} \left(\id^{j-1} \tens  \id \tens  \id_E \tens  \nabla_{n-j} \right)
    \right\} \sigma_{E,j}\\
    &= \sigma_{E,j+1}\left(\nabla_{j-1} \tens  \id_E \tens  \id \tens  \id^{n-j}\right)
    + \sigma_1 \dots \sigma_{j-1}\left(\id^{j-1} \tens  \nabla_{(1,E,0)}\sigma_{E,1} \tens  \id^{n-j} \right)\\
    &\quad +\sigma_1 \dots \sigma_j \sigma_{E,j+1}\sigma_{E,j} \left(\id^{j-1} \tens  \id_E \tens  \id \tens  \nabla_{n-j} \right).
\end{align*}
The first term is already in the desired form. For the second we apply Leibniz compatibility for $\nabla_E$ to get
\begin{align*}
    \sigma_1 \dots \sigma_{j-1}&\left(\id^{j-1} \tens  \sigma_{E,2} \nabla_{(1,E,0)}\tens  \id^{n-j} \right)
    =\sigma_1 \dots \sigma_{j-1}\sigma_{E,j+1}\left(\id^{j-1} \tens  \nabla_{(1,E,0)} \tens  \id^{n-j} \right)\\
    &= \sigma_{E,j+1}\sigma_1 \dots \sigma_{j-1}\left(\id^{j-1} \tens  \nabla_{(1,E,0)} \tens  \id^{n-j} \right).
\end{align*}
For the last term we can use the implied coloured braid relations to pass $\sigma_{E,j+1}$ to the utmost left position $\sigma_1 \dots \sigma_j\sigma_{E,j+1}\sigma_{E,j}= \sigma_{E,j+1}\sigma_1 \dots \sigma_{j-1} \sigma_{E,j}\sigma_{i+1}.$ Summing the three expressions, we have
\begin{align*}
    \nabla_{(j,E,n-j)} \sigma_{E,j}
    &= \sigma_{E,j+1}\left\{\nabla_{j-1} \tens  \id_E \tens  \id \tens  \id^{n-j}
    + \sigma_1 \dots \sigma_{j-1}\left(\id^{j-1} \tens  \nabla_{(0,E,1)} \tens  \id^{n-j} \right)\right.\\
    &\quad \left.+\sigma_1 \dots \sigma_{j-1} \sigma_{E,j} \sigma_{i+1}
    \left(\id^{j-1} \tens  \id_E \tens  \id \tens  \nabla_{n-j} \right)
    \right\}
   \end{align*}
which we recognise as $\sigma_{E,j+1}\nabla_{(j-1,E,n-j-1)}$ when this is similarly decomposed.

Then computation for the stated left and right $n$-th Leibniz rules now proceeds analogously to the proof of Lemma \ref{lem:leibrule}, now using equation (\ref{nablanEswap}) and keeping in mind that equation (\ref{nablajEnswap})  implies
\begin{align*}
    \nabla_{(n,E,0)} (\id^{l-1} \tens \sigma^{n-l+1}_{E})
    = (\id^{l} \tens \sigma^{n-l+1}_{E}) \nabla_{(l-1,E,n-l+1)}
\end{align*}
for the computation of the right stated identity.
\end{proof}

In our geometric setting where in addition $\nabla$ is torsion free, flat, $\wedge$-compatible and extendable and $\nabla_E$ is flat and extendable, we can write this lemma in terms of the bimodule product $\odot_E$ as
\begin{align}\label{nablaEcircEderiv}
    &\nabla^n_E(as) = \sum^{n}_{k=0} \nabla^{n-k} a \odot_E  \nabla^k_E s&
    &\nabla^n_E(sa) = \sum^{n}_{k=0} \nabla^{n-k}_E s \odot_E  \nabla^k a.
\end{align}

\subsection{Higher Jet bundles for a generic bimodule $E$}\label{secJEhigher}

With the above considerations on higher derivatives on $\Omega^{1\tens_A k} \tens_A E$, we are now ready to construct  higher jets for a generic bimodule $E$.

\begin{theorem}
\label{thm:JkE}
Let $\nabla$ be a torsion free, flat, $\wedge$-compatible, extendable and Leibniz compatible bimodule connection on $\Omega^1$ as in Fig. \ref{fig1}(a)-(e) and $\nabla_E$ a flat, extendable and Leibniz compatible bimodule connection on $E$ as in Fig.~\ref{fig:J2E} (a)-(c). Then
\begin{align*}
    \CJ^k_E
    = E \oplus \Omega^1_S \tens_A  E \oplus \cdots \oplus \Omega^k_S \tens_A E,\quad   j^k_E: E \rightarrow \CJ^k_E,\quad
    j^k_E(s) = s + \nabla^1_E s + \dots + \nabla^k_E s,
\end{align*}
 form an $A$-bimodule and bimodule maps for the bullet action $\bullet_k$ given by
\begin{align*}
    &a \bullet_k (\omega_j \tens s)  =j^{k-j} a \odot_E  (\omega_j \tens  s) = \sum^k_{i=j} \nabla^{i-j} a \odot_E (\omega_j \tens  s),\nonumber\\
    &(\omega_j \tens s)  \bullet_k a = (\omega_j \tens s) \odot_E j^{k-j} a
    = \sum^k_{i=j} (\omega_j \tens  s)  \odot_E  \nabla^{i-j} a
\end{align*}
for all $\omega_j \in \Omega^j_S$, $s\in E$ and $j=0,\dots,k$. Quotienting out $\Omega^k_S\tens_AE$ gives a bimodule map surjection   $\pi_k\colon \CJ^{k}_E \rightarrow \CJ^{k-1}_E$ such that $\pi_k\circ j^k_E=j^{k-1}_E$. \end{theorem}
\begin{proof}
This proof is analogous to the proof of Theorem \ref{thm:JkA} with $\nabla^n,\odot$ replaced by $\nabla^n_E,\odot_E$ where necessary. That the action $\bullet_k$ and jet prolongation map $j^k_E$ land in $\CJ^k_E$ is from  Lemmas~\ref{cor:OmegaSEbimodule} and \ref{lem:higherderivetivessymmetricE}.  Thus, to show that $j^k_E$ is a bimodule map, we compute for $a\in A, s \in E$,
\begin{align*}
    j^k_E(as) &= \sum^k_{i=0} \nabla^{i}_E(as)
    = \sum^k_{i=0} \sum^i_{j=0} \nabla^{i-j}a \odot_E \nabla^j_E s
    = \sum^k_{j=0} \sum^k_{i=j} \nabla^{i-j}a \odot_E \nabla^j_E s
    = \sum^k_{j=0} j^{k-j}(a) \odot_E \nabla^js\\
    &= \sum^k_{j=0} a \bullet_{k} \nabla^j_Es
    =a \bullet_{k} j^k_E(s).
\end{align*}
Similarly, for  $j^k_E(sa) = j^k_E(s) \bullet_k a$. Likewise, that the right action $\bullet_k$ is compatible with the tensor product is
\begin{align*}
    &((\omega_j \tens s) \bullet_k b) \bullet_k a
    = ((\omega_j \tens s) \odot_E j^{k-j}(b))\bullet_k a
    = \sum^k_{i=j} ((\omega_j \tens s) \odot_E \nabla^{i-j}b)\bullet_k a\\
    &= \sum^k_{i=j} ((\omega_j \tens s) \odot_E \nabla^{i-j}b) \odot_E j^{k-i}(a)
    = \sum^k_{i=j} (\omega_j \tens s) \odot_E (\nabla^{i-j}b \odot j^{k-i}(a))
    = \sum^k_{i=j} (\omega_j \tens s) \odot_E (\nabla^{i-j}b \bullet_{k-j} a)\\
    &= (\omega_j \tens s) \odot_E (j^{k-j}(b) \bullet_{k-j} a)
    = (\omega_j \tens s) \odot_E j^{k-j}(ba)
    = (\omega_j \tens s) \bullet_k (ab)
\end{align*}
where in the 4th equality we have used that $\odot_E$ is a bimodule action w.r.t. the product $\odot$. The rest of the properties follow in a similar manner. That $\pi_k\colon \CJ^k_E \rightarrow \CJ^{k-1}_E$ is a bimodule surjection which connects $j^k_E$ to $j^{k-1}_E$ is then also clear. From the very definition of $\CJ^k_E$ it is also clear that we have an exact sequence
\[ 0\to \Omega^k_S\tens_AE \to \CJ^k_E{\buildrel \pi_k\over\longrightarrow} \CJ^{k-1}_E\to 0\]
of underlying bimodules.
\end{proof}

As in the $k=1$ case, we want to investigate in what sense the jet bimodules $\CJ^k_E$ are independent of the choice of connections $(\nabla,\sigma)$ on $\Omega^1$ and $(\nabla_E,\sigma_E)$ on $E$. We say that two jet bimodules $\CJ^k_E,\tilde \CJ^k_E$ with repective jet prologantion maps $j^k_E\colon E \to \CJ^k_E$ and $\tilde j^k_E\colon E \to \tilde \CJ^k_E$ and bullet actions $\bullet_k$, $\tilde \bullet_k$are equivalent as jet bimodules if there is a isomorphism $\varphi \in {}_A\Hom_A(\CJ^k_E,\tilde{\CJ}^k_E)$ such that $\varphi \circ j^k_E = \tilde j^k_E$.

In contrast to our result for $k=1$, where $\CJ^1_E$ did not depend on the connection $\nabla_E$ and also not on its braiding $\sigma_E$, in higher degrees we find that $\CJ^k_E$ does depend, but only weakly via the braiding $\sigma_E$ alone. 

\begin{lemma}
Consider $\CJ^k_E, \tilde \CJ^k_E$ as constructed from the connection $(\nabla,\sigma)$ on $\Omega^1$ and from the connections $(\nabla_E, \sigma_E)$, $(\tilde \nabla_E, \sigma_E)$ on $E$ respectively. Then $\CJ^k_E \simeq \tilde \CJ^k_E$ as jet bimodules.
\end{lemma}

\begin{proof}
We take the map $\varphi\colon \CJ^k_E \to \tilde \CJ^k_E$ defined on $s+\sum^k_{j=1} \omega_j \tens s_j \in \CJ^k_E$ as 
\[
\varphi(s+\sum^k_{j=1} \omega_j \tens s_j ) = s + \tilde j^k_E(s) - j^k_E(s) + \sum^k_{j=1} \omega_j \tens s_j .
\]
As before it automatically fulfils the above commuting diagram as $\varphi \circ j^k_E = \tilde j^k_E$. In this case the bullet $A$-actions on $\CJ^k_E,\tilde \CJ^k_E$ are the same $\bullet_k = \tilde \bullet_k$ as they are both constructed on the same $(\nabla,\sigma)$ and $\sigma_E$. To check that $\varphi$ is indeed a bimodule map we compute
\begin{align*}
&\varphi(a \bullet_k (s+ \sum^k_{j=1} \omega_j \tens s_j)) 
= \varphi(a s + (j^k(a) - a)  \odot_E s + a \bullet_k \sum^k_{j=1}\omega_j \tens s_j)\\
&= as + \tilde j^k_E(as) - j^k_E(as) + (j^k(a) - a)  \odot_E s + a \bullet_k \sum^k_{j=1}\omega_j \tens s_j\\
&= a \bullet_k s + a \bullet_k \tilde j^k_E(s) - a\bullet_k j^k_E(as + \extd a \tens s + a \bullet_k \sum^k_{j=1}\omega_j \tens s_j
= a\tilde \bullet_1\varphi(s+\omega_1 \tens s_1). 
\end{align*}
and similarly for the right action.
\end{proof}

\section{Construction of noncommutative examples}
\label{secex}

For $\CJ^{1}_A$ and $\CJ^1_E$, we need only a differential algebra $(A,\Omega,\extd)$ (plus a connection on $E$ in the vector bundle case) and for $\CJ^2_A$ and $\CJ^2_E$ also a torsion free connection on $\Omega^1$ (and the connection in $E$ should be flat). Thus we can use for $\nabla$ any number of connections in the literature, including the QLCs  of the many known quantum Riemannian geometries as covered  in \cite{BegMa}.

For general $\CJ^k_A$ and $\CJ^k_E$, however, we need a torsion free flat connection $\nabla$ with the extra conditions of being $\wedge$-compatible, extendable and the Leibniz-compatibility condition. In the vector bundle case we also need the connection on $E$ to be flat and the extra conditions of extendable and Leibniz compatible. We also know that Leibniz-compatibility implies the relevant braid relations. In the classical case, all of the extra conditions hold automatically  as $\sigma,\sigma_E$ are the flip map. Hence, classically, we see only the geometric part of this data. This is still a significant restriction as discussed in the Introduction.

An obvious example in the noncommutative case is to assume that $A$ has a central  basis of 1-forms $\{e^i\}_{i=1}^{i=n}$ which are closed and have the usual Grassmann algebra wedge product. This applies for example to the noncommutative torus and many other differential algebras. If we set $\nabla e^i=0$ on the basis then  \begin{align*}
    &\extd a\tens e^i=\nabla(ae^i) = \nabla(e^i a) = \sigma(e^i \tens \extd a)
  \end{align*}
 for all $a\in A$, which means (pulling the coefficients of  $\extd a = \del_i a e^i$ to one side by centrality) that $\sigma={\rm flip}$ on the basis. Sum over repeated indices should be understood here. As we assumed $\extd e^i=0$ then $\nabla$ is torsion free and clearly flat when computed on the basis (which is enough as the torsion and curvature tensors are left module maps). For the other conditions of extendability, $\wedge$-compatibility and Leibniz-compatibility, all of these can likewise can be checked on the basis (moving all coefficients to the left) and then they all obviously hold since $\sigma$ is just the flip map. Hence all conditions of Theorem~\ref{thm:JkA} apply in this easy case. The prolongation map is
  \begin{equation}\label{jbasis} j^\infty(s)=s+(\del_i s)e^i+(\del_j\del_i s)e^j\tens e^i+\cdots.\end{equation}
That the 3rd term here lives in $\Omega^2_S$ follows from $\extd^2 s=0$ because, for the Grassmann algebra, this implies that $\del_j\del_i s=\del_i\del_j s$ while $\Omega^2_S$ has a basis of symmetric tensor products $e_i\tens e_j+e_j\tens e_i$. The bullet bimodule structure for $k=2$, for example, is
\[ a\bullet_2 s=j^2(a)s,\quad a\bullet_2 \omega=a\omega+\extd a\tens\omega+\omega\tens \extd a,\,\quad a\bullet_2(\omega^1\tens\omega^2)=a\omega^1\tens\omega^2\]
for $\omega\in \Omega^1$ and $\omega^1\tens\omega^2\in \Omega^2_S$ and similarly on the other side. Moreover,  the algebra $\Omega_S$ can be identified with $S_\sigma(\Omega^1)$ generated over $A$ by the polynomials $k[e^1,\cdots,e^n]$. $\CJ^\infty_A$ is similar but much bigger in allowing powerseries, and for each $k$ has
bimodule actions transferred from $\bullet_k$.

For more nontrivial examples, we look at the opposite extreme from classical, where $\Omega$ is inner.  This says that there is a $\theta\in \Omega^1$ such that $\extd = [\theta,\,\}$, where $[\,,\,\}$ is the graded commutator. This happens frequently in noncommutative geometry, in spite of no classical analogue. The main theorem we need in this case is a result in \cite{Ma:gra,BegMa} that any bimodule connection on $\Omega^1$ can be constructed as
\[ \nabla= \theta\tens(\ )-\sigma((\ )\tens\theta)+\alpha,\quad \sigma:\Omega^1\tens_A\Omega^1\to \Omega^1\tens_A\Omega^1,\quad \alpha:\Omega^1\to \Omega^1\tens_A\Omega^1\]
for bimodule maps $\sigma,\alpha$ as shown. Then $\sigma$ becomes the generalised braiding. We will say that $\nabla$ is {\em inner} if $\alpha=0$.

\begin{proposition}\label{inner} For an inner $\nabla$ on an inner differential algebra $(A,\Omega,\theta)$:

\begin{enumerate}
\item $\nabla$ has zero torsion iff $\wedge\sigma(\omega\tens\theta)=-\omega\wedge\theta$ for all $\omega\in \Omega^1$.

\item $\nabla$ is $\wedge$-compatible iff the 2nd part of Fig.~\ref{fig1}(c) gives a well-defined map $\sigma_{\Omega^2}$. In this case $\nabla_{\Omega^2}=\theta\tens(\ )-\sigma_{\Omega^2}((\ )\tens\theta)$.

\item If $\nabla$ is extendable then it is flat iff $\sigma(\omega\tens\theta^2)=\theta^2\tens \omega$ for all $\omega\in \Omega^1$.

\item If $\nabla$ is extendable then $R_\nabla$ is a bimodule map.

\item $\nabla$ obeys the Leibniz-compatibility Fig.~\ref{fig1}(e) iff $\sigma$ obeys the braid relations Fig.~\ref{fig1}(g). The latter hold iff $\sigma$ they hold on $\Omega^1\tens_A\Omega^1\tens_A\theta$.
\end{enumerate}
\end{proposition}
\begin{proof} (1) This is immediate from putting the form of $\nabla$ into Fig.~\ref{fig1}(a) and $\extd=\{\theta,(\ )\}$ on 1-forms. As in \cite{Ma:gra}, it means that the the second part of (a) is equivalent to torsion freeness in this case. (2) Putting the form of $\nabla$ into Fig.~\ref{fig1}(c) and cancelling immediately gives the stated formula for $\nabla_{\Omega^2}$. Hence in this case, the 2nd part of Fig.~\ref{fig1}(c) implies and is hence equivalent to the first part. (3) Putting in the form of $\nabla$ into Fig.~\ref{fig1}(b) and $\extd=\{\theta,(\ )\}$ gives flatness iff $\theta \wedge \theta \tens \omega = \wedge_1\sigma_2\sigma_1(\omega\tens\theta\tens\theta)$ for all $\omega\in \Omega^1$, as noted in \cite{Ma:gra,BegMa}. In the extendable case, this simplifies as shown. We actually just needed Fig.~\ref{fig1}(d) applied to $\theta\tens\theta$. (4) We put in the form of $\nabla$ and $\extd=\{\theta,(\ )\}$ into Fig.~\ref{fig1}(f) and find in the extendable case that all the terms cancel. We actually just need Fig.~\ref{fig1}(d) applied to $\Omega^1\tens\theta$ and $\theta\tens\Omega^1$. (5) We put in the form of $\nabla$ into Fig.~\ref{fig1}(e) and find that all terms cancel except for two, which are Fig.~\ref{fig1}(g) applied to $\Omega^1\tens_A\Omega^1\tens_A\theta$.  Hence this is equivalent to Leibniz compatible. But if this holds then it implies that $\sigma$ obeys the braid relations on all of $\Omega^{1\tens_A 3}$. \end{proof}

\subsection{The algebra $M_2(\C)$ with its 2D calculus.}\label{sec:M2}

Here we do not study the full moduli of connections since there is an obvious connection that turns out to meet all our criteria, namely the standard QLC for the quantum metric $g=s\tens t-t\tens s$ \cite{BegMa}. One could also analyse the full moduli of allowed $\nabla$ by similar methods.

The standard 2-dimensional $\Omega^1$ for this algebra has a central basis $s,t$ while the exterior algebra is unusual as it obey $s\wedge t=t\wedge s$  (it is commutative!) along with $s^2=t^2=0$ and
\[ \extd s=2\theta \wedge s,\quad\extd t=2\theta \wedge t,\quad \theta=E_{12}s+E_{21}t,\quad \theta^2=s\wedge t,\quad\extd\theta=2s\wedge t.\]
Here $E_{ij}$ is the elementary matrix with $1$ at place $(i,j)$ and zero elsewhere and $\theta$ makes the calculus inner. Then
\[  \nabla s=2\theta\tens s,\quad \nabla t=2\theta\tens t,\quad R_\nabla=T_\nabla=0\]
is the standard connection. It corresponds to $\alpha=0$ and $\sigma=-{\rm flip}$ on the generators and is part of a larger moduli of QLCs for this metric in \cite[Ex 8.13]{BegMa}.

For our extra conditions, since the basis is central, it is enough to check everything on the generators. In this case, as $\sigma$ is just the $- \rm flip$ map it is obvious that $\sigma_{\Omega^2}$ is just the flip map, so any relations among the basic 1-forms are respected and hence the connection is $\wedge$-compatible. Likewise for extendability on the other side.  It is also obvious that $\sigma$ obeys the braid relations and hence the Leibniz-compatibility necessarily holds by Proposition~\ref{inner} but is also easy to check (eg on $s\tens t$, both side of Fig.~\ref{fig1}(e) are $-2\theta\tens t\tens s$.)  So all conditions for Theorem~\ref{thm:JkA} are met.

For the jet bimodule, we find after some computation that
\[ \CJ^\infty_A={\rm span}_{M_2(\C)}\<1,s,t, s\tens s, t\tens t, g, \cdots\>,\]
 where $g=s\tens t-t\tens s$ was the metric we began with (this is quantum symmetric in the sense $\wedge(g)=0$, so lives in $\Omega^2_S$) and the dots denote higher degrees. The map $j^\infty$ can be computed as
\[ j^\infty(a)= a + (\del_s a)s+(\del_t a)t - \{E_{21},\del_s a\} g\]
for all $a\in M_2(\C)$, from which we see that there is nothing in higher degrees. Here $\del_sa=[E_{12},a]$, $\del_ta=[E_{21},a]$ are commutators while the curly brackets denote anticommutator. The last term is $\nabla\extd a$ and can also be written as $+\{E_{12},\del_t a\}g$. Moreover, $\nabla_{2}\nabla\extd a=0$ in part because the connection is a QLC so $\nabla_{2}g=0$. The bullet bimodule actions on lower degree for $\CJ^\infty_A$ (or any $k>4$) come out explicitly as
\begin{align*}
    &a \bullet b = j^\infty(a) b,&
    &b \bullet a = bj^\infty(a) ,\\
    &a\bullet s = as - (\del_ta) g,&
    &s \bullet a = sa + (\del_ta) g,\\
    &a\bullet t = at + (\del_sa) g,&
    &t \bullet a = ta - (\del_sa) g,
\end{align*}
\begin{align*}
    &a\bullet (s\tens s) = (s\tens s) \bullet a = a(s\tens s) + (\del_s a)s\tens s\tens s+ (\del_ta)(s\tens s\tens t- g\tens s)
    -\{E_{21},\del_s a\}g,\\
    &a\bullet (t\tens t) = (t\tens t) \bullet a = a(t\tens t) + (\del_t a)t\tens t\tens t+ (\del_sa)(t\tens t\tens s- t\tens g)
    -\{E_{21},\del_s a\}g,\\
    &a\bullet g = g \bullet a = ag.
\end{align*}

For the algebra of quantum symmetric forms $\Omega_S$ itself, the lower degree products are
\begin{align*}
  &s \odot s = t\odot t = 0,&
  &s\odot t = - t\odot s = g,\\
  &s \odot (s\tens s) = (s \tens s) \odot s = s\tens s\tens s,&
  &t \odot (t\tens t) = (t \tens t) \odot t = t\tens t\tens t,\\
  &t \odot (s\tens s) = (s \tens s) \odot t = s\tens g + t\tens s\tens s = s\tens s \tens t - g \tens s,\\
  &s \odot (t\tens t) = (t\tens t) \odot s =  g\tens t + t\tens t\tens s = s\tens t \tens t - t \tens g,\\
  &s\odot g = g \odot s = t \odot g = g \odot s = 0,\\
  &(s\tens s) \odot (s\tens s) = 2 s\tens s \tens s\tens s,&
  &(t\tens t) \odot (t\tens t) = 2 t\tens t \tens t\tens t,\\
  &(s\tens s) \odot (t\tens t) = (t\tens t) \odot (s\tens s) = s\tens s \tens t\tens t + t\tens t \tens s\tens s - g\tens g,\\
  &g \odot (s\tens s) = (s\tens s)\odot g = s \tens s \tens g + g \tens s \tens s,\\
  &g \odot (t\tens t) = g \odot (t\tens t) = t\tens t \tens g + g \tens t \tens t,\\
  &g\odot g = 0.
\end{align*}

Finally, the braided-symmetric algebra $S_\sigma(\Omega^1)$ in  Section~\ref{secSsigma} is much smaller and can be identified as generated over  $M_2(\C)$ by the Grassmann algebra in $s,t$  (i.e. these now {\em anti}-commute with each other, and commute with elements of $M_2(\C)$). Here $[2,\sigma](s\tens t)=g$ so that the image of $j^k$ lies in the corresponding subalgebra of $\Omega_S$ for all $k$. Hence we have a reduced jet bimodule $\CJ^k_\sigma$ for each $k$, albeit zero for degree 3 and above, and hence
\begin{align*}
&\CJ^\infty_\sigma={\rm span}_{M_2(\C)}\<1,s, t, g\>\subsetneq \CJ^\infty_A,& 
&j^\infty_\sigma=\sum_{i=0}^2d^i,&
&\\  
&d^0(a)=a,& 
&d^1(a)=(\del_s a)s+(\del_ta)t,& 
&d^2(a)= - \{E_{21},\del_s a\}st 
\end{align*}
with the bullet bimodule structure
\begin{align*}
    &a \bullet b = j^\infty_\sigma(a) b,&
    &b \bullet a = bj^\infty_\sigma(a),\\
    &a\bullet s = as - (\del_ta) st,&
    &s \bullet a = sa + (\del_ta) st,\\
    &a\bullet t = at + (\del_sa) st,&
    &t \bullet a = ta - (\del_sa) st,\\
    &a\bullet (st) = (st) \bullet a = ast.
\end{align*}
This is much as before but on our subalgebra.

\subsection{Functions on the group $S_3$}

We take the group $S_3$ of permutations of 3 elements with generators $u=(12)$, $v=(23)$ and set $w=(13)$ so that $uvu=vuv=w$. Its algebra of functions $A=\C(S_3)$ has a standard 3-dimensional $\Omega^1$ with left-invariant basis $e_u,e_v,e_w$ and relations in degree 2 of the exterior algebra
\[  e_a^2=0,\quad e_u\wedge e_v+e_v\wedge e_w+ e_w\wedge e_u=0,\quad  e_v\wedge e_u+e_w\wedge e_v+ e_u\wedge e_w=0 ,\]
\[ \extd e_u=-e_v\wedge e_w-e_w\wedge e_v,\quad \extd e_v=-e_w\wedge e_u-e_u\wedge e_w,\quad\extd e_w=-e_u\wedge e_v-e_v\wedge e_u.\]
The relations with functions are $e_a f=R_a(f)e_a$ where $a=u,v,w$ and $R_a(f)(g)=f(ga)$ denotes the right translation operator. The calculus is inner by $\theta=e_u+e_v+e_w$. Moreover, the products of $u,v,w$ are never in the set $\{u,v,w\}$,  hence there can be no bimodule map $\alpha:\Omega^1\to \Omega^1\tens_A\Omega^1$, as  a bimodule map must respect the commutation relations with functions, which in our case means it must respect the $S_3$-grading where $|e_u|=u$, etc. Hence, all connections $\nabla$ on $\Omega^1$ are inner.

\begin{proposition}\label{propinner} $\Omega(S_3)$ admits precisely two left-invariant flat torsion free connections obeying the braid relations, given by $q=e^{\pm {2\pi\imath\over 3}}$, namely
\[ \nabla e_u={1\over q-1}\left(q e_u\tens e_u+q e_u\tens (e_v+e_w)+q(e_v+e_w)\tens e_u+e_v\tens e_w+e_w\tens e_v+q^{-1}e_v\tens e_v+q^{-1}e_w\tens e_w\right)\]
and the same for a permutation of $u,v,w$. For the braiding, we have
\begin{align*}\sigma(e_u\tens e_u)&={1\over 1-q}\left(e_u\tens e_u+q^{-1}e_v\tens e_v+q^{-1}e_w\tens e_w\right),\\
\sigma(e_v\tens e_v)&={1\over 1-q}\left(q^{-1}e_u\tens e_u+e_v\tens e_v+q^{-1}e_w\tens e_w\right),\\
\sigma(e_w\tens e_w)&={1\over 1-q}\left(q^{-1}e_u\tens e_u+q^{-1}e_v\tens e_v+e_w\tens e_w\right),\\
\sigma(e_u\tens e_v)&={1\over 1-q}\left(qe_u\tens e_v+e_v\tens e_w+e_w\tens e_u\right),& \sigma(e_v\tens e_u)&={1\over 1-q}\left(qe_v\tens e_u+e_w\tens e_v+e_u\tens e_w\right),\\
\sigma(e_v\tens e_w)&={1\over 1-q}\left(qe_v\tens e_w+e_w\tens e_u+e_u\tens e_v\right),&\sigma(e_w\tens e_v)&={1\over 1-q}\left(qe_w\tens e_v+e_u\tens e_w+e_v\tens e_u\right),\\
\sigma(e_w\tens e_u)&={1\over 1-q}\left(qe_w\tens e_u+e_u\tens e_v+e_v\tens e_w\right),&\sigma(e_u\tens e_w)&={1\over 1-q}\left(qe_u\tens e_w+e_v\tens e_u+e_w\tens e_v\right).
\end{align*}
These are extendable,  $\wedge$-compatible and Leibniz-compatible, so all the conditions of Theorem~\ref{thm:JkA} apply.
\end{proposition}
\begin{proof} We use the analysis of \cite[Ex. 3.76]{BegMa}  for which left and right translation invariant connections on $S_3$ are torsion free and flat to narrow down to a 1-parameter family (case (ii) of the curvature bimodule maps applies) with $e=c-1, d=c, b=1-2c$ and $a=(c-1)(1-2c)/c$ in terms of the parameterisation there. We then intersect this with those that obey the braid relations (case (5) applies) which needs $1+3c^2-3c=0$. This has two solutions,
\[a=c=d=qb=q^{-1}e={1\over 1-q},\quad q=e^{\pm{2\pi\imath\over3}}.\]
We then look in \cite[Ex.~4.18]{BegMa} for which are   extendable and find that these both are. By Proposition~\ref{propinner}, the Leibniz-compatibility property automatically holds.

We also translate the parameterisation  in \cite[Ex.~4.18]{BegMa} into the explicit formulae stated. Then we check  $\wedge$-compatibility. Thus, we would need from the 2nd part of Fig.~\ref{fig1}(c),
\begin{align*} \sigma_{\Omega^2}&(e_u e_v\tens e_u) ={1\over 1-q}\wedge_2(\sigma(e_u\tens (\ ) )\tens\id)(qe_v\tens e_u+e_w\tens e_v+e_u\tens e_w)\\
&={1\over(1-q)^2}(q^2 e_u\tens e_v e_u+q e_v\tens e_we_u+q e_u\tens e_w e_v+ e_v\tens e_ue_v+ e_u\tens e_ue_w+ q^{-1}e_v\tens e_v e_w).
\end{align*}
We similarly find
\begin{align*} \sigma_{\Omega^2}(e_v e_w\tens e_u) ={1\over(1-q)^2}(q^2 e_v\tens e_w e_u+q e_u\tens e_ve_u+q e_v\tens e_u e_v+ e_u\tens e_we_v+ q^2e_u\tens e_ue_w+ e_v\tens e_v e_w),\\
 \sigma_{\Omega^2}(e_w e_u\tens e_u) ={1\over(1-q)^2}(e_u\tens e_v e_u+ e_v\tens e_we_u+q^2 e_u\tens e_w e_v+ q^2e_v\tens e_ue_v+ qe_u\tens e_ue_w+ qe_v\tens e_v e_w).
\end{align*}
Adding these together, we get zero using $1+q+q^2=0$. Moreover, cyclic rotation $u\to v\to w\to u$ of these formulae  generate 6 similar expressions sufficient to similarly conclude that $\sigma_{\Omega^2}((\ )  \tens e_v)$ and $\sigma_{\Omega^2}((\ )  \tens e_w)$ are  well-defined as regards the first relation of $\Omega^2$. The proof for the other relation of $\Omega^2$ is similar. \end{proof}

The map $\sigma$ here is not involutive, i.e. it is strictly braided with eigenvalues $-1,-q^2$ (each 4 times) and $-q$. The jet bimodule is
\[ \CJ^\infty_A={\rm span}_{\C(S_3)}\<1,e_u,e_v,e_w, e_u\tens e_u, e_v\tens e_v,e_w\tens e_w,e_{uv},e_{vu},\cdots\>,\]
\[e_{uv}:=e_u\tens e_v+e_v\tens e_w+e_w\tens e_u,\quad e_{vu}:=e_v\tens e_u+e_w\tens e_v+e_u\tens e_w,\]
where the dots indicate higher order terms. The prolongation map works out as
\[ j^\infty(a)=a+\sum_{i=u,v,w}(\del_i a) e_i+ \sum_{i=u,v,w}D_i(a)e_i\tens e_i+D_{uv}(a)e_{uv}+D_{vu}(a)e_{vu}+\cdots \]
where
\[ \del_i=R_i-\id,\quad D_i={1\over q-1}(q^2(R_u+R_v+R_w)+ R_i+ 4q),\]
\[ D_{uv}=R_{uv}+{1\over q-1}(R_u+R_v+R_w-3q),\quad D_{vu}=R_{vu}+{1\over q-1}(R_u+R_v+R_w-3q)\]
for $i=u,v,w$ and $R_x(a)(y)=a(yx)$ for any function $a\in \C(S_3)$ is the right translation operator. Here $R_{uv}=R_{vw}=R_{wu}$ and $R_{vu}=R_{wv}=R_{uw}$ where products are in the group $S_3$. The $\bullet_k$ actions can be similarly given in such terms.

In terms of the $S_\sigma(\Omega^1)$ algebra, one can compute that in degree 2
\[ \ker(\id+\sigma)={\rm span}_{\C(S_3)}\<e_w\tens e_v-e_u\tens e_w, e_v\tens e_u-e_u\tens e_w, e_v\tens e_w-e_u\tens e_v, e_w\tens e_u-e_u\tens e_v\>.\]
Hence, $S_\sigma(\Omega^1)$  is generated over $A$ by the algebra $S_\sigma(\C^3)$ with at least the quadratic relations $e_we_v=e_ue_w=e_ve_u, e_ve_w=e_ue_v=e_we_u$. (These can be expected to be all the relations.) We recognise this is the braided enveloping algebra of the quandle on $\{u,v,w\}$ regarded as a braided-Lie algebra \cite{BegMa}. Moreover, this is an example where $S^k_\sigma(\Omega^1)\isom \Omega^k_S$ (as it is classically). This is because for $k=2$ the dimension of $S_\sigma(\C^3)$ is 5 and hence $S^2_\sigma(\Omega^1)$ is not only included in but isomorphic  $\Omega^2_S$ via $\id+\sigma$. Hence,  one could also build the jet bimodule essentially on $S_\sigma(\Omega^1)$ at least for each finite $k$ and using the $\bullet_k$ bimodule structure transferred to this.

\subsection{Bicrossproduct Minkowski spacetime in 1+1 dimensions with its 2D calculus}

Another example of interest \cite{BegMa:gra,BegMa} is the 1+1 bicrossproduct model quantum spacetime $[r,t]=\lambda r$ with its 2D calculus
\[  [r,t]=\lambda r,\quad [r,\extd t]=\lambda\extd r,\quad [ t,\extd t]=\lambda\extd t,\quad [r,\extd r]=[t,\extd r]=0,\]
\[ (\extd r)^2=0,\quad \extd r\wedge v+v\wedge \extd r=0,\quad v^2=\lambda v\wedge\extd r,\]
where $\extd r$ and $v=r\extd t-t\extd r$ form a central basis. Here $\extd t,\extd r$ obey the usual exterior algebra relations, resulting in the relations stated in terms of the central basis. The calculus is inner with $\theta=-{1\over\lambda r}(v+t\extd r)$ and one can check that the proposed form of bimodule connection in \cite{BegMa:gra,BegMa}
  \begin{align*}\nabla\extd r&={1\over r}\left(  \alpha v \tens v + \beta v \tens \extd r +  \gamma \extd r \tens v + \delta \extd r\tens \extd r \right),\\
\nabla v&={1\over r}\left(  \alpha' v \tens v + \beta' v \tens \extd r +  \gamma' \extd r \tens v +  \delta' \extd r\tens \extd r  \right),\\
    \sigma(\extd r \tens \extd r) &= \extd r \tens \extd r,\quad \sigma(v \tens \extd r) = \extd r \tens v,\\
    \sigma(\extd r \tens v)& = \lambda \alpha v \tens v + (1+\lambda \beta) v \tens \extd r + \lambda \gamma \extd r \tens v + \lambda \delta \extd r\tens \extd r,\\
    \sigma(v \tens v) &= (1+\lambda \alpha') v \tens v + \lambda \beta' v \tens \extd r + \lambda \gamma' \extd r \tens v + \lambda \delta' \extd r\tens \extd r
  \end{align*}
with $\C$-number coefficients is indeed inner. Hence $\nabla$ is Leibniz-compatible by Proposition~\ref{inner}. This  ansatz is motivated from geometry and sufficient to includes a QLC (which is then unique if we ask for a classical limit)\cite{BegMa:gra}, but here we look more broadly within the same ansatz.

\begin{lemma}\label{ybe} $\sigma$ of the general form of the stated ansatz has four cases where it obeys the braid relations and where the parameters have a classical limit:
\begin{align*}(i):&\quad \alpha= \gamma=0,\quad \alpha'={\beta\gamma'\over\delta},\quad \beta'={\beta\delta'\over\delta},\quad \gamma'=\delta+{\beta\delta'\over\delta};\quad \delta\ne 0,\beta,\delta'\in \C\\
(ii):&\quad \alpha= \delta=\alpha'=\delta'=0,\quad \beta=-\gamma,\quad \gamma'=-\beta';\quad \gamma,\beta'\in \C\\
(iii):&\quad \alpha=\beta=\gamma=\delta=0,\quad  \alpha'={\beta'{}^2\over \delta'},\quad \gamma'=\beta';\quad\delta'\ne 0,\beta'\in \C\\
(iv):&\quad \alpha=\gamma=\delta=\beta'=\gamma'=\delta'=0;\quad \beta,\alpha'\in \C
\end{align*}
\end{lemma}
\begin{proof} This is a matter of writing out the YBE for $\sigma$ on the basis $\extd r, v$. One of the equations
is $\alpha^2(1+\lambda(\alpha'+\beta))=0$ which requires $\alpha=0$ or $\alpha'+\beta=-1/\lambda$ and we exclude
the latter case as having no limit as $\lambda\to 0$. We then set $\alpha=0$ and analyse the simplified equations that result.  \end{proof}

Torsion-freeness and flatness/extendability (these turn out to be the same for this form of connection) have already been analysed in \cite{BegMa:bia,BegMa}, so all that remains is to intersect these with the braid relations and check for $\wedge$-compatibility.

\begin{proposition}\label{types12} For the bicrossproduct Minkowski spacetime with its  2D calculus and within the stated ansatz, bimodule connections with a classical limit which are torsion free and obey the braid relations are given by:

\noindent (i) $\alpha=\beta=\gamma=\alpha'=\beta'=0$, $\delta=\gamma'=2$, i.e.
\[\sigma(\extd r \tens v) = v\tens\extd r+2\lambda \extd r \tens \extd r,\quad \sigma(v \tens v) = v\tens v+2\lambda\extd r\tens( v +\delta'\extd r),\quad \delta'\in\C\]
(ii) $\alpha=\beta=\gamma=\delta=\alpha'=\delta'=0$,  $-\beta'=\gamma'=1$, i.e.
\[ \sigma(\extd r \tens v) = v \tens \extd r,\quad \sigma(v \tens v) = v\tens (v-\lambda\extd r)+\lambda\extd r\tens v.\]
These are all extendable,  $\wedge$-compatible and Leibniz-compatible, so all the conditions of Theorem~\ref{thm:JkA} apply.
\end{proposition}
\begin{proof} The moduli of torsion free connections of the general  form of our ansatz is given \cite{BegMa} by
\begin{equation*} \lambda\alpha +\beta=\gamma,\quad \lambda\alpha' +\beta'+2=\gamma' . \end{equation*}
The intersection of this with the braid relations in Lemma~\ref{ybe} leads to the cases (i) and (ii) stated. We will show that the extendability and $\wedge$-compatibility hold for both, so they both meet all the conditions of Theorem~\ref{thm:JkA}.

First, it was shown in \cite{BegMa:bia} that the flat/extendable ones within our ansatz are given by four equations  which, under the assumption of a classical limit, results in 3 cases
\begin{align*}(a):&  \quad \alpha=\beta=\alpha'=\beta'=0,\quad \gamma,\gamma',\delta,\delta'\in \C\\\
(b):&\quad \alpha=\beta=\gamma=\alpha'=0,\quad \gamma'=1+\delta,\quad  \beta'\ne 0, \delta,\delta'\in \C\\
(c):&\quad \alpha'=-\beta,\quad \beta'=-{\beta^2\over\alpha},\quad \gamma'=\delta-1-2{\beta\gamma\over\alpha},\quad \delta'=-{\beta(\alpha+\beta\gamma)\over\alpha^2},\quad \alpha\ne 0,\beta,\gamma,\delta\in\C\end{align*}
(where the analysis in \cite{BegMa:bia} missed case (b) here).  From this, we see that both our solutions  are extendable without further restriction.

For $\wedge$-compatibility, in all cases, from the general form of our ansatz, the flip form of $\sigma((\ )\tens\extd r)$ on the basis means that if we define $\sigma_{\Omega^2,\Omega^1}$ as required then
  \begin{align*}
    \sigma_{\Omega^2,\Omega^1}(\extd r \wedge\extd r \tens\extd r) &=\wedge_2 (\sigma(\extd r\tens(\ ))\tens\id)(\extd r\tens\extd r)= \extd r\tens\extd r\wedge\extd r=0,\\
       \sigma_{\Omega^2,\Omega^1}(v \wedge(v-\lambda\extd r) \tens\extd r) &=\wedge_2 (\sigma(v\tens(\ ))\tens\id)(\extd r\tens(v-\lambda\extd r))= \extd r\tens v\wedge(v-\lambda\extd r)=0,\\
          \sigma_{\Omega^2,\Omega^1}((v \wedge\extd r+\extd r\wedge v) \tens\extd r) &=\wedge_2 (\sigma(v \tens(\ ))\tens\id)(\extd r\tens\extd r)+ \wedge_2 (\sigma(\extd r \tens(\ ))\tens\id)(\extd r\tens v) \\
          & = \extd r\tens(v\wedge \extd r+\extd r\wedge v)=0
  \end{align*}
  so this part is automatic and we only have to check $\sigma_{\Omega^2,\Omega^1}((\ )\tens v)$.

 Next, both our solutions are such that $\sigma(\extd r\tens  X)$ has the form $X'\tens\extd r$ for any linear combination $X$ and some linear combination $X'$ of our generators. Hence in both cases
  \begin{align*}
    \sigma_{\Omega^2,\Omega^1} &(\extd r \wedge \extd r \tens v) = \wedge_2 \sigma(\extd r \tens(\ ))\tens\id)( X \tens \extd r)=  \sigma(\extd r \tens X') \wedge \extd r=0
    \end{align*}
 as $\extd r\wedge\extd r=0$. So both our solutions are compatible with this relation.

Next, for the first solution (i), we have
\begin{align*} \sigma((v-\lambda\extd r)\tens v)&=v\tens (v-\lambda\extd r)+2\lambda\extd r\tens(v+(\delta'-\lambda)\extd r),\\
\sigma_{\Omega^2,\Omega^1}(v \wedge(v-\lambda\extd r) \tens v) &=\wedge_2 (\sigma(v\tens(\ ))\tens\id)(v\tens (v-\lambda\extd r)+2\lambda\extd r\tens(v+(\delta'-\lambda)\extd r)  )\\
&=\sigma(v\tens v)\wedge(v-\lambda\extd r)+2\lambda\extd r\tens v\wedge (v+(\delta'-\lambda)\extd r)\\
&=(v\tens v+ 2\lambda\extd r\tens(v+\delta'\extd r))\wedge (v-\lambda\extd r)+2\lambda \delta' \extd r\tens v\wedge\extd r\\
&=2\lambda\delta'\extd r\tens (\extd r\wedge v+v\wedge\extd r)=0
\end{align*}
using the relations in the exterior algebra. Similarly
\begin{align*} \sigma_{\Omega^2,\Omega^1}((v \wedge\extd r+\extd r\wedge v) \tens v) &=\wedge_2 (\sigma(v \tens(\ ))\tens\id)( v\tens\extd r+2\lambda \extd r \tens \extd r     )\\
&\quad + \wedge_2 (\sigma(\extd r \tens(\ ))\tens\id)(  v\tens v+2\lambda\extd r\tens( v +\delta'\extd r)  ) \\
&=(v\tens v+2\lambda\extd r\tens(v+\delta'\extd r))\wedge\extd r+ 2\lambda\extd r\tens v\wedge\extd r\\
&\quad+(v+2\lambda\extd r)\tens\extd r\wedge v+2\lambda\extd r\tens\extd r\wedge(v+\delta'\extd r)\\
&=(v+ 4\lambda\extd r   )\tens(v\wedge\extd r+\extd r\wedge v)=0
\end{align*}
so this solution is $\wedge$-compatible. For the solution (ii), we have more simply
\begin{align*} \sigma((v-\lambda\extd r)\tens v)&= v\tens (v-2\lambda\extd r)+\lambda\extd r\tens v   , \\
\sigma_{\Omega^2,\Omega^1}(v \wedge(v-\lambda\extd r) \tens v) &=\wedge_2 (\sigma(v\tens(\ ))\tens\id)(   v\tens (v-2\lambda\extd r)+\lambda\extd r\tens v     )\\
&=\sigma(v\tens v)\wedge(v-2\lambda\extd r)+\lambda \extd r\tens v\wedge v\\
&=v\tens (v-\lambda\extd r)\wedge(v-2\lambda\extd r)+\lambda\extd r\tens v\wedge(v-2\lambda\extd r)+\lambda\extd r\tens v^2\\
&=-v\tens(v\wedge \extd r+\extd r\wedge v)+2\lambda\extd r\tens v\wedge (v-\lambda\extd r)=0
\end{align*}
using the relations in the exterior algebra. Similarly
\begin{align*} \sigma_{\Omega^2,\Omega^1}((v \wedge\extd r+\extd r\wedge v) \tens v) &=\wedge_2 (\sigma(v \tens(\ ))\tens\id)( v\tens \extd r     ) + \wedge_2 (\sigma(\extd r \tens(\ ))\tens\id)( v\tens (v-\lambda\extd r)+\lambda\extd r\tens v  ) \\
&=\sigma(v\tens v)\wedge\extd r+v\tens\extd r\wedge(v-\lambda\extd r)+\lambda\extd r\tens\extd r\wedge v\\
&=v\tens v\wedge\extd r+\lambda\extd r\tens v\wedge\extd r+v\tens \extd r\wedge v+\lambda\extd r\tens\extd r\wedge v=0
\end{align*}
so this solution is also $\wedge$-compatible. \end{proof}

The solution (ii) is involutive, $\sigma^2=\id$, hence not strictly braided, while the 1-parameter solution (i) is not involutive and hence is strictly braided when $\lambda\ne 0$. In both cases the eigenvalues are $\pm 1$. More typical solutions in Lemma~\ref{ybe} are strictly braided with eigenvalue that depend on some of the parameters. These could be of interest in other contexts.

For the jet bimodule, we focus on the 1-parameter solution (i). We have
\[ \CJ^\infty_A={\rm span}_A\<1, \extd r, v, \extd r\tens \extd r, v\tens (v-\lambda\extd r), v\tens\extd r+\extd r\tens v,\cdots\>\]
where the dots denote higher order, and
\[ j^\infty(a)=a+(\del_r a)\extd r+(\del_va)v+\nabla \extd a+\cdots,\]
\[ \nabla\extd a=\left(\del_r^2a+(\del_r a){2\over r}+ (\del_v a){\delta'\over r}\right)\extd r\tens\extd r+(\del_v^2a)v\tens (v-\lambda\extd r)+(\del_v\del_r a+\lambda\del^2_v a)(v\tens \extd r+\extd r\tens v)\]
where provided $a(r,t)$ is normal ordered with $t$ to the right,
\[ \del_ra={\del a\over\del r} +(\del_\lambda a)r^{-1} t,\quad \del_va=(\del_\lambda a)r^{-1};\quad \del_\lambda a(r,t)={a(r, t)-a(r,t-\lambda)\over\lambda}\]
and $\del/\del r$ is the usual partial derivative leaving $t$ constant. This is obtained by standard formulae~\cite{BegMa} for $\extd a$ in the $\extd r,\extd t$ basis and re-expressing $\extd t$ in terms of $\extd r, v$. One can check that
\[ \del_r\del_v a+(\del_v a){2\over r}=\del_v\del_r a+\lambda\del^2_v a\]
as needed for the formula for $j^\infty$ to collect as shown. The $\bullet_k$ actions are then determined by the latter.

In terms of $S_\sigma(\Omega^1)$, one can compute in degree 2 and for generic $\lambda,\delta'$ that
\[ {\rm ker}(\id+\sigma)={\rm span}_A\<v\tens\extd r-\extd r\tens v+\lambda\extd r\tens\extd r  \>\]
from which it follows that the braided symmetric algebra $S_\sigma(\Omega^1)$ is spanned over $A$ by the algebra $S_\sigma(\C^2)$ generated by $\extd r,v$ with at least the quadratic relations $[\extd r,v]=\lambda(\extd r)^2$. (These are expected to be all the relations). The image in degree 2 is 3-dimensional over the algebra in agreement with $\CJ_A$ in degree 2.  Thus, this is another example where, for generic parameter values, the braided symmetric algebra and the jet bimodule essentially coincide and one could build the jet bundle theory directly on the former for each degree $k$, but with $\bullet_k$.

\subsection{The quantum group $\C_q[SL_2]$}

In this example we take the quantum group $A = \C_q[SL_2]$ \cite{Ma:book} with $q \in \C^{\times}$ and generators $a,b,c,d$ obeying the relations
\begin{align*}
	&ba = qab,&
	&ca = qac,&
	&db = qbd,&
	&dc = qcd,&
	&da-ad = (q-q^{-1})bc,&
	&bc = cb,&
	&ad - q^{-1}bc = 1.
\end{align*}
We will focus on its 3D calculus with basis \cite{Wor,BegMa}
\begin{align*}
	&e^0 = d\extd a - qb\extd c,&
	&e^+ = q^{-1}a \extd c -q^{-2} c\extd a,&
	&e^- = d\extd b -qb \extd d
\end{align*}
and relations $e^\pm f = q^{|f|}f e^\pm$, $e^0 f = q^{2|f|} fe^0$, for $f \in  \C_q[SL_2]$ an element of homogeneous degree and $|a| = |c| = 1$, $|b| = |d| = -1$. Note that the above definitions imply that $\extd$ acts on the generators as
\begin{align*}
	&\extd a = qb e^+ + ae^0,&
	&\extd b = ae^- -q^{-2} be^0,&
	&\extd c = qd e^+ + ce^0,&
	&\extd d = ce^- - q^{-2}de^0.
\end{align*}

This calculus in \emph{left-invariant} in the sense that the coproduct $\Delta$ induces a ``left coaction'' $\Delta_L \colon \Omega^1 \to  \C_q[SL_2] \tens \Omega^1$ via $\Delta_L \extd = (\id \tens \extd) \Delta$, and can be prolongated by setting 
\begin{align*}
	\label{eq:relationsCqSU2}
	&\extd e^0 = q^3 e^+ \wedge e^-,&
	&\extd e^\pm = \mp q^{\pm2} [2]_{q^{-2}} e^\pm \wedge e^0, \nonumber\\
	&e^\pm \wedge e^\pm = e^0 \wedge e^0 = 0,&
	&q^2 e^+ \wedge e^- + e^- \wedge e^+ = 0,&
	&e^0\wedge e^\pm + q^{\pm4} e^{\pm} \wedge e^0 = 0,
\end{align*}
where $[n]_q = (1-q^n)/(1-q)$ denotes the $q$-integer. When discussing higher tensor powers of these generators, we will use the notation
\(
e^{i_1 \dots i_n}_n \coloneq e^{i_1} \tens \dots \tens e^{i_n} \in (\Omega^1_A)^{\tens_A n}.
\) for $i_j \in \{0,\pm\}$.
 To define the jet bundle of $\C_q[SL_2]$, we will restrict ourselves to \emph{left-invariant} bimodule connections, i.e. connections $\nabla$ satisfying $\Delta_L \nabla = (\id \tens \nabla) \Delta$, where $\Delta_L$ here is the tensor product coaction on $\Omega^1 \tens \Omega^1$. Such connections where investigated in \cite[Example 3.77]{BegMa}. We now build the jet bundle for $\C_q[SL_2]$ with $q=\pm i$. 
\begin{lemma} Left-invariant bimodule connections on $\Omega^1$ which are torsion-free, flat, $\wedge$-compatibility,  extendable and Leibniz compatible exist only for $q=\pm i$, and in this case are characterised by $\nu \in \C$ with
\begin{align*}
	&\nabla e^0 = \nu (e^{+-} - e^{-+}) -q e^{-+},&
	&\nabla e^\pm = 0,&
	&\sigma(e^r \tens e^s) = (-1)^{rs} e^s \tens e^r,
\end{align*}
where $r,s \in \{0,\pm\}$.
\end{lemma}
\begin{proof}
A generic left-invariant connection is given by as \cite[Example 3.77]{BegMa}
\begin{align*}
	&\nabla e^0 = \gamma e^0 \tens e^0 + \nu e^+ \tens e^- + \mu e^- \tens e^+,&
	&\nabla e^\pm = \alpha_\pm e^0 \tens e^\pm + \beta_\pm e^\pm \tens e^0
\end{align*}
on the generators, for parameters $ \alpha_\pm,  \beta_\pm, \nu, \mu, \gamma \in \C$. Note that we can write this connection as $\nabla e^i = - \Gamma^i_j \tens e^j$ for the matrix with basis order $e^+, e^0, e^-$
\[
\Gamma = -
\begin{pmatrix}
	\alpha_+e^0 & \beta_+ e^+ & 0 \\
	\mu e^- & \gamma e^0 & \nu e^+\\
	0 & \beta_- e^- & \alpha_- e^0
\end{pmatrix}.
\]
The corresponding map $\sigma$ is computed in \cite[Example 3.77]{BegMa} as well as the torsion
\begin{align*}
	&T_\nabla(e^0) = (\nu -q^2 \mu - q^3)e^+ \wedge e^-,&
	&T_\nabla(e^\pm) = (\beta_\pm - q^{\pm 4} \alpha_\pm \pm q^{\pm 2} (1+q^{-2}) )e^\pm \wedge e^0.
\end{align*}
The new part is to compute the curvature from $R_\nabla = -(\extd \Gamma + \Gamma \wedge \Gamma)$  as
\[ R_\nabla = 
\begin{pmatrix}
	(\alpha_+ q^3 + \beta_+ \mu) e^+ \wedge e^- & 
	\beta_+ (-q^2 [2]_{q^{-2}} -q^4 \alpha_+ + \gamma) e^+ \wedge e^0 & 
	0\\
	\mu(q^{-2} [2]_{q^{-2}} + \alpha_+ - q^{-4} \gamma) e^- \wedge e^0&
	(\gamma q^3 - q^2 \mu \beta_+ + \nu \beta_-) e^+ \wedge e^-&
	\nu(-q^2 [2]_{q^{-2}} - q^4 \gamma + \alpha_-) e^+ \wedge e^0\\
	0 &
	\beta_- (q^{-2} [2]_{q^{-2}} + \gamma - q^{-4} \alpha_-) e^- \wedge e^0&
	\alpha_- q^3 - \beta_- \nu q^2 e^+ \wedge e^-
\end{pmatrix}.
\]
We then use Mathematica to impose flatness, torsion freeness and the YBE to find the joint solution as $\alpha_\pm =  \beta_\pm = \gamma = 0$, $q= \pm i$ and $\mu = -\nu - q$. This then gives the braiding $\sigma(e^r \tens e^s) = (-1)^{rs}e^s \tens e^r$ for  $r,s \in \{0,\pm\}$, meaning that $\sigma$ is the flip map whenever $r=0$ or $s= 0$ and $-\mathrm{flip}$ otherwise. It is therefore clear the it is $\wedge$-compatible and extendable, so only Leibniz compatibility needs to be checked, namely $\sigma_2 \nabla_2 = \nabla_2 \sigma$. Due to $\nabla e^\pm = 0$ this is immediate on $e^r \tens e^s$ for $r,s \in {\pm}$. For the rest this is a matter of straight forward computation, we have for example
\begin{align*}
	&\nabla_2 \sigma(e^{0r}) = \nabla_2 (e^{r0}) = \sigma_1(\nu e^{r+-} + \mu e^{r-+}) = -\nu e^{+r-} - \mu e^{-r+}, \\
	&\sigma_2 \nabla_2 (e^{0r}) = \sigma_2 (\nu e^{+-r} + \mu e^{-+r} ) = -\nu e^{+r-} - \mu e^{-r+}\\
\end{align*}
and similarly on $e^{r0}, e^{00}$.
\end{proof}

Hence by our general results we have  a jet bundle for these values of $q$. To compute the jet bimodule note that the $\wedge$-relations reduce to $e^0 \wedge e^0 = e^\pm \wedge e^\pm = 0$, $e^+ \wedge e^- = e^- \wedge e^+$, $e^0\wedge e^\pm = - e^{\pm} \wedge e^0$ in the $q=\pm i$ case. This is very close to a Grassmann algebra, except for the commutation relations for $e^+$ and $e^-$. The infinite jet bimodule is then spanned by 
\[
\CJ^\infty_{\C_q[SL_2]} = {\rm span}_{\C_q[SL_2]} \< 1, e^0, e^\pm, e^0 \tens e^0, e^{\pm} \tens e^{\pm}, e^{+} \tens e^{-} - e^{-} \tens e^{+}, e^{0} \tens e^{\pm} + e^{\pm} \tens e^{0}, \dots \>
\]
up to degree 2. Higher degrees can be computed by starting with a basis element $e^{i_1\dots i_n}_n$ and adding terms to ensure that it is indeed in the kernel of $\wedge$ at every factor. To analyse the jet prolongation maps $j^n$ let us introduce the following notation for $n>1$
\begin{align*}
&e^{\{\pm 0 \dots 0\}}_n = e^{\pm 0\dots 0}_n + e^{0\{\pm 0 \dots 0\}}_{n}, &
&e^{\{\pm \mp 0 \dots 0\}}_n = e^{\pm \{\mp 0 \dots  0\}}_n - e^{\mp \{\pm 0 \dots  0\}}_n + e^{0 \{\pm \mp 0 \dots  0\}}_n,
\end{align*}
with $e^{\{\pm\}}_1 = e^{\pm}$ and $e^{\{\pm \mp\}}_2 = e^{\pm \mp}_2 - e^{\mp \pm}_2$. These are essentially all the permutations of the sets $\{+0\dots 0\}$ and $\{\pm \mp 0 \dots 0\}$, where in the latter we take into account that when we swap signs, the corresponding basis element gains a $-1$ factor. Note that in both cases we have $e^{\{\pm 0\dots 0\}}_n \sim [n,\sigma]! e^{\pm 0\dots 0}_n$ and  $e^{\{\pm \mp 0\dots 0\}}_n \sim [n,\sigma]! e^{\pm \mp 0\dots 0}_n$, up to some normalisation constant.

\begin{proposition}
The jet prolongation map acts on the generators of $\C_q[SL_2]$ as
\begin{align*}
	&j^n(a) = a \sum^n_{k=1} e^{0\dots 0}_n + qb \sum^n_{k=1} e_n^{\{+0 \dots 0\}} + \nu a \sum^n_{k=2} e_n^{\{+-0 \dots 0\}},\\
	&j^n(b) = b \sum^n_{k=1} e^{0\dots 0}_n + a \sum^n_{k=1} e_n^{\{-0 \dots 0\}} + \nu b \sum^n_{k=2} e_n^{\{-+0 \dots 0\}},\\
	&j^n(c) = c \sum^n_{k=1} e^{0\dots 0}_n + qd \sum^n_{k=1} e_n^{\{+0 \dots 0\}} + \nu c \sum^n_{k=2} e_n^{\{+-0 \dots 0\}},\\
	&j^n(d) = d \sum^n_{k=1} e^{0\dots 0}_n + c \sum^n_{k=1} e_n^{\{-0 \dots 0\}} + \nu d \sum^n_{k=2} e_n^{\{-+0 \dots 0\}}.
\end{align*}
\end{proposition}

\begin{proof}
We will focus on the computation of $j^n(a)$, the rest of the jet prolongations can be computed in a similar way. We will do this by induction. For $n=0$, $n=1$ this is clear due to $\extd a = ae^0 + qb e^+$. For $n=2$ we have $j^2(a) = a + \extd a + \nabla \extd a$ and
\begin{align*}
	\nabla \extd a = \nabla (ae^0 + qb e^+) = a e^{00} + qb e^{+0} + \nu a (e^{+-} - e^{-+}) - qa e^{-+} + qb e^{0+} + q a e^{-+}  = ae^{00} + qb e^{\{+0\}} + \nu a e^{\{+-\}}
\end{align*}
showing that the formula indeed holds for $n=2$.
Before doing the computation for any $n$, we need to show some identities, namely
\begin{align*}
	\nabla_n e^{0\dots0}_n 
	&= \sum^{n-1}_{k=0} \sigma_1 \dots \sigma_k e^{0\dots 0}_{k} \tens (\nabla e^0) \tens e^{0\dots 0}_{n-k-1}\\
	&= \nu e^+ \sum^{n-1}_{k=0}e^{0\dots 0}_{k} \tens e^- \tens e^{0\dots 0}_{n-k-1} + \mu e^- \sum^{n-1}_{k=0} e^{0\dots 0}_{k} e^+ e^0_{n-k-1}
	= \nu e^{+\{-0\dots0\}} + \mu e^{-\{+0\dots0\}} 
\end{align*}
which follows directly from the form of $\nabla e^0 = \nu e^{+-} + \mu e^{-+}$ and Equation \eqref{eq:Splitnablan}.
The second identity we will need is $\nabla_n e^{\{+0\dots0\}}_n = \nu e^{+\{-+0 \dots 0\}}_{n+1}$, which we prove by induction. For $n=2$ we have
\begin{align*}
	\nabla_2 (e^{+0} + e^{0+}) = \sigma_1 (\nu e^{++-} + \mu e^{+-+}) + \nu e^{+-+} + \mu e^{-++}
	= \nu (e^{+-+} - e^{++-}) + \mu (e^{-++} - e^{-++}) = \nu e^{+\{-+\}}
\end{align*}
the induction step then follows as
\begin{align*}
	\nabla_n e^{\{+0\dots0\}}_n 
	&= \nabla_n(e^{+0\dots 0}_n + e^{0\{+0\dots0\} }_n)
	= \sigma_1 e^+ \nabla_{n-1} e^{0\dots 0}_{n-1}
	+ (\nabla e^0)e^{\{+0\dots0\}}_{n-1}  + \sigma_1 (e^0 \nabla_{n-1}e^{\{+0\dots0\}}_{n-1}  )\\
	&= - \nu e^{++\{-0\dots0\}}_{n+1} - \mu e^{-+\{+0\dots0\}}_{n+1}
	+ \nu e^{+-\{+0\dots0\}} _{n+1}+ \mu e^{-+\{+0\dots0\}}_{n+1}
	+ \nu e^{+0\{-+0\dots0\}}_{n+1}\\
	&= \nu e^{+\{-+0 \dots 0\}}_{n+1}.
\end{align*}
One can show that $\nabla_n e^{\{-0\dots0\}}_n = \mu e^{-\{+-0 \dots 0\}}_{n+1} $ in a similar manner. The third and last identity we will need is $\nabla_n e^{\{+- 0 \dots 0\}}_n = 0$. This is clearly the case for $n=2$ as $\nabla e^\pm = 0$. The induction step is then 
\begin{align*}
	\nabla_n e^{\{+- 0 \dots 0\}}_n
	&= \nabla_n (e^{+\{-0\dots 0\}}_{n} - e^{-\{+0\dots0\}}_n +e^{0\{+-0\dots 0\}}_n)
	= \sigma_1 e^+ \nabla_{n-1} e^{\{-0\dots 0\}}_{n-1} - \sigma_1 e^-\nabla_{n-1} e^{\{+0\dots0\}}_{n-1} + \nabla e^0 e^{\{+-0\dots 0\}}_{n-1}\\
	&= -\mu e^{-+\{+-0\dots0\}}_{n+1} + \nu e^{+-\{-+0\dots 0\}}_{n+1} + \nu e^{+-\{+-0\dots 0\}}_{n+1} + \mu e^{-+\{+-0\dots 0\}}_{n+1}
	= 0.
\end{align*}
With these 3 identities in hand, we can compute $\nabla^n a$ by induction as
\begin{align*}
	\nabla^n a &= \nabla_{n-1} \nabla^{n-1} a 
	= \nabla_{n-1}(a e^{0\dots0}_{n-1} + qb e^{\{+0\dots 0\}}_{n-1} + \nu a e^{\{+-0\dots\}}_{n-1}) \\
	&= a e^{0\dots0}_{n} + qb e^{+0\dots0}_{n} 
	+ \nu a (e^{+\{-0\dots 0\}}_n
	- e^{-\{+0\dots 0\}}_n) 
	- q a e^{-\{+0\dots 0\}}_n
	+ qb e^{0\{+0\dots 0\}}_{n}+ qa e^{-\{+0\dots 0\}}_{n} \\
	&\,\,\,+  q \nu b e^{+\{-+0\dots0\}}_n 
	+  q \nu a e^{0\{+-0\dots0\}}_n 
	+  q \nu b e^{+\{+-0\dots0\}}_n 
	\\
	&= a e^{0\dots0}_{n} + qb e^{\{+0\dots0\}}_{n} + \nu a e^{\{+-0\dots0\}}_n .
\end{align*}
\end{proof}

In this example we again have $\im \nabla_n \subset \im [n,\sigma]!$ on the generators, which implies $\im \nabla_n \subset \im [n,\sigma]!$ on the whole algebra. To see this take $f,g \in \C_q[SL_2]$ with $\nabla^i f = [i,\sigma]!d^i f \in \im [i,\sigma]!$, $\nabla^j g = [j,\sigma]!d^jg \in \im [j,\sigma]!$, $i,j \leq n$, which in turn implies $\nabla^n (fg) \in \im [n,\sigma]!$ for their product
\[
\nabla^n (fg) = \sum^n_{k=0} \left[{n \atop k},\sigma\right] ([n-k,\sigma]! \tens [k,\sigma]! )d^{n-k} f \tens d^k g = \left[{n},\sigma\right]! \sum^n_{k=0} d^{n-k} f \tens d^k g.
\]
Thus, there is a reduced jet bundle $\CJ^\infty_\sigma = S_\sigma(\Omega^1)$ in this example as well, though it is smaller than the original jet bundle, as one can already see in degree 2, since
\begin{align*}
	&\im (\id + \sigma) = {\rm span}_{\C_q[SL_2]} \< e^0 \tens e^0, e^{+} \tens e^{-} - e^{-} \tens e^{+}, e^{0} \tens e^{\pm} + e^{\pm} \tens e^{0}\>,
\end{align*}
which is 4 dimensional, whereas $\Omega^2_S$ is 6 dimensional as shown above.

We remark that the infinite jet bimodule $\CJ^\infty_{\C_q[SL_2]}$ in this example can also be cast as a braided symmetric algebra $S_{\tilde \sigma}$ but for a different braiding $\tilde \sigma$ obeying the braid relations and defined on the basis as the identity on the diagonal and $\sigma$ otherwise. This is such that $\im [2,\tilde \sigma] = \ker \wedge=\Omega^2_S$, and similarly has $\Omega^n_S = \im [n,\tilde\sigma]!$  at least up to $n=4$ verified by hand.

\subsection{Constructions for $\CJ^3_E$}

Here we note  that our remark about inner calculi also applies for a bimodule connection $\nabla_E$ on any  bimodule $E$, i.e. we can construct all such in the form
\begin{equation}\label{Einner} \nabla_E= \theta\tens(\ )-\sigma_E((\ )\tens\theta)+\alpha_E,\quad \sigma_E\colon E\tens_A\Omega^1\to \Omega^1\tens_AE,\quad \alpha_E\colon E\to \Omega^1\tens_AE.\end{equation}
The proof is analogous to that in \cite{Ma:gra} for $\Omega^1$, namely given a bimodule connection $\nabla_E$, we define
\[ \alpha_E=\nabla_E-\theta\tens(\ )+ \sigma_E((\ )\tens\theta)\]
and deduced from the connection properties that this is a bimodule map. Conversely, given bimodule maps $\sigma,\alpha$ it is easy to see that $\nabla_E$ defined by the formula (\ref{Einner}) has the properties of a bimodule connection. We say that $\nabla_E$ is inner if $\alpha_E=0$.

\begin{proposition} For an inner differential graded algebra $(\Omega,\extd)$, if $\nabla_E$ is an inner bimodule connection then:
\begin{enumerate}
\item The Leibniz-compatibility in Fig.~\ref{fig:J2E}(b) holds iff the coloured braid relations  Fig.~\ref{fig:J2E}(d) hold. The latter hold is iff they  hold on $E\tens\Omega^1\tens\theta$.
\item If $\nabla_E$ is extendable as in Fig.~\ref{fig:J2E}(c) then it is flat iff $\sigma_E((\ )\tens\theta^2)=\theta^2\tens(\ )$.
\item  If $\nabla_E$ is extendable then the bimodule curvature map condition in Fig.~\ref{fig:J2E}(e) holds.
\end{enumerate}
\end{proposition}
\begin{proof} The proofs  follow the same steps as for the $\Omega^1$ case in Proposition~\ref{inner}, putting in the form of $\nabla_E$ into the diagrams in Fig.~\ref{fig:J2E} and cancelling terms. For the last two parts we also use $\extd=\{\theta,(\ )\}$ on 1-forms. \end{proof}

If we want an example, we can of course take $E=\Omega^1$ giving the example $\CJ^k_{\Omega^1}$ as $\nabla$ also has the properties needed for $\nabla_E$  in this case. Note that $\CJ^k_{\Omega^1}\supset \CJ^{k+1}_A$ but is not required to be zero under $\wedge_k$. We  refrain from giving more concrete examples here since $\CJ^k_E$ is not far from $\CJ^k_A$ when $E$ is a free module. Non-free examples are significantly harder to construct and will be studied elsewhere.

\section{Concluding remarks}\label{seccon}

We constructed jet bimodules $\CJ^k_A$ and $\CJ^k_E$ under the assumption of some geometric data on the differential algebra of $A$ and the $A$-bimodule $E$. We also showed that these data a available in a wide variety of $A$, including a discrete space (functions on a finite group Cayley graph), and a matrix algebra, the enveloping algebras of a solvable Lie algebra (a standard quantum spacetime) and, perhaps surprisingly, $\C_q[SL_2]$ at $q$ a fourth root of unity. This also shows in particular that the class of differential algebras equipped with a flat torsion free connection, to which the theory applies, is a rich class of noncommutative geometries. This class also includes fuzzy $\R^3$ (the enveloping algebra $U(su_2)$ as flat quantum spacetime relevant to 3D quantum gravity with its quantum group invariant 4D differential structure) and the bicrossproduct Minkowski spacetime (another enveloping algebra of the solvable type in Section~\ref{secex} but with its quantum Poincar\'e-invariant 5D differential structure). These both have an obvious quantum metric and flat QLC $\nabla$ given by zero on the basis, see \cite{BegMa}, which one can check meets our requirements for Theorems~\ref{thmJ3} and \ref{thm:JkE}. The proof is similar to the easy case at the start of Section~\ref{secex} as the $\extd x^\mu,\theta'$ basis for the bicrossproduct Minkowski model obey the usual Grassmann exterior algebra and $\nabla=0$, $\sigma=\rm flip$ on the basis. The basis now is not central, but the commutation relations are sufficient to move all coefficients in a tensor to the left so that it suffices to check our conditions on the basis. These are already plenty of examples on which it would be interesting to proceed to the Anderson double complex and Lagrangian field theory in a systematic way, which was so far lacking. This will be looked at elsewhere. 

On the mathematical side, a novel feature of our construction was the use of a braiding $\sigma$ obeying the braid relations on $A$-bimodules to construct the algebra $S_\sigma(\Omega^1)$ for the reduced jet bundle in Section~\ref{secSsigma}. Both this and $\Omega_S$ carry the further structure of a braided-Hopf algebra. Indeed, given an $A$-bimodule $\Omega^1$ equipped with invertible $\sigma$ obeying the braid relations (ignoring the rest of the structure, i.e. what follows applies to any $A$-bimodule equipped with a braiding), we have a braided category generated within the monoidal category of $A$-bimodules by $\Omega^1$. The braiding $\sigma$ is extended to tensor products of $\Omega^1$ and their direct sums in the usual way. Within this braided subcategory, we can make the braided shuffle algebra ${\rm Sh}(\Omega^1)$ (this is $T_A(\Omega^1)=\oplus_n\Omega^{1\tens_A n}$ with product given by the braided binomials as in Corollary~\ref{cor:AlgebraQuantumSymmetricForms}) into a braided-Hopf algebra with coproduct
\[ \underline\Delta(\omega^1\tens\cdots\tens\omega^n)=\sum_{k=0}^n(\omega^1\tens\cdots\tens\omega^k)\underline\otimes(\omega^{k+1}\tens\cdots\tens\omega^n)\]
where the underlines denote the braided Hopf algebra structure and the braided tensor product relevant to it. The $k=1,n$ cases should be understood as $1\underline\tens\id$ and $\id\underline\tens1$ respectively. The counit is zero except on degree 0 where it is the identity map. This is the $A$-bimodule version of the construction of the braided shuffle algebra ${\rm Sh}(V)$ over a field in \cite{MaTao:dua}. Our construction restricts to $\Omega_S$ making this a braided-Hopf algebra in the same way assuming the braided category is extended to include such subobjects. Similarly, one can take $T_A(\Omega^1)$ with its tensor product algebra structure, to make it into a braided Hopf algebra by $\underline\Delta \omega=\omega\tens_A 1+1\tens_A \omega$ and counit zero on $\omega\in \Omega^1$. Its extension to $\omega^1\tens\cdots\tens \omega^n$ is then given by the braided binomials as in the original construction \cite{Ma:book,Ma:hod} except that now we work with $A$-bimodules rather than over a field. This is a dual construction to the first one and this time we quotient by $\ker [n,\sigma]!$ as in Section~\ref{secSsigma} to the algebra $S_\sigma(\Omega^1)$ which remains a braided-Hopf algebra assuming our braided category is extended to include such quotients. This amounts to the $A$-bimodule of the construction of braided-linear spaces (or Nichols-Woronowicz algebras) in the approach of \cite{Ma:book,Ma:hod}. In the special case of $A$ a Hopf algebra,  $S_\sigma(\Omega^1)$ is a symmetric version of the Woronowicz exterior algebra \cite{Wor}. 

These braided-Hopf algebra structures on $\Omega_S$ and $S_\sigma(\Omega^1)$ (classically, and often in general,  they are isomorphic) have a classical meaning as follows. Indeed, classically on a manifold $M$, one can think of either of these as sections of the symmetric tensors bundle
\[ {\rm Sym}_{C^\infty(M)}(\Omega^1)= C^\infty_{poly}(T M)\]
where functions $\phi$ are polynomial in the fibre direction. In this case the generalised Hopf algebra structure over $C^\infty(M)$ merely represents pointwise fibre addition $(\Delta \phi)((v,x),(w,x))=\phi((v+w),x)$ for $(v,x)\in T_xM$.
The braided linear space partial derivatives appear as interior product $i_X$ for any vector field $X$, extended to symmetric tensors as a derivation, which in function terms functions amounts to differentiation in the fibre direction. By contrast, the jet bimodule action $\bullet_k$ of $a\in C^\infty(M)$ for each degree $k$ acts on $\phi$ according to derivatives along the base $M$, as does the jet prolongation map.  This gives some insight into the classical meaning of our construction in Section~\ref{sec:JA123} and Section~\ref{sec:JAgeneral}. The case of $\CJ^k_E$ has a little more structure which can be similarly developed. This classical meaning of $\Omega_S$ and $S_\sigma(\Omega^1)$  opens up the possibility of a noncommutative version of Lagrangian mechanics on $TM$, to be developed elsewhere and to which the braided-Hopf algebra structure is more relevant.

In terms of further developments of the general jet bundle, there are two key directions that should be addressed in further work. The first, which would be important for physics is how our constructions interact with $*$-structures needed to ensure unitarity. The natural setting here would be  $A$ a $*$-algebra, $\Omega,\extd$ a $*$-calculus and $\nabla$ $*$-preserving as in \cite{BegMa:rie,BegMa}. Another point of interest is that variational problems as expressed via jet bundles should also relate to geodesics. There is a parallel theory of quantum geodesics based on more general $A$-$B$-bimodule connections (with $B$ the coordinate algebra of the parameter space) and where the covariance of the generalised braiding requires (as for us) that it obeys the braid relations. In this context, compatibility with $*$ was resolved at least in the presence of a twisted trace \cite{BegMa:cur} and it was also shown how this braid condition could be naturally relaxed. 

The second direction for general development would be to generalise our theory to allow our background torsion free connection $\nabla$ to have curvature. This should be possible given that the jet bundle theory is somewhat dual to the algebra of differential operators on a manifold, as explained in  \cite{KraVer}. This is known to be a Hopf algebroid and a  quantum version of it was constructed in the setting of flat bimodule connections in \cite{Gho}, building on the left-module case treated in \cite[Chap~6]{BegMa}. In the latter work it is again useful to have a reference connection, but it does not need to be flat. We also note that there are contexts where a curvature obstruction can be absorbed by introducing nonassociativity, for example on the semiclassical theory of quantum differential structures. Moreover, flatness of the connection only plays a limited role in our constructions (to ensure that $\nabla^n$ lands in the right place in the construction of the prolongation maps) and is not needed for most of the other results in Section~\ref{sec:JAgeneral} at the tensor algebra level. For example, one could  drop the quantum symmetry restriction altogether and work with the braided shuffle algebra ${\rm Sh}(\Omega^1)$ in place of $\Omega_S$, albeit this would not then have the right classical limit. These are some  directions for further work suggested by our results. 

\section*{Declaration}

Data sharing is not applicable as no datasets were generated or analysed during the current study. The authors have no competing interests to declare that are relevant to the content of this article. 


\begin{thebibliography}{ggghhh}

\bibitem{And} I. Anderson,  Introduction to variational bicomplex, Contemp. Math. 132 (1992) 51

\bibitem{Ati} M.F. Atiyah, Complex analytic connections in fibre bundles,. Trans. Am. Math. Soc. 85 (1957) 181--207

\bibitem{BegMa:gra}E.J. Beggs and S. Majid, Gravity induced from quantum spacetime, Class. Quantum Grav. 31 (2014) 035020 (39pp)

 \bibitem{BegMa} E.J. Beggs and S. Majid, {\em Quantum Riemannian Geometry},  Grundlehren der mathematischen Wissenschaften, Vol. 355, Springer (2020) 809pp

\bibitem{BegMa:rie}
E.J.\ Beggs and  S.\ Majid, *-compatible connections in noncommutative Riemannian geometry, J.\ Geom.\ Phys.\ 61 (2011)  95--124


\bibitem{BegMa:bia} E.J. Beggs and S. Majid, Quantum Bianchi identities and characteristic classes via DG categories, J. Geom. Phys. 124 (2018) 350--370

\bibitem{BegMa:cur} E.J. Beggs and S. Majid,   Quantum geodesics and curvature, arXiv: 2201.08244 (math.QA)

\bibitem{Bok} N. Bokan, Curvature tensors of Riemannian manifold with connection without torsion,
Matematicki Vesnik 37 (1985); Lecture notes: Torsion free connections, topology, geometry and differential operators on smooth manifolds.

\bibitem{Con}
A. Connes,   Noncommutative Geometry,
Academic Press, Inc., San Diego, CA, 1994

\bibitem{DVM}
M. Dubois-Violette and  P.W.\ Michor, Connections on central bimodules in
noncommutative differential geometry, J.\ Geom.\ Phys.\ 20 (1996) 218--232

\bibitem{FMW} K.J. Flood. M. Mantegazza and H. Winther, Jet functors in noncommutative geometry, arXiv:2204.12401v1 (math.qa)

\bibitem{ForRom} M. Forger and H. R\"omer, Currents and the energy-momentum tensor in classical field theory:
A fresh look at an old problem, Annals of Physics 309 (2003)  306--389

\bibitem{Gho}A. Ghobadi, Hopf algebroids, bimodule connections and noncommutative geometry, arXiv:2001.08673

\bibitem{KraVer}J. Krasil'shchik and A. Verbovetsky,  Homological methods
in equations of mathematical physics, arXiv: 9808130

\bibitem{LakTho}D. Laksov and A. Thorup, The algebra of jets, Michigan Math J., 48 (2000) 393--416

\bibitem{Ma:book}S. Majid, {\em Foundations of quantum group theory}, CUP 1995

\bibitem{Ma:alg}S. Majid, Algebras and Hopf algebras in braided categories, in Lec. Notes Pure and Applied Maths 158 (1994) 55--105. Marcel Dekker

\bibitem{Ma:gra}S. Majid, Noncommutative Riemannian geometry of graphs, J. Geom. Phys. 69 (2013) 74--93

\bibitem{Ma:hod}S. Majid, Hodge star as braided Fourier transform, Alg. Repn. Theory, 20 (2017) 695--733

\bibitem{MaTao:dua}  S. Majid and W.-Q. Tao, Duality for generalised differentials on quantum groups, J. Algebra 439 (2015) 67--109

\bibitem{Mou}
J.\ Mourad, Linear connections in noncommutative geometry, Class.\ Quant.
Grav.\ 12 (1995)  965--974

\bibitem{MusHro}J. Musilov\'a and S. Hronek, The calculus of variations on jet bundles as a universal approach for a variational formulation of fundamental physical theories, Commun. Math. 24 (2016) 173--193

\bibitem{Saunders} D.J. Saunders, The geometry of jet bundles, Cambridge University Press 1989

\bibitem{Sta}S.T. Stafford and M. van den Bergh,  Noncommutative curves and noncommutative surfaces, Bull. A. M. S. 38 (2001) 171--216

\bibitem{Wor} S.L. Woronowoicz, Differential calculus on compact matrix pseudogroups (quantum groups), Comm. Math. Phys. 122 (1989) 125

\bibitem{delgado2018lagrangian} N. L. Delgado, Lagrangian field theories: ind/pro-approach and L-infinity algebra of local observables, Doctoral Thesis 2018 and arXiv:1805.10317

\end{thebibliography}
\end{document}